\documentclass[a4paper,10pt,leqno]{amsart}
\title[Approximating $L^2$-invariants by their classical counterparts]
{Survey on approximating $L^2$-invariants by their classical counterparts: Betti numbers, torsion invariants and homological growth}
\author{L\"uck, W.}
        \address{Mathematical Institute of the Univerisity of Bonn\\
                Endenicher Allee 60\\
                53115 Bonn, Germany}
         \email{wolfgang.lueck@him.uni-bonn.de}
          \urladdr{http://www.him.uni-bonn.de/lueck}
         \date{October, 2016}
              \keywords{$L^2$-invariants,  approximation, homological growth}
     \subjclass[2010]{59Q99,22D25,46L99, 58J52}

\usepackage{hyperref}
\usepackage{color}
\usepackage{pdfsync}
\usepackage{calc}
\usepackage{enumerate,amssymb}
\usepackage[arrow,curve,matrix,tips,2cell]{xy}
  \SelectTips{eu}{10} \UseTips
  \UseAllTwocells

\DeclareMathAlphabet\EuR{U}{eur}{m}{n}
\SetMathAlphabet\EuR{bold}{U}{eur}{b}{n}

\makeindex             

\theoremstyle{plain}
\newtheorem{theorem}{Theorem}[section]
\newtheorem{lemma}[theorem]{Lemma}

\newtheorem{corollary}[theorem]{Corollary}
\newtheorem{conjecture}[theorem]{Conjecture}
\newtheorem{problem}[theorem]{Problem}
\newtheorem{question}[theorem]{Question}

\theoremstyle{definition}
\newtheorem{assumption}[theorem]{Assumption}
\newtheorem{definition}[theorem]{Definition}
\newtheorem{example}[theorem]{Example}

\newtheorem{remark}[theorem]{Remark}
\newtheorem{notation}[theorem]{Notation}
\newtheorem{setup}[theorem]{Setup}

{\catcode`@=11\global\let\c@equation=\c@theorem}

\newcommand{\comsquare}[8]                   
{\begin{CD}
#1 @>#2>> #3\\
@V{#4}VV @V{#5}VV\\
#6 @>#7>> #8
\end{CD}
}

\newcommand{\xycomsquare}[8]                   
{\xymatrix
{#1 \ar[r]^{#2} \ar[d]^{#4} &
#3 \ar[d]^{#5}  \\
#6\ar[r]^{#7} &
#8
}
}

\newcommand{\xycomsquareminus}[8]                      
{\xymatrix{#1 \ar[r]^-{#2} \ar[d]^-{#4} &
#3 \ar[d]^-{#5}  \\
#6\ar[r]^-{#7} &
#8
}
}

\newcommand{\calcw}{\mathcal{CW}}

\newcommand{\calh}{\mathcal{H}}

\newcommand{\calf}{\mathcal{F}}
\newcommand{\calp}{\mathcal{P}}

\newcommand{\caln}{{\mathcal N}}

\newcommand{\IC}{{\mathbb C}}

\newcommand{\IF}{{\mathbb F}}

\newcommand{\IQ}{{\mathbb Q}}
\newcommand{\IR}{{\mathbb R}}

\newcommand{\IZ}{{\mathbb Z}}

\newcommand{\an}{\operatorname{an}}

\newcommand{\coker}{\operatorname{coker}}
\newcommand{\colim}{\operatorname{colim}}

\newcommand{\cone}{\operatorname{cone}}

\newcommand{\cyl}{\operatorname{cyl}}
\newcommand{\cost}{\operatorname{cost}}

\newcommand{\ev}{\operatorname{ev}}

\newcommand{\GL}{\operatorname{GL}}

\newcommand{\id}{\operatorname{id}}

\newcommand{\im}{\operatorname{im}}
\newcommand{\Int}{\operatorname{Int}}

\newcommand{\odd}{\operatorname{odd}}
\newcommand{\Ore}{\operatorname{Ore}}
\newcommand{\pr}{\operatorname{pr}}
\newcommand{\RG}{\operatorname{RG}}

\newcommand{\rk}{\operatorname{rk}}
\newcommand{\sign}{\operatorname{sign}}

\newcommand{\topo}{\operatorname{top}}

\newcommand{\tors}{\operatorname{tors}}
\newcommand{\tr}{\operatorname{tr}}
\newcommand{\trun}{\operatorname{trun}}

\newcommand{\vol}{\operatorname{vol}}

\newcommand{\higherlim}[3]{{\setbox1=\hbox{\rm lim}
        \setbox2=\hbox to \wd1{\leftarrowfill} \ht2=0pt \dp2=-1pt
        \mathop{\vtop{\baselineskip=5pt\box1\box2}}
        _{#1}}^{#2}#3}

\newcommand{\version}[1]                       
{\begin{center} last edited on #1\\
last compiled on \today\\
name of texfile: \jobname
\end{center}
}

\newcounter{commentcounter}

\begin{document}




\typeout{---------------------------- L2approx_survey.tex ----------------------------}


\typeout{------------------------------------ Abstract  ----------------------------------------}

\begin{abstract}
In this paper we discuss open problems concerning $L^2$-invariants focusing on
approximation by towers of finite coverings.
\end{abstract}

\maketitle


\typeout{------------------------------- Section 0: Introduction   --------------------------------}

\setcounter{section}{-1}
\section{Introduction}

We want to study in this paper the following general situation:

\begin{setup}\label{set:restricted}
Let $G$ be a (discrete) group together with a descending chain of subgroups
\begin{eqnarray}
  & G = G_0 \supseteq G_1 \supseteq G_2 \supseteq \cdots &
  \label{normal_chain}
\end{eqnarray}
such that $G_i$ is normal in $G$, the index $[G:G_i]$ is finite and $\bigcap_{i \ge 0} G_i = \{1\}$. 

Let $p \colon \overline{X} \to X$ be a $G$-covering. Put $X[i] := G_i\backslash \overline{X}$. 
\end{setup}

We obtain a $[G:G_i]$-sheeted covering $p[i] \colon X[i] \to X$. Its total space $X[i]$
inherits the structure of a finite $CW$-complex, a closed manifold or a closed Riemannian
manifold respectively if $X$ has the structure of a finite $CW$-complex, a closed manifold
or a closed Riemannian manifold respectively.

Let $\alpha$ be a classical
topological invariant such as the Euler characteristic, the signature, the $n$th Betti
number with coefficients in the field $\IQ$ or $\IF_p$, torsion in the sense of
Reidemeister or Ray-Singer, the minimal number of generators of the fundamental group, the
minimal number of generators of  the $n$th homology group with integral coefficients,
or the logarithm of the cardinality of the torsion subgroup of the $n$th homology group
with integral coefficients. We want to study the sequence
\begin{eqnarray*}
  \left(\frac{\alpha(X[i])}{[G:G_i]}\right)_{i \ge 0}.
  \label{alpha(X[i])/[G:G_i]}
\end{eqnarray*}

\begin{problem}[Approximation Problem]\
  \label{pro:approximation_problem}
  \begin{enumerate}

  \item Does the sequence converge?

  \item If yes, is the limit independent of the chain?

  \item If yes, what is the limit?

  \end{enumerate}
\end{problem}

The hope is that the answer to the first two  questions is yes and the limit turns out to
be an $L^2$-analogue $\alpha^{(2)}$ of $\alpha$ applied to the $G$-space $\overline{X}$, i.e., one can prove
an equality of the type
\begin{eqnarray}
 \lim_{i \to \infty} \frac{\alpha(X[i])}{[G:G_i]} = \alpha^{(2)}(\overline{X};\caln(G)).
\label{expected_formula}
\end{eqnarray}
Here $\caln(G)$ stands for the group von Neumann algebra and is a reminiscence of the  fact that the $G$-action on
$\overline{X}$ plays a role. 
Equation~\eqref{expected_formula} is often used to compute the $L^2$-invariant  $\alpha^{(2)}(\overline{X};\caln(G))$ by its
finite-dimensional analogues $\alpha(X[i])$.
On the other hand, it implies the existence of finite  coverings with large $\alpha(X[i])$, if  $\alpha^{(2)}(\overline{X};\caln(G))$
is known to be positive.

For some important invariants $\alpha$ one can prove~\eqref{expected_formula}, for
instance for $\alpha$ the Euler characteristic, the signature or the $n$-th Betti number
with rational coefficients. In other very interesting cases
Problem~\ref{pro:approximation_problem} and the equality~\eqref{expected_formula} are open, 
and hence there is the intriguing and hard challenge to
find a proof. Here we are thinking of $\alpha$ as one of the following invariants:
\begin{itemize}
\item the $n$-th Betti number $b_n(X[i];\IF_p)$ of $X[i]$ with coefficients in the field $\IF_p$ for a prime $p$;
\item  the minimal number of generators $d(G_i)$ or the deficiency ${\rm def}(G_i)$
 of $G_i = \pi_1(X[i])$, if $\overline{X} $ is contractible;
\item 
Reidemeister or Ray-Singer torsion $\rho_{\an}(X[i])$ if $X$ is a closed Riemannian manifold;
\item the logarithm of the cardinality of the torsion in the $n$-th singular homology with
  integer coefficients $\ln\big(\bigl|\tors\bigl(H_n(X[i])\bigr)\bigr|\bigr)$, 
  if $X$ is an aspherical closed manifold and $\overline{X}$ its  universal covering.
\end{itemize}

Here are two highlights of open problems which will be treated in more detail later in the manuscript.
\\[3mm]
\begin{bf} Question~\ref{que:Rank_gradient_cost_first_L2_Betti_number_and_approximation} \end{bf}
(Rank gradient, cost, first $L^2$-Betti number and approximation)
  
  Let $G$ be a finitely presented residually finite group. Let $(G_i)$ be a
  descending chain of normal subgroups of  finite index  of $G$ with $\bigcap_{i \ge 0 } G_i=\{1\}$. 
  Let $F$ be any field. 

   When do we have 
  \[
  \lim_{i \to \infty} \frac{b_1(G_i;F)-1 }{[G:G_i]} = b_1^{(2)}(G) - b_0^{(2)}(G)= {\rm cost}(G)-1 =   \RG(G;(G_i)_{i \ge 0}),
  \]
  where $ {\rm cost}(G)$ denotes the cost and $\RG(G;(G_i)_{i \ge 0})$ the rank gradient?
\\[5mm]
\begin{bf} Conjecture~\ref{con:Homological_growth_and_L2-torsion_for_aspherical_manifolds} \end{bf}
(Homological growth and $L^2$-torsion for aspherical closed manifolds)

  Let $M$ be an aspherical closed manifold of dimension $d \ge 1$ and fundamental group $G =  \pi_1(M)$. 
  Let $\widetilde{M}$ be its universal covering.  Then

  \begin{enumerate}

  \item 
    For any natural number $n$ with $2n \not= d$ we get 
    \[
    b_n^{(2)}(\widetilde{M}) = 0.
    \]
    \noindent
    If $d = 2n$, we have
    \[
    (-1)^{n} \cdot \chi(M)  = b_n^{(2)}(\widetilde{M}) \ge 0.
    \]
     \noindent
    If  $d = 2n$ and $M$ carries a Riemannian metric of negative sectional curvature, then
    \[
    (-1)^n \cdot \chi(M) = b_n^{(2)}(\widetilde{M}) > 0;
    \]

   \item 
    Let $(G_i)_{i \ge 0}$ be  any chain of normal subgroups $G_i \subseteq G$ of finite index $[G:G_i]$
    and trivial intersection $\bigcap_{i \ge 0} G_i = \{1\}$. Put $M[i] = G_i \backslash \widetilde{M}$.

    Then we get for any natural number $n$ and any field $F$
    \[
    b_n^{(2)}(\widetilde{M}) = \lim_{i \to \infty} \frac{b_n(M[i];F)}{[G:G_i]} = \lim_{i \to \infty} \frac{d\bigl(H_n(M[i];\IZ)\bigr)}{[G:G_i]},
    \]
    where $d\bigl(H_n(M[i];\IZ)\bigr)$ is the minimal numbers of generators of $H_n(M[i];\IZ)$,
    and for $n = 1$
    \begin{multline*}
    \quad \quad \quad \quad \quad \quad 
    b_1^{(2)}(\widetilde{M}) = \lim_{i \to \infty} \frac{b_1(M[i];F)}{[G:G_i]} = \lim_{i \to \infty} \frac{d\bigl(G_i/[G_i,G_i]\bigr)}{[G:G_i]} 
    \\
    = RG(G,(G_i)_{i \ge 0}) 
    = \begin{cases} 0 & d \not= 2; \\ -\chi(M) & d = 2; \end{cases}
  \end{multline*}

    \item We get for the truncated Euler characteristic in dimension $m$

    \[
    \lim_{i \to \infty} \frac{\chi_m^{\trun}(M[i])}{[G:G_i]} 
    = \begin{cases}
    \chi(M) & \text{if} \; d \;\text{is even and}\; 2m \ge d;
    \\
    0 & \text{otherwise};
   \end{cases}
    \]

    \item 
    If $d = 2n+1$ is odd, we get for the $L^2$-torsion
    \[
     (-1)^n \cdot \rho^{(2)}_{\an}\bigl(\widetilde{M}\bigr) \ge 0;
    \]
    If $d = 2n+1$ is odd and $M$ carries a Riemannian metric with negative sectional curvature, we have
    \[
     (-1)^n \cdot \rho^{(2)}_{\an}\bigl(\widetilde{M}\bigr) > 0;
    \]

  \item 
    Let $(G_i)_{i \ge 0}$ be any chain of normal subgroups $G_i \subseteq G$ of finite index $[G:G_i]$
    and trivial intersection $\bigcap_{i \ge 0} G_i = \{1\}$. Put $M[i] = G_i \backslash \widetilde{M}$.

   Then we get for  any natural number $n$ with $2n +1 \not= d$
    \[
    \lim_{i \to \infty} \;\frac{\ln\big(\bigl|\tors\bigl(H_n(M[i])\bigr)\bigr|\bigr)}{[G:G_i]} = 0,
    \]
    and we get in the case $d = 2n+1$
    \[
    \lim_{i \to \infty} \;\frac{\ln\big(\bigl|\tors\bigl(H_n(M[i])\bigr)\bigr|\bigr)}{[G:G_i]} = (-1)^n \cdot
    \rho^{(2)}_{\an}\bigl(\widetilde{M}\bigr) \ge 0.
    \]
  \end{enumerate}

The earliest reference, where a version of Problem~\ref{pro:approximation_problem} appears,
is to the best of our knowledge Kazhdan \cite{Kazhdan(1975)}, where for  a closed manifold  $X$ the inequality
$\limsup_{i \to \infty} \frac{b_n(X[i];F)}{[G:G_i]} \le b^{(2)}_n(\overline{X};\caln(G))$
is discussed; see also  Gromov~\cite[pages 13, 153]{Gromov(1993)}.

Commencing with Section~\ref{sec:Dropping_the_finite_index_condition} we will drop the condition
that $[G:G_i]$ is finite.

We assume that the reader is familiar with basic concepts concerning
$L^2$-Betti numbers and $L^2$-torsion.  More information about these notions can be found for instance
in~\cite{Lueck(2002),Lueck(2009algview)}.

Most of the article consists of surveys of open problems and known results. There are 
a few new aspects in this manuscript:

\begin{itemize}

\item In Section~\ref{subsec:Truncated_Euler_characteristics} we introduce the truncated
  Euler characteristic which leads to a high-dimensional version of the rank gradient and
  to Question~\ref{que:asymptotic_Morse_equality_general} about asymptotic Morse
  inequalities;

\item In Theorem~\ref{the:the_uniform_logarithmic_estimate} we discuss a strategy to prove 
the Approximation Conjecture for Fuglede-Kadison determinants under a uniform logarithmic estimate;

\item The vanishing of the regulators on the homology comparing the inner product coming from
a Riemannian metric with the one coming from a triangulation in the $L^2$-acyclic case, see
Theorem~\ref{the:comparing_torsion_and_FK-determinants}, or more generally 
Theorem~\ref{the:comparing_analytic_and_chain_complexes}.

\item The question whether $L^2$-torsion can be approximated by integral torsion depends only on the 
$\IQ G$-chain homotopy type of a  finite based free $L^2$-acyclic $\IZ G$-chain complex,
see Lemma~\ref{lem:invariance_of_homological_conjecture_under_Q-homotopy_equivalence}.

\end{itemize}

{\bf Acknowledgments.}  This paper is financially supported by the Leibniz-Preis of the
author, granted by the Deutsche Forschungsgemeinschaft {DFG}. The author thanks Nicolas Bergeron,
Florian Funke, Holger Kammeyer, Clara L\"oh, Roman Sauer, and Andreas Thom  for their useful comments.

The paper is organized as follows:
\tableofcontents


\typeout{------------------------ Section 1: Euler characteristic and signature  --------------------}

\section{Euler characteristic and signature}
\label{sec:Euler_characteristic_and_signature}


\subsection{Euler characteristic}
\label{subsec:Euler_characteristic}

Let us begin with one of the oldest and most famous invariants, the \emph{Euler characteristic}
$\chi(X)$ for a finite $CW$-complex.  It is defined as $\sum_{n \ge 0 } (-1)^n \cdot c_n$, where 
$c_n$ is the number of $n$-cells.  It is easy to see that it is
multiplicative under finite coverings. Since this implies $\chi(X) =
\frac{\chi(X[i])}{[G:G_i]}$, the answer in this case is yes to the  questions
appearing in Problem~\ref{pro:approximation_problem}, and the limit is
\begin{eqnarray}
  \lim_{i \to \infty} \frac{\chi(X[i])}{[G:G_i]} & = & \chi(X).
  \label{solution_for_chi(X)}
\end{eqnarray}


\subsection{Signature of closed oriented manifolds}
\label{subsec:Signature_of_closed_oriented_manifolds}

Next we consider the \emph{signature} of a closed oriented topological $4k$-dimensional manifold $M$.
It is defined as the signature of the non-degenerate symmetric bilinear $\IR$-pairing
given by the intersection form
\[
H^{2k}(M;\IR) \times H^{2k}(M;\IR) \to \IR, \quad (x,y) \mapsto \langle x \cup
y,[M]_{\IR}\rangle.
\]
It is known that it is multiplicative under finite coverings, however, the proof is more involved
than the one for the Euler characteristic. It follows for instance from Hirzebruch's
Signature Theorem, see~\cite{Hirzebruch(1970)}, or Atiyah's $L^2$-index
theorem~\cite[(1.1)]{Atiyah(1976)} in the smooth case; for closed topological manifolds 
see Schafer~\cite[Theorem 8]{Schafer(1970)}. 
Since this implies $\sign(X) = \frac{\sign(X[i])}{[G:G_i]}$, each of the 
questions appearing in Problem~\ref{pro:approximation_problem} has a positive answer, and the limit is
\begin{eqnarray}
  \lim_{i \to \infty} \frac{\sign(X[i])}{[G:G_i]} & = & \sign(X).
  \label{solution_for_sign(X)_manifolds}
\end{eqnarray}


\subsection{Signature of finite Poincar\'e complexes}
\label{subsec:Signature_of_finite_Poincare_complexes}

The next level of generality is to pass from a topological manifold to a finite Poincar\'e
complex whose definition is due to Wall~\cite{Wall(1967)}.  For them the signature is
still defined if the dimension is divisible by $4$.  There are Poincar{\'e} complexes $X$ for which
the signature is not multiplicative under finite coverings,
see~\cite[Example~22.28]{Ranicki(1992)},~\cite[Corollary~5.4.1]{Wall(1967)}.  Hence the
situation is more complicated here.  Nevertheless, it turns out that the answer in this case each of the 
questions appearing in Problem~\ref{pro:approximation_problem} has a positive answer,  and the limit is
\begin{eqnarray}
  \lim_{i \to \infty} \frac{\sign(X[i])}{[G:G_i]} & = & \sign^{(2)}(X;\caln(G)),
  \label{solution_for_signi(M)_Poincare}
\end{eqnarray}
where $\sign^{(2)}(\overline{X};\caln(G))$ denotes the $L^2$-signature, which is in general
different from $\sign(X)$ for a finite Poincar\'e complex $X$.

Actually, for any closed oriented $4k$-dimensional topological manifold $M$ one has
$\sign(M) = \sign^{(2)}(\overline{M};\caln(G))$. Equation~\eqref{solution_for_signi(M)_Poincare} 
for finite Poincar\'e complexes extends to finite Poincar\'e pairs. For a detailed discussion of these notions
and results we refer to~\cite{Lueck-Schick(2003),Lueck-Schick(2005)}.


\typeout{------------------------ Section 2: Betti numbers --------------------}

\section{Betti numbers}
\label{sec:Betti_numbers}


\subsection{Characteristic zero}
\label{subsec:characteristic_zero}
Fix a field $F$ of characteristic zero. We consider the \emph{$n$-th Betti number
  with $F$-coefficients} $b_n(X;F) := \dim_F(H_n(X;F))$.  Notice that $b_n(X;F) =
b_n(X;\IQ) = \rk_{\IZ}(H_n(X;\IZ))$ holds, where $\rk_{\IZ}$ denotes the rank of a finitely
generated abelian group. In this case each of the 
questions appearing in Problem~\ref{pro:approximation_problem} has a positive answer by the main
result of L\"uck~\cite{Lueck(1994c)}.

\begin{theorem}\label{the:approx_Betti_char_zero}
  Let $F$ be a field of characteristic zero and let $X$ be a finite   $CW$-complex. Then we have
  \[
  \lim_{i \to \infty} \frac{b_n(X[i];F)}{[G:G_i]} = b^{(2)}_n(\overline{X};\caln(G)),
  \]
  where $b^{(2)}(\overline{X};\caln(G))$ denotes the $n$-th $L^2$-Betti number.
\end{theorem}


\subsection{Prime characteristic}
\label{subsec:prime_characteristic}

Fix a prime $p$. Let $F$ be a field of characteristic $p$. We consider the
\emph{$n$-th Betti number with $F$-coefficients} $b_n(X;F) := \dim_F(H_n(X;F))$.  Notice
that $b_n(X;F) = b_n(X;\IF_p)$ holds, where $\IF_p$ is the field of $p$-elements.  In this
setting a general answer to Problem~\ref{pro:approximation_problem} is only known in
special cases.  The main problem is that one does not have an analogue of the von Neumann
algebra in characteristic $p$ and the construction of an appropriate extended dimension
function, see~\cite{Lueck(1998a)}, is not known in general.

If $G$ is torsion-free elementary amenable, one gets the full positive answer 
by Linnell-L\"uck-Sauer~\cite[Theorem~0.2]{Linnell-Lueck-Sauer(2011)}, where we give more
explanations, e.g., about Ore localizations, and actually virtually torsion-free elementary
amenable groups are considered.

\begin{theorem}\label{the:dim_approximation_over_fields}
  Let $F$ be a field (of arbitrary characteristic) and $X$ be a connected finite
  $CW$-complex.  Let $G$ be a torsion-free elementary amenable group. Then:
  \[
  \dim_{FG}^{\Ore} \bigl(H_n(\overline{X};F)\bigr) = \lim_{n \to \infty}
  \frac{b_n(X[i];F)}{[G:G_n]}.
  \]
\end{theorem}

For a brief survey on elementary amenable groups we refer for instance 
to~\cite[Section~6.4.1 on page 256ff]{Lueck(2002)}. Solvable groups are examples of elementary amenable
groups. Every elementary amenable group is amenable, but the converse is not true in general.

Notice that Theorem~\ref{the:dim_approximation_over_fields} is consistent with
Theorem~\ref{the:approx_Betti_char_zero} since for a field $F$ of characteristic zero 
and  a torsion-free elementary amenable group $G$ we
have $b_n^{(2)}(\overline{X};\caln(G)) = \dim_{FG}^{\Ore}
\bigl(H_n(\overline{X};F)\bigr)$. The latter equality follows 
from~\cite[Theorem~6.37 on page~259, Theorem~8.29 on page~330, 
Lemma~10.16 on page~376, and Lemma~10.39 on page~388]{Lueck(2002)}. 

Here is another special case taken from 
Bergeron-L\"uck-Linnell-Sauer~\cite{Bergeron-Linnell-Lueck-Sauer(2014)}, 
(see also Calegari-Emerton~\cite{Calegari-Emerton(2009bounds), Calegari-Emerton(2011)}), where we know 
the answer only for special chains.  Let $p$ be a prime, let $n$ be a positive integer,
and let $\phi\colon G \rightarrow \GL_n (\IZ_p)$ be a homomorphism, where $\IZ_p$ denotes
the $p$-adic integers. The closure of the image of $\phi$, which is denoted by $\Lambda$,
is a $p$-adic analytic group admitting an exhausting filtration by open normal subgroups
$\Lambda_i = \ker \left( \Lambda \rightarrow \GL_n (\IZ / p^i \IZ) \right)$. Put $G_i = \phi^{-1} (\Lambda_i )$.

\begin{theorem} \label{the:BLLS} Let $F$ be a field (of arbitrary characteristic). Put $d= \dim  (\Lambda)$.  
  Let $X$ be a finite $CW$-complex.  Then for any integer $n$ and
  as $i$ tends to infinity, we have:
  \[
  b_n(X[i];F) = b_n^{(2)}(\overline{X};F) \cdot [G : G_i] + O\left([G :
    G_i]^{1-{1/d}} \right),
  \]
  where $b_n^{(2)}(\overline{X};F)$ is the $n$th mod $p$ $L^2$-Betti numbers occurring 
  in~\cite[Definition~1.3]{Bergeron-Linnell-Lueck-Sauer(2014)}.
  In particular
  \[
  \lim_{i \to \infty} \frac{b_n(X[i] ;F)}{[G:G_i]} = b_n^{(2)}(\overline{X};F).
  \]
\end{theorem}

Returning to the setting of arbitrary chains $(G_i)_{i \ge 0}$, we get by the universal coefficient theorem 
$b_n(X[i];\IQ) \le b_n(X[i];F)$ for any field
$F$ and hence by Theorem~\ref{the:approx_Betti_char_zero} the inequality 
\[
\liminf_{i \to  \infty} \frac{b_n(X[i] ;F)}{[G:G_i]} \ge b_n^{(2)}(\overline{X};\caln(G)).
\]  
If $p$ is a prime and we additionally assume that each index $[G:G_i]$ is a $p$-power, then the sequence
$\frac{b_n(X[i] ;F)}{[G:G_i]}$ is monotone decreasing and in particular 
$\lim_{i \to  \infty} \frac{b_n(X[i] ;F)}{[G:G_i]}$ exists, see~\cite[Theorem~1.6]{Bergeron-Linnell-Lueck-Sauer(2014)}.

\begin{conjecture}[Approximation in zero and prime characteristic]
  \label{con:Approximation_in_zero_and_prime_characteristic}
    We get
  \[
  \lim_{i \to \infty} \frac{b_n(X[i] ;F)}{[G:G_i]} = b_n^{(2)}(\overline{X};\caln(G))
  \]
  for all fields $F$ and $n \ge 0$, provided that $X$ is finite,  and $\overline{X}$ is contractible, or, equivalently,
that $X$ is aspherical, $G = \pi_1(X)$ and  $\overline{X}$ is the universal covering $\widetilde{X}$.
\end{conjecture}

The assumption that $\overline{X} $ is contractible is necessary in
Conjecture~\ref{con:Approximation_in_zero_and_prime_characteristic}, 
see~\cite[Example~6.2]{Linnell-Lueck-Sauer(2011)}.  An obvious modification 
of~\cite[Example~6.2]{Linnell-Lueck-Sauer(2011)} applied to $G = \IZ$ and $X = S^1 \vee Y$
for a finite aspherical $CW$-complex $Y$ with $H_n(Y;\IQ) \not= 0$ and $H_n(Y;\IF_p) \not=0$ 
yields a counterexample, where $X$ is aspherical, (but $\overline{X}$ is not the universal covering).

Estimates of the growth of Betti-numbers  in terms of the volume of the underlying manifold 
and examples of aspherical manifolds, where this growth is sublinear,
are given in~\cite{Sauer(2016)}.


\subsection{Minimal number of the generators of the homology}
\label{subsec:Minimal_number_of_the_generators_of_the_homology}

Recall the standard notation that $d(G)$ denotes the minimal
number of generators of a finitely generated group $G$. The 
Universal Coefficient Theorem implies
$d\bigl(H_n(X[i] ;\IZ)\bigr) \ge b_n(X[i];F)$ if $F$ has characteristic zero,
but this inequality is not necessarily true  in prime characteristic.
One can make the following  version of
Conjecture~\ref{con:Approximation_in_zero_and_prime_characteristic},
which is in some sense stronger; see the discussion 
in~\cite[Remark~1.3 and Lemma~2.13]{Lueck(2013l2approxfib)}.

\begin{conjecture}[Growth of number of generators of the homology]
  \label{con:Growth_of_number_of_generators_of_the_homology}
    We get
  \[
  \lim_{i \to \infty} \frac{d\bigl(H_n(X[i] ;\IZ)\bigr)}{[G:G_i]} = b_n^{(2)}(\overline{X};\caln(G)),
   \]
  provided that $\overline{X} $ is contractible.
\end{conjecture}


\typeout{------------------------ Section 3: Rank gradient --------------------}

\section{Rank gradient and cost}
\label{sec:Rank_gradient_and_cost}

Let $G$ be a finitely generated group. Let $(G_i)_{i\ge 0}$
be a descending chain of subgroups of finite index of $G$. The \emph{rank gradient of
  $G$} (with respect to $(G_i)$)  is defined by
\begin{eqnarray}
  \RG(G;(G_i)_{i \ge 0}) &= & \lim_{i \to\infty} \frac{d(G_i) -1}{[G:G_i]}.
  \label{rank_gradient}
\end{eqnarray}
The above limit always exists since for any finite index subgroup $H$ of $G$ one has
$\frac{d(H)-1}{[G:H]} \le d(G)-1$ by the Schreier index formula and hence the sequence
$\frac{d(G_i) -1}{[G:G_i]}$ is a monotone decreasing sequence of non-negative rational
numbers.

The rank gradient was originally introduced by Lackenby~\cite{Lackenby(2005expanders)} as a
tool for studying 3-manifold groups, but is also interesting from a purely group-theoretic
point of view, see, e.g.,
\cite{Abert-Jaikin-Zapirain-Nikolov(2011),Abert-Nikolov(2012),Osin(2011_rankgradient),Schlage-Puchta(2012)}.

In the sequel let $(G_i)_{i \ge 0}$ be a descending chain of normal subgroups of finite index  of
$G$ with $\bigcap_{i \ge 0} G_i=\{1\}$. The following inequalities are known to hold:
\begin{equation}
  \label{inequalities1}
  b_1^{(2)}(G) - b_0^{(2)}(G) \leq \cost(G)-1\leq  \RG(G;(G_i)_{i \ge 0}).
\end{equation}
The first inequality is due to
Gaboriau~\cite[Corollaire~3.16,~3.23]{Gaboriau(2002a)}
and the second was proved by Ab\'ert and
Nikolov~\cite[Theorem~1]{Abert-Nikolov(2012)}.
See~\cite{Gaboriau(2000b),Gaboriau(2002a),Gaboriau(2002b)} for the definition and some 
key results about the cost of a group.

It is not known whether  the  inequalities in~\eqref{inequalities1} are equalities. This is true,
if $G$ is finite since then all values are $|G|^{-1}$.  The
following questions remain open:

\begin{question}[Rank gradient, cost, and first $L^2$ Betti number]
  \label{que:Rank_gradient_cost_and_first_L2_Betti_number}
  Let $G$ be an infinite finitely generated residually finite group. Let $(G_i)_{i \ge 0}$ be a
  descending  chain of normal subgroups of finite index of $G$ with $\bigcap_{i \ge 0}   G_i=\{1\}$. 
 
  Do we have
  \[
  b_1^{(2)}(G) = {\rm cost}(G)-1 = \RG(G;(G_i)_{i \ge 0})?
  \]
\end{question}

\begin{question}[Rank gradient, cost, first $L^2$-Betti number and approximation]
  \label{que:Rank_gradient_cost_first_L2_Betti_number_and_approximation}
  Let $G$ be a finitely presented residually finite group. Let $(G_i)$ be a
  descending chain of normal subgroups of  finite index  of $G$ with $\bigcap_{i \ge 0 } G_i=\{1\}$. 
  Let $F$ be any field. 

   Do we have 
  \[
  \lim_{i \to \infty} \frac{b_1(G_i;F)-1}{[G:G_i]} = b_1^{(2)}(G) - b_0^{(2)}(G) = {\rm cost}(G)-1 =   \RG(G;(G_i)_{i \ge 0})?
  \]
\end{question}

Notice that a positive answer to the questions above also includes the statement, that
$\lim_{i \to \infty} \frac{b_n(X[i];F)}{[G:G_i]}$ and $\RG(G;(G_i)_{i \ge 0})$ are independent of
the chain.

One can ask for any finitely generated group $G$ (without assuming that it is residually
finite) whether $b_1^{(2)}(G) = {\rm cost}(G)-1$ is true, and whether the Fixed Price
Conjecture is true which predicts that the cost of every standard action of $G$, i.e., an
essentially free $G$-action on a standard Borel space with $G$-invariant probability
measure, is equal to the cost of $G$.

\begin{remark}[Minimal number of generators versus rank of the abelianization]
  \label{rem:Minimal_number_of_generators_versus_rank_of_the_abelianization}
  A positive answer to Question~\ref{que:Rank_gradient_cost_first_L2_Betti_number_and_approximation} 
  is equivalent to the assertion
  that
  \[
  \lim_{i \to \infty} \frac{d(G_i) - \rk_{\IZ}(H_1(G_i))}{[G:G_i]} = 0.
  \]
  This is surprising since in general one would not expect that for a finitely generated
  group $H$ the minimal number of generators $d(H)$ agrees with the rank of its
  abelianization $\rk_{\IZ}(H_1(H))$. So a positive answer to
  Question~\ref{que:Rank_gradient_cost_first_L2_Betti_number_and_approximation} would
  imply that this equality is true asymptotically. This raises the question whether this
  equality holds for a ``random group'' in the sense of Gromov.
\end{remark}

\begin{remark}[Known cases]
  \label{rem:known_cases}
  The answers to Question~\ref{que:Rank_gradient_cost_and_first_L2_Betti_number}
  and~\ref{que:Rank_gradient_cost_first_L2_Betti_number_and_approximation} is known to be
  positive if $G$ contains a normal infinite amenable subgroup. Namely, in this case all
  values are $0$ since $\RG(G;(G_i)_{i \ge 0}) = 0$ for all descending chains $(G_i)_{i
    \ge 0}$ of normal subgroups of finite index of $G$ with trivial intersection, as
  proved by Lackenby~\cite[Theorem~1.2]{Lackenby(2005expanders)} when $G$ is finitely
  presented, and by Ab\'ert and Nikolov~\cite[Theorem~3]{Abert-Nikolov(2012)} in general,
  where actually more general chains are considered.
  
  It is also positive for limit groups by Bridson-Kochlukova~\cite[Theorem~A and
  Corollary~C]{Bridson-Kochlukova(2013)}, where all values are $- \chi(G)$.
\end{remark}

\begin{remark}[All conditions are necessary]
  \label{rem:All_conditions_are_necessary}
  One cannot drop in 
  Question~\ref{que:Rank_gradient_cost_first_L2_Betti_number_and_approximation} the assumption 
  that the intersection  $\bigcap_{i \ge 0} G_i$ is trivial.  Indeed, there exists a finitely presented group $G$
  together with a descending chain $(G_i)_{i \ge 0}$ of normal subgroups $G_i$ of finite index of $G$,
  but with non-trivial intersection $\bigcap_{i \ge 0} G_i$, such that
  \[
  \lim_{i \to \infty} \frac{b_1(G_i;\IQ)}{[G:G_i]} < \lim_{i \to \infty}
  \frac{b_1(G_i;\IF_p)}{[G:G_i]} < \lim_{i \to \infty} \frac{d(H_1(G_i;\IZ))}{[G:G_i]} <
  RG(G;(G_i))
  \]
  holds, see Ershov-L\"uck~\cite[Section~4]{Ershov-Lueck(2014)}.

  The condition that each subgroup $G_i$ is normal in $G$ cannot be discarded in
  Question~\ref{que:Rank_gradient_cost_first_L2_Betti_number_and_approximation}.  Indeed,
  one can conclude from Ab\'ert and Nikolov~\cite[Proposition~14]{Abert-Nikolov(2012)},
  see~\cite[Proposition~3.14]{Funke(2013)} for details that there exists for every prime
  $p$ a finitely presented group together with a descending chain $(G_i)_{i \ge 0}$ of 
  (not normal) subgroups  of finite index of $G$   with $\bigcap_{i \ge 0} G_i$ satisfying
  \[
  \lim_{i \to \infty} \frac{b_1(G_i;\IQ)}{[G:G_i]} < b_1^{(2)}(G) < \lim_{i \to \infty}
  \frac{b_1(G_i;\IF_p)}{[G:G_i]} < \lim_{i \to \infty} \frac{d(H_1(G_i;\IZ))}{[G:G_i]} <
  RG(G;(G_i)).
  \]
  One can find also examples where $G$ is the fundamental group of an  oriented hyperbolic $3$-manifold
  of finite volume and the rank gradient is positive for a descending chain $(G_i)_{i \ge 0}$ of 
  (not normal) subgroups  of finite index of $G$   with $\bigcap_{i \ge 0} G_i$ trivial, whereas 
  the first $L^2$-Betti number of $G$ is zero; see Gir\~ao~\cite{Girao(2013), Girao(2014)}.

  Also the condition that $G$ is finitely presented has to appear in
  Question~\ref{que:Rank_gradient_cost_first_L2_Betti_number_and_approximation}. 
 (Notice that in Question~\ref{que:Rank_gradient_cost_and_first_L2_Betti_number} 
  we demand $G$ only to be finitely generated.)
  For   finitely generated $G$ and $F$ of characteristic zero one knows at least 
  $\limsup_{i \to \infty} \frac{b_1(G_i;F))}{[G:G_i]} \le b_1^{(2)}(G)$,
  see L\"uck-Osin~\cite[Theorem~1.1]{Lueck-Osin(2011)}. However, for every prime $p$ there
  exists an infinite finitely generated residually $p$-finite $p$-torsion group $G$ such
  that for any descending chain of normal subgroups $(G_i)_{i \ge 0} $, for which $[G:G_i]$ is
  finite and a power of $p$ and $\bigcap_{i \ge 0} G_i$ is trivial,
  \[
  0 = \lim_{i \to \infty} \frac{b_1(G_i;\IQ)}{[G:G_i]} < b_1^{(2)}(G) \le \lim_{i \to
    \infty} \frac{b_1(G_i;\IF_p)}{[G:G_i]}
  \]
  holds. This follows from Ershov-L\"uck~\cite[Theorem~1.6]{Ershov-Lueck(2014)}
  and L\"uck-Osin~\cite[Theorem~1.2]{Lueck-Osin(2011)}.
\end{remark}


\typeout{------------------ Section 4: A high dimensional version of the rank gradient  ----------------}

\section{A high dimensional version of the rank gradient}
\label{sec:A_high_dimensional_version_of_the_rank_gradient}

One may speculate about the following higher dimensional analogue of
Question~\ref{que:Rank_gradient_cost_first_L2_Betti_number_and_approximation}. 


\subsection{Truncated Euler characteristics}
\label{subsec:Truncated_Euler_characteristics}

Let $d$ be a natural number and let $X$ be a space. Denote by $\calcw_d(X)$ the set of
 $CW$-complexes $Y$ which have  a finite $d$-skeleton $Y_d$ and are homotopy equivalent to $X$.

Provided that $\calcw_d(X)$ is not empty,
define the \emph{$d$-th truncated Euler characteristic of $X$} by
\begin{eqnarray}
  \label{chimin_d} & & 
  \\
  \chi^{\trun}_d(X) & := & 
  \begin{cases}
    \min \{\chi(Y_d) \mid Y \in \calcw_d(X) \} & \text{if} \; d \; \text{is even};
    \\
    \max \{\chi(Y_d) \mid Y   \in \calcw_d(X) \} & \text{if} \; d \; \text{is odd},
  \end{cases}
  \nonumber
\end{eqnarray}
where $\chi(Y_d)$ is the Euler characteristic of the $d$-skeleton $Y_d$ of $Y$.  

If $X$ is a finite $CW$-complex, then $\chi^{\trun}_d(X) = \chi(X)$ if $d \ge \dim(X)$.

Fix a $G$-covering $\overline{X} \to X$. Consider $Y \in \calcw_d(X)$.
Choose a homotopy equivalence $h \colon Y \to X$.    We obtain a 
$G$-covering $\overline{Y} \to Y$ by applying the pullback construction to
$\overline{X} \to X$ and  $h \colon Y \to X$. We get using~\cite[Theorem~6.80~(1)  on page~277]{Lueck(2002)}
\begin{eqnarray*}
  \chi(Y_d)
  & = & 
  \chi^{(2)}(\overline{Y_d};\caln(G))
  \\
  & = & 
  \sum_{n = 0}^d (-1)^n \cdot b_n^{(2)}(\overline{Y_d};\caln(G))
  \\
  & = & 
  (-1)^d \cdot b_d^{(2)}(\overline{Y_d};\caln(G))  +  \sum_{n = 0}^{d-1} (-1)^n \cdot b_n^{(2)}(\overline{Y_d};\caln(G))
  \\
  & = & 
  (-1)^d \cdot b_d^{(2)}(\overline{Y_d};\caln(G))  +  \sum_{n = 0}^{d-1} (-1)^n \cdot b_n^{(2)}(\overline{Y};\caln(G))
  \\
  & = & 
  (-1)^d \cdot \bigl(b_d^{(2)}(\overline{Y_d};\caln(G))   -  b_d^{(2)}(\overline{Y};\caln(G))\bigr) +  
  \sum_{n = 0}^{d} (-1)^n \cdot b_n^{(2)}(\overline{Y};\caln(G)) 
  \\
  & = & 
  (-1)^d \cdot \bigl(b_d^{(2)}(\overline{Y_d};\caln(G))   -  b_d^{(2)}(\overline{Y};\caln(G))\bigr) +  
  \sum_{n = 0}^{d} (-1)^n \cdot b_n^{(2)}(\overline{X},\caln(G)).
\end{eqnarray*}
Since $b_d^{(2)}(\overline{Y_d};\caln(G)) \ge b_d^{(2)}(\overline{Y};\caln(G))$ holds,
we always have the inequality
\[
\chi(Y_d) \; \begin{cases} \ge \; \sum_{n = 0}^d (-1)^n \cdot b_n^{(2)}(\overline{X},\caln(G)) &
  \text{if} \; d \; \text{is even};
  \\
  \le \; \sum_{n = 0}^d (-1)^n \cdot b_n^{(2)}(\overline{X},\caln(G)) & \text{if} \; d \; \text{is odd}.
\end{cases}
\]
This implies that $\chi^{\trun}_d(X)$ is a well-defined integer satisfying
\[
\chi^{\trun}_d(X) \; \begin{cases} \ge \; \sum_{n = 0}^d (-1)^n \cdot b_n^{(2)}(\overline{X},\caln(G)) &
  \text{if} \; d \; \text{is even};
  \\
  \le \; \sum_{n = 0}^d (-1)^n \cdot b_n^{(2)}(\overline{X},\caln(G)) & \text{if} \; d \; \text{is odd}.
\end{cases}
\]

Next we show that the limit $\lim_{i \to \infty} \frac{\chi^{\trun}_d(X[i])} {[G:G_i]}$ always
exists.  Consider a natural number $i$. Choose an element $Y[i] \in \calcw_d(X[i])$ such that
$\chi^{\trun}_d(X[i]) = \chi(Y[i]_d)$ holds. We can find a $[G_i:G_{i+1}]$-sheeted covering
$Y[i+1] \to Y[i]$ such that $Y[i+1]$  belongs to $\calcw_d([X[i+1])$. Obviously  
$\chi(Y[i+1]_d) = \chi(Y[i]_d) \cdot [G_i : G_{i+1}]$. Suppose that $d$ is
even. We conclude
\begin{eqnarray*}
  \frac{\chi_d^{\trun}(X[i+1])}{[G:G_{i+1}]} 
  & = & 
  \frac{\chi_d^{\trun}(X[i+1])}{[G:G_i] \cdot [G_i:G_{i+1}]}
  \\
  & \le & 
  \frac{\chi(Y[i+1]_d)}{[G:G_i] \cdot [G_i:G_{i+1}]}
  \\
  & = & 
  \frac{\chi(Y[i]_d)}{[G:G_i]}
  \\
  & = & 
  \frac{\chi_d^{\trun}(X[i])}{[G:G_i]}.
\end{eqnarray*}
Hence the sequence $\left(\frac{\chi^{\trun}_d(X[i])} {[G:G_i]}\right)_{i \ge 0}$ is
monotone decreasing. Since we get by an argument similar to the one above
\[
\frac{\chi^{\trun}_d(X[i])}{[G:G_i]} \ge \sum_{n = 0}^d (-1)^n \cdot b_n^{(2)}(X[i];\caln(G/G_i)) 
=  \sum_{n = 0}^d (-1)^n \cdot \frac{b_n(X[i];\IQ)}{[G:G_i]}
\]
for all $i$, we conclude from Theorem~\ref{the:approx_Betti_char_zero}
that the sequence $\left(\frac{\chi^{\trun}_d(X[i])} {[G:G_i]}\right)_{i \ge 0}$ is bounded from
below by $\sum_{n = 0}^d (-1)^n \cdot b_n^{(2)}(\overline{X};\caln(G))$. Hence its limit exists and satisfies
\[
\lim_{i \to \infty} \frac{\chi^{\trun}_d(X[i])} {[G:G_i]} \ge \sum_{n = 0}^d (-1)^n \cdot
b_n^{(2)}(\overline{X};\caln(G)),
\]
if $d$ is even.  Provided that $d$ is odd, one analogously shows that the sequence
$\left(\frac{\chi^{\trun}_d(X[i])} {[G:G_i]}\right)_{i \ge 0}$ is monotone increasing,
bounded from above by $\sum_{n = 0}^d (-1)^n \cdot b_n^{(2)}(\overline{X};\caln(G))$ and
hence converges with
\[
\lim_{i \to \infty} \frac{\chi^{\trun}_d(X[i])} {[G:G_i]} \le \sum_{n = 0}^d (-1)^n \cdot
b_n^{(2)}(\overline{X};\caln(G)).
\]
This leads to

\begin{question}[Asymptotic Morse equality]\label{que:asymptotic_Morse_equality_general}
  Let $\overline{X} \to X$ be a $G$-covering and let $d$ be a natural number such that
  $\calcw_d(X)$ is not empty.  When do we have
  \[
  \lim_{i \to \infty} \frac{\chi^{\trun}_d(X[i])} {[G:G_i]} = \sum_{n = 0}^d (-1)^n \cdot
  b_n^{(2)}(\overline{X};\caln(G))?
  \]
\end{question}

In this generality, the answer to Question~\ref{que:asymptotic_Morse_equality_general}
is not positive in general. For instance if $G$ is trivial and $d =1$, a positive answer to
Question~\ref{que:asymptotic_Morse_equality_general} would mean 
for a connected $CW$-complex $X$ with non-empty
$\calcw_1(X)$ that $\pi_1(X)$ is finitely generated and satisfies $d(\pi_1(X)) = b_1(X)$
which is not true in general.

Of particular interest is the case where $\overline{X}$ is contractible, or,
equivalently, $X = BG$ and $\overline{X} = EG$.  Since $\chi^{\trun}_d(X)$ depends only on the homotopy type of
$X$, we will abbreviate $\chi^{\trun}_d(G_i) := \chi^{\trun}_d(BG_i)$, provided that
$\calcw_d(BG)$ for some (and hence all) model for $BG$ is not empty. Then
Question~\ref{que:asymptotic_Morse_equality_general} reduces to
the following question, for which we do not know an example where the answer is negative.

\begin{question}[Asymptotic Morse equality for groups]
\label{que:asymptotic_Morse_equality}
  Let $G$ be a group  and let $d$ be a natural number such that $\calcw_d(BG)$ is not empty.
 When do we have
  \[
  \lim_{i \to \infty} \frac{\chi^{\trun}_d(G_i)} {[G:G_i]} = \sum_{n = 0}^d (-1)^n \cdot
  b_n^{(2)}(G)?
  \]
\end{question}

\begin{example}[Morse relation in degree $d = 1,2$]
  \label{exa:Morse_relation_in_degree_d_is_1,2}
  Question~\ref{que:asymptotic_Morse_equality} is in the case $d = 1$ precisely
  Question~\ref{que:Rank_gradient_cost_first_L2_Betti_number_and_approximation}, since a
  group $H$ is finitely generated if and only if there is a model for $BH$ with finite $1$-skeleton
  and in this case $\chi^{\trun}_1(H) = 1 - d(H)$.

  In the case $d =2$ Question~\ref{que:asymptotic_Morse_equality} can be rephrased as the
  question, for a finitely presented group $G$, when  do we have
  \[
  \lim_{i \to \infty} \frac{1 - {\rm def}(G_i)}{[G:G_i]} = b_2^{(2)}(G) - b_1^{(2)}(G) + b_0^{(2)}(G),
  \]
  where ${\rm def}(H)$ denotes for a finitely presented group $H$ its \emph{deficiency},
  i.e, the maximum over the numbers $g-r$ for all finite presentations $H = \langle s_1,
  s_2, \ldots , s_g \mid R_1, R_2, \ldots, R_r\rangle$.
\end{example}

\begin{remark}[Asymptotic Morse inequalities imply Approximation for Betti numbers over any field]
\label{rem:Asymptotic_Morse_inequalities_imply_Approximation_for_Betti_numbers_over_any_field}
 Let $G$ be a group with a finite model for $BG$. It is not hard to show
that Conjecture~\ref{con:Approximation_in_zero_and_prime_characteristic} is true for $G$,
provided that  the answer to Question~\ref{que:asymptotic_Morse_equality} is positive for all $d \ge 0$.
The main idea of the proof is to show for every field  $F$ and  $CW$-complex $Z$ with non-empty $\calcw_d(Z)$

\[
\chi_d^{\trun}(Z) \; \begin{cases} \ge \; \sum_{n = 0}^d (-1)^n \cdot b_n(Z;F) &
  \text{if} \; d \; \text{is even};
  \\
  \le \; \sum_{n = 0}^d (-1)^n \cdot b_n(Z;F) & \text{if} \; d \; \text{is odd},
\end{cases}
\]
and then  show in the situation of  Conjecture~\ref{con:Approximation_in_zero_and_prime_characteristic} 
for $n = 0,1,2, \ldots$ by induction using Theorem~\ref{the:approx_Betti_char_zero} the equality
\[
\lim_{i \to \infty} \frac{b_n(X[i];F)}{[G:G_i]} = b_n^{(2)}(\widetilde{X}).
\]
More generally,  Conjecture~\ref{con:Growth_of_number_of_generators_of_the_homology} is true for $G$,
provided that  the answer to Question~\ref{que:asymptotic_Morse_equality} is positive for all $d \ge 0$, since for 
$Y \in \calcw_d(Z)$ we get $b_d(Y_d;\IQ) = d(H_d(Y_d;\IZ)) \ge  d(H_d(Y;\IZ)) =  d(H_d(Z;\IZ))$.
\end{remark}


\subsection{Groups with slow growth}
\label{subsec:groups_with_slow_growth}

The answer to Question~\ref{que:asymptotic_Morse_equality}
is positive by Bridson-Kochlukova~\cite[Theorem~A and Corollary~C] {Bridson-Kochlukova(2013)}
if $G$ is a limit group.  Their proofs are based on various notions  of groups with slow growth.
It is interesting that limit groups may have non-trivial first $L^2$-Betti numbers.

Here is a another  case, where the answer to
Question~\ref{que:asymptotic_Morse_equality} is positive.  
Following Bridson-Kochlukova~\cite[page~4] {Bridson-Kochlukova(2013)} we make

\begin{definition}[Slow growth in dimensions $\le d$]
  \label{def:Slow_growth_in_dimensions_le_d}
  We say that a residually finite group has \emph{slow growth in dimension $\le d$} if for
  any chain $(G_i)_{i \ge 0}$ of normal subgroups of finite index with trivial
  intersection there is a choice of $CW$-complexes $(X[i])_{i \ge 0} $ such that $X[i]$ has a
  finite $d$-skeleton and is a model for $BG_i$ for each $i \ge 0$, and 
  $\lim_{i \to \infty} \frac{c_k(X[i])}{[G:G_i]} = 0$ holds for every $k = 0,1,2 \ldots, d$, where
  $c_k(X[i])$ is the number of $k$-cells in $X[i]$.
\end{definition}

\begin{lemma} \label{lem:consequences_of_slow_growth}
Suppose that $G$ has \emph{slow growth in dimension $\le d$}.
Then we get for $k = 0,1,2, \ldots ,d$
\begin{eqnarray*}
\lim_{i \to \infty} \frac{\chi^{\trun}_k(G_i)}{[G:G_i]} & = & 0,
\\
b_k^{(2)}(G) & = & 0.
\end{eqnarray*}
\end{lemma}
\begin{proof}
By assumption there is a choice of $CW$-complexes $(X[i])_{i \ge 0} $ such that $X[i]$ has a
  finite $d$-skeleton and is a model for $BG_i$ for each $i \ge 0$, and 
  $\lim_{i \to    \infty} \frac{c_n(X[i])}{[G:G_i]} = 0$ holds for every $n = 0,1,2 \ldots, d$.
Since $b_n(G_i;\IQ) \le c_n(X[i])$ holds, we conclude
$\lim_{i \to    \infty} \frac{b_n(G_i;\IQ)}{[G:G_i]} = 0$ for every $n = 0,1,2 \ldots, d$.
Theorem~\ref{the:approx_Betti_char_zero} implies
\[
b_k^{(2)}(G) = 0 \quad \text{for}\; k = 0,1,2 \ldots, d.
\]
If $k \in \{0,1,2, \ldots, d\} $ is even, we get 
\begin{eqnarray*}
0  
& = & 
\sum_{n = 0}^k (-1)^n \cdot b_n^{(2)}(G) 
\\
& \le &
\lim_{i \to \infty} \frac{\chi_k^{\trun}(G_i)}{[G:G_i]}
\\
& \le &
\lim_{i \to \infty} \frac{\chi\bigl((X[i])_k\bigr)}{[G:G_i]}
\\
& = &
\lim_{i \to \infty} \sum_{n = 0}^k \frac{c_n(X[i])}{[G:G_i]}
\\
& = &
 \sum_{n = 0}^k  \lim_{i \to \infty}\frac{c_n(X[i])}{[G:G_i]}
\\
& = &
 0,
\end{eqnarray*}
and hence 
\[
\lim_{i \to \infty} \frac{\chi^{\trun}_k(G_i)}{[G:G_i]}  =  0.
\]
The proof in the case where $k$ is odd is analogous.
\end{proof}

\begin{lemma}
\label{lem:slow_growth_and_extensions}
Let $1 \to K \xrightarrow{j} G \xrightarrow{q} Q \to 1$ be an extension of groups.
Suppose that $K$ has slow growth in dimensions $\le d$. Suppose that there is a model for $BQ$ with
finite $d$-skeleton or that there is a model for $BG$ with finite $d$-skeleton.  

Then $G$ has slow growth in dimensions $\le d$.
\end{lemma}
\begin{proof} 
If $BG$ has a model with finite $d$-skeleton, then also $BQ$ has a model with finite $d$-skeleton 
by~\cite[Lemma~7.2~(2)]{Lueck(1997a)}. Hence it suffices to treat the case 
where   $BQ$ has a model with finite $d$-skeleton.

Consider any chain $(G_i)_{i \ge 0}$ of normal subgroups of finite index with trivial intersection.
Put $K_i = j^{-1}(G_i)$ and $Q_i = q(G_i)$. 
We obtain an exact sequence of groups $1 \to K_i \xrightarrow{j_i} G_i \xrightarrow{q_i} Q_i \to 1$,
where $j_i$ and $q_i$ are obtained from $j$ and $q$ by restriction. The 
subgroups $K_i \subseteq K$, $G_i \subseteq G$ and $Q_i \subseteq Q$
are normal subgroups of finite index and $[G:G_i] = [K : K_i] \cdot [Q:Q_i]$. 
We have $\bigcap_{i \ge 0} K_i = \{1\}$. By assumption 
there is a choice of $CW$-complexes $(X[i])_{i \ge 0} $ such that $X[i]$ has a
 finite $d$-skeleton and is a model for $BK_i$ for each $i \ge 0$, and 
$\lim_{i \to  \infty} \frac{c_m(X[i])}{[K:K_i]} = 0$ holds for every $m = 0,1,2 \ldots, d$, where
$c_m(X[i])$ is the number of $m$-cells in $X[i]$.
Choose a $CW$-model $Z$ for $BQ$ with finite $d$-skeleton. 
Let $Z_i \to BQ$ be the $[Q:Q_i]$-sheeted finite covering associated to $Q_i \subseteq Q$.
Equip $Z_i$ with the $CW$-structure induced by the one of $Z$. Then $Z_i$ is a model for $BQ_i$,
has a finite $d$-skeleton, and we get for the number of $n$-cells for $n \in \{0,1,2, \ldots ,d\}$ that
\begin{eqnarray*}
c_n(Z_i) 
& = &
[Q:Q_i] \cdot c_n(Z).
\end{eqnarray*}
There is a fibration $X[i] \to BG_i \to Z_i$ such that after taking fundamental groups we
obtain the exact sequence $1 \to K_i \xrightarrow{j_i} G_i \xrightarrow{q_i} Q_i \to
1$. Then one can find a $CW$-complex $Y_i$ which is a model for $BG_i$ such that we get
for the number of $k$-cells for $k \in \{0,1,2, \ldots ,d\}$
\begin{eqnarray*}
c_k(Y_i) 
& = & 
\sum_{m +n = k} c_m(X[i]) \cdot c_n(Z_i),
\end{eqnarray*}
see for instance~\cite[Section~3]{Farrell-Lueck-Steimle(2010)}. This implies for
$k \in \{0,1,2, \ldots ,d\}$
\begin{eqnarray*}
\lim_{i \to \infty} \frac{c_k(Y_i)}{[G:G_i]} 
&  = & 
\lim_{i \to \infty}  \sum_{m +n = k} \frac{c_m(X[i]) \cdot c_n(Z_i)}{[G:G_i]}
\\
&  = & 
\lim_{i \to \infty}   \sum_{m +n = k} \frac{c_m(X[i])}{[K:K_i]}  \cdot \frac{c_n(Z_i)}{[Q:Q_i]}
\\
&  = & 
\sum_{m +n = k} c_n(Z) \cdot \lim_{i \to \infty}   \frac{c_m(X[i])}{[K:K_i]}  
\\
& = 0.
\end{eqnarray*}
\end{proof}

\begin{example}[Examples of groups with slow growth in dimensions $\le d$]
  \label{exa:Examples_of_groups_with_slow_growth_in_dimensions_le_d}
  A residually finite group has slow growth in dimensions $\le 0$ if and only if it is infinite.
  
  Obviously $\IZ$ has slow growth in dimensions $\le d$ for all natural numbers $d$ since
  any non-trivial subgroup $K$ of $\IZ$ is isomorphic to $\IZ$ again and has $S^1$ as
  model for $BK$. 

  We conclude from Lemma~\ref{lem:slow_growth_and_extensions} that
  any infinite virtually poly-$\IZ$-group has slow growth in dimensions $\le d$. 

  Moreover,  if $G$ is any residually finite 
  group which possesses  a finite sequence $K_0 \subseteq K_1 \subseteq \cdots \subseteq K_n = G$ of subgroups
  such that $K_0 \cong \IZ$, $K_i$ is normal in $K_{i+1}$ and $B(K_{i+1}/K_i)$ has a model with finite $d$-skeleton for $i = 0, \ldots, (n-1)$, 
  then $G$ has slow growth in  dimensions $\le d$ by Lemma~\ref{lem:slow_growth_and_extensions}.
\end{example}


\typeout{------------------------ Section 5: Speed of convergence --------------------}

\section{Speed of convergence}
\label{sec:Speed_of_convergence}

  The speed of convergence of 
  \[
   \lim_{i \to \infty} \frac{b_n(X[i];F)}{[G:G_i]} = b_n^{(2)}(\widetilde{X})
   \] 
    (if it converges) and  of
   \[
   \lim_{i \to \infty} \frac{d(G_i) -1)}{[G:G_i]} = \RG(G;(G_i)_{i \ge 0})
   \] 
  can be arbitrarily slow for one chain
  and very fast for another chain in a sense that we shall now explain. Fix a prime $p$ and two functions
   $F^s,F^f \colon \{i \in \IZ    \mid i \ge 1\} \to (0,\infty)$  such that 
   \begin{eqnarray*}
   \lim_{i \to \infty} F^s(i) & = & 0;
   \\
    \lim_{i \to \infty} F^f(i) & = & 0;
    \\
    \lim_{i \to \infty} i \cdot F^f (i) & = & \infty.
  \end{eqnarray*}


\subsection{Betti numbers}
\label{subsec:Betti_numbers}

  \begin{theorem}\label{the:L2-Betti_numbers_speed_of_convergence}
  For every integer $n \ge 1$, there is a closed  $(2n+1)$-dimensional 
Riemannian manifold $X$ with non-positive sectional curvature
  and two chains  $(G_i^s)_{i \ge 0}$ and $(G_i^f)_{i \ge 0}$ for $G = \pi_1(X)$ such that $G_i^s$ and
  $G_i^f$ are normal subgroups of $G$ of finite $p$-power index, the
  intersections $\bigcap_{i \ge 0} G_i^s$ and $\bigcap_{i \ge 0} G_i^f$ are trivial, and
  we have for every field $F$
  \begin{eqnarray*}
  \lim_{i \to \infty} \frac{b_n(X^s[i];F)}{[G:G^s_i]} & = & b_n^{(2)}(\widetilde{X}) = 0;
  \\
  \frac{b_n(X^s[i];F)}{[G^s:G^s_i]} & \ge & F^s([G:G^s_i]);
   \\
   \lim_{i \to \infty}  \frac{b_n(X^f[i];F)}{[G:G^f_i]} & = & b_n^{(2)}(\widetilde{X}) = 0;
  \\
   \frac{b_n(X^f[i];F)}{[G:G^f_i]} & \le & F^f([G:G^f_i]).
 \end{eqnarray*}
\end{theorem}
\begin{proof}
Consider any finite connected $CW$-complex $Y$ with universal
covering $\widetilde{Y} \to Y$ and $b_n^{(2)}(\widetilde{Y}) +
b_{n-1}^{(2)}(\widetilde{Y}) > 0$ such that $K = \pi_1(Y)$ is infinite and residually
$p$-finite. Choose any chain $(K_i)_{i \ge 0}$ of normal subgroups of $K$ of finite index
$[K : K_i]$ which is a power of $p$ and with trivial intersection $\bigcap_{i \ge 0} K_i =
\{1\}$. Because of Theorem~\ref{the:approx_Betti_char_zero} 
and~\cite[Theorem~1.6]{Bergeron-Linnell-Lueck-Sauer(2014)} the limit 
$\frac{b_n(Y[i];F) + b_{n-1}(Y[i];F)}{[K:K_i]}$ exists and is greater than $0$.
Hence there exist real numbers $C_1$
and $C_2$ (independent of $i$) with $0 < C_1 \le C_2$ such that for each $i \ge 1$
  \[
  C_1 \le  \frac{b_n(Y[i];F) + b_{n-1}(Y[i];F)}{[K:K_i]} 
\le C_2.
  \]
  Let $k_i$ be the natural number for which $[K:K_i] = p^{k_i}$ holds. Then $(k_i)_{i \ge 0}$ is a monotone increasing 
   sequence of natural numbers with $\lim_{i \to \infty} k_i   = \infty$. 
  Since $\lim_{k \to \infty} F^s\bigl(p^k \cdot p^m\bigr) = 0$  holds for any integer $m \ge 0$, 
  we can construct a strictly monotone increasing sequence of natural numbers 
  $(j_i)_{i \ge 0}$ such that we get for all $i \ge 0$
  \begin{eqnarray*}
   F^s\bigl(p^{k_{j_i}} \cdot p^i\bigr) & \le &  C_1 \cdot  p^{-i}.
 \end{eqnarray*}
 Since $\lim_{n \to \infty} p^n \cdot F^f\bigl(p^k \cdot p^n\bigr) = \infty$ for any natural number $k$,
 we can construct a strictly monotone increasing sequence of natural numbers $(n_i^f)_{i \ge 0}$
 satisfying
  \begin{eqnarray*}
  C_2 \cdot p^{-n_i^f} & \le &  F^f\bigl(p^{k_{j_i}} \cdot p^{n_i^f}\bigr).
 \end{eqnarray*}
   Put
   \[ 
   X = Y \times S^1.
   \]
   Then $G = \pi_1(X)$ can be identified with $K \times \IZ$. We conclude $b_n^{(2)}(\widetilde{X}) = 0$
   from~\cite[Theorem~1.39 on page~42]{Lueck(2002)}.  Put
  \begin{eqnarray*}
   G_i^s & = &  K_{j_i} \times (p^i \cdot \IZ);
   \\
   G_i^f & = &  K_{j_i} \times (p^{n_i^f} \cdot \IZ).
 \end{eqnarray*}
  Then $(G_i^s)_{i \ge 0}$ is a chain
  of normal subgroups of $G$ with finite index $[G:G_i^s] = [K:K_{j_i}] \cdot p^i$ which is a $p$-power, namely $p^{k_{j_i} + i}$, and trivial
  intersection $\bigcap_{i \ge 0} G_i^s = \{1\}$, and analogously for  $(G_i^f)_{i \ge 0}$.  We estimate
  using the K\"unneth formula:
    \begin{eqnarray*}
    \frac{b_n(X^s[i];F)}{[G:G_i^s]}
    & = & 
    \frac{b_n(Y[j_i];F) + b_{n-1}(Y[j_i];F)}{[K:K_{j_i}] \cdot p^{i}}
     \\
     & \ge &
     C_1 \cdot p^{-i}
     \\
     & \ge &
     F^s\bigl(p^{k_{j_i}} \cdot p^{i}\bigr)
     \\
     & = & 
     F^s([G:G_i^s]),
  \end{eqnarray*}
  and 
 \[
\frac{b_n(X^s[i];F)}{[G:G_i^s]} =     \frac{b_n(Y[j_i];F) + b_{n-1}(Y[j_i];F)}{[K:K_{j_i}] \cdot p^{i}}     \le      C_2 \cdot p^{-i}.
 \]    
    The latter implies $\lim_{i \to \infty}   \frac{b_n(X^s[i];F)}{[G:G_i^s]} = 0$.
    We estimate
    \begin{eqnarray*}
    \frac{b_n(X^f[i];F)}{[G:G_i^f]}
    & = & 
    \frac{b_n(Y[j_i];F) + b_{n-1}(Y[j_i];F)}{[K:K_{j_i}] \cdot p^{n_i^f}}
     \\
     & \le &
     C_2 \cdot p^{-n_i^f}
     \\
     & \le &
     F^f\bigl(p^{k_{j_i}} \cdot p^{n_i^f}\bigr)
     \\
     & = & 
     F^f([G:G_i^f]).
  \end{eqnarray*}
Since $\lim_{i \to \infty}  C_2 \cdot p^{-n_i^f} = 0$, we also get $\lim_{i \to \infty}   \frac{b_n(X^f[i];F)}{[G:G_i^f]} = 0$.

It remains to construct the desired finite $CW$-complex $Y$.  The fundamental group of an oriented closed surface 
of genus $\ge 2$ is residually free and hence residually $p$-finite for any prime $p$ by~\cite{Baumslag(1962)}.
The $L^2$-Betti numbers of its universal covering are all zero except in dimension $1$, where it is non-zero,
see~\cite[Example~1.36 on page~40]{Lueck(2002)}.
We conclude from  the K\"unneth  formula for $L^2$-Betti numbers~\cite[Theorem~6.54~(5) on page~266]{Lueck(2002)}
that an example for $Y$ is the  direct product of $n$ closed oriented surfaces of genus $\ge 2$. 
So we can arrange that $X$ is an aspherical closed
$(2n+1)$-dimensional Riemannian manifold with non-positive sectional curvature.
\end{proof}

Theorem~\ref{the:L2-Betti_numbers_speed_of_convergence} implies
that one can find  for any $\epsilon > 0$ two chains $(G_i^s)_{i \ge 0}$ and $(G_i^f)_{i \ge 0}$
 satisfying
\begin{eqnarray*}
\lim_{i \to \infty} \frac{b_n(X[i];F)}{{[G:G^s_i]}^{1-\epsilon}} & = & \infty;
\\
\lim_{i \to \infty} \frac{b_n(X[i];F)}{{[G:G^f_i]}^{\epsilon}} & = & 0,
\end{eqnarray*}
since we can take $F^s(i) = i^{-\epsilon/2}$ and $F^f(i) = i^{\epsilon/2 - 1}$.

The condition $\lim_{i \to \infty} i \cdot F^f(i) = \infty$ is reasonable. Namely, in
  most cases one would expect $\lim_{i \to \infty} b_n(X[i];F) = \infty$ and if this is
  true, we get
\[
\lim_{i \to \infty} [G:G_i] \cdot \frac{b_n(X[i];F)}{[G:G_i]} = \infty.
\]


\subsection{Rank gradient}
\label{subsec:rank_gradient}

\begin{theorem}\label{the:rank_gradient_speed_of_convergence}
 Let $G$ be the direct product of $\IZ$ with a finitely generated free group of rank $\ge 2$
 or the product of $\IZ$ with the fundamental group of a closed surface of genus $\ge 2$.

Then  there are  two chains
  $(G_i^s)_{i \ge 0}$ and $(G_i^f)_{i \ge 0}$  such that $G_i^s$ and
  $G_i^f$ are normal subgroups of $G$ of finite $p$-power index, the
  intersections $\bigcap_{i \ge 0} G_i^s$ and $\bigcap_{i \ge 0} G_i^f$ are trivial, and
  we have

  \begin{eqnarray*}
  \RG(G,(G_i^s)_{i\ge 0})   & = &  0;
  \\
  \frac{d(G_i^s) -1 }{[G:G_i^s]}   & \ge & F^s([G:G^s_i]);
  \\
 \RG(G,(G_i^f)_{i\ge 0}) & = &0;
  \\
   \frac{d(G_i^f) -1 }{[G:G_i^f]}  & \le & F^f([G:G^f_i]).
 \end{eqnarray*}
\end{theorem}
\begin{proof}
Essentially the same argument as in the proof of Theorem~\ref{the:L2-Betti_numbers_speed_of_convergence}
also applies  to the rank gradient. Let $K$ be a finitely presented
group with $b_1^{(2)}(K) > 0$. Choose any chain $(K_i)_{i \ge 0}$ of normal subgroups of $K$ of finite index
$[K : K_i]$ which is a power of $p$ and with trivial intersection $\bigcap_{i \ge 0} K_i =
\{1\}$.  Then $\left(\frac{d(K_i)-1}{[K:K_i]}\right)_{i \ge 0}$ is a monotone decreasing sequence.
Its limit $\RG(K,(K_i)_{i \ge 0}) := \lim_{i \to \infty} \frac{d(K_i)-1}{[K:K_i]}$
exists and is greater than $b_1^{(2)}(K)$. Hence we can choose constants $C_1> 0$ and $C_2 > 0$
such that such that for each $i \ge 1$ we get
\[
 C_1 \le  \frac{d(K_i) -1 }{[K:K_i]} \le  \frac{d(K_i)}{[K:K_i]} \le C_2.
\]
Put $G = K \times \IZ$. Now construct the sequences $(j_i)_{i \ge 0}$ and $(n_i^s)_{i\ge 0}$ and define $G_i^s$ and
$G_i^f$ as in the proof of Theorem~\ref{the:L2-Betti_numbers_speed_of_convergence}.
Then $d(K_{j_i}) \le d(G_i^s) \le d(K_{j_i}) +1$ and $d(K_{j_i}) \le d(G_i^f) \le d(K_{j_i}) +1$ holds. Now a calculation analogous to the one
in the proof of Theorem~\ref{the:L2-Betti_numbers_speed_of_convergence} shows
\begin{eqnarray*}
   \frac{d(G_i^s) -1 }{[G:G_i^s]}  & \ge & F^s([G^s:G^s_i]);
  \\
  \frac{d(G_i^f) -1 }{[G:G_i^f]}  & \le & F^f([G^f:G^f_i]).
 \end{eqnarray*}
If $K$ is a finitely generated free group of rank $\ge 2$
or  the fundamental group of a closed surface of genus $\ge 2$,
then $K$ is finitely presented, residually $p$-finite  and $b_1^{(2)}(K) > 0$.
\end{proof}


\typeout{------------------------ Section 6: Determinants --------------------}

\section{The Approximation Conjecture for Fuglede-Kadison determinants}
\label{sec:The_Approximation_Conjecture_for_Fuglede-Kadison_determinants}

Let $A \in M_{r,s}(\IQ G)$ be a matrix. It induces by right multiplication a $G$-equivariant bounded 
operator $r_A^{(2)} \colon L^2(G)^r \to L^2(G)^s$.  We denote by
${\det}^{(2)}_{\caln(G)}\bigl(r_A^{(2)}\colon L^2(G)^r \to L^2(G)^s\bigr)$ its
Fuglede-Kadison determinant.

Denote by $A[i] \in M_{r,s}(\IQ[G/G_i])$ the matrix obtained from $A$ by applying
elementwise the ring homomorphism $\IQ G \to \IQ[G/G_i]$ induced by the projection $G \to
G/G_i$.  It induces a $\IC$-homomorphism of finite-dimensional complex Hilbert spaces
$r_{A[i]}^{(2)} \colon \IC[G/G_i]^r \to \IC[G/G_i]^s$.

Consider a $\IC$-homomorphism of finite-dimensional Hilbert spaces $f \colon V \to W$.  It
induces an endomorphism $f^*f \colon V \to V$. We have $\ker(f) = \ker(f^*f)$.  Denote by
$\ker(f)^{\perp}$ the orthogonal complement of $\ker(f)$. Then $f^*f$ induces an
automorphism $(f^*f)^{\perp} \colon \ker(f)^{\perp} \to \ker(f)^{\perp}$.  Define
\begin{eqnarray} {\det}'(f) & := & \sqrt{\det\bigl((f^*f)^{\perp}\bigr)}.
  \label{detprime(f)}
\end{eqnarray}
If $0 < \lambda_1 \le \lambda_2 \le \lambda_3 \le \cdots$ are the non-zero eigenvalues
(listed with multiplicity) of the positive operator $f^*f \colon V \to V$ , then
\[ {\det}'(f) = \prod_{j \ge 1} \sqrt{\lambda_j}.
\]
If $f$ is an isomorphism, then $\det'(f)$ reduces to $\sqrt{\det(f^*f)}$.

\begin{conjecture}[Approximation Conjecture for Fuglede-Kadison determinants]
\label{con:Approximation_conjecture_for_Fuglede-Kadison_determinants_with_finite_index}
  Consider a matrix $A \in M_{r,s}(\IQ G)$. Then we get
  \[
  \ln\bigl({\det}^{(2)}_{\caln(G)}(r_A^{(2)})\bigr) = \lim_{i \to \infty}
  \frac{\ln\bigl({\det}'(r_{A[i]}^{(2)})\bigr)}{[G:G_i]}.
  \]

\end{conjecture}

\begin{remark} \label{rem:complicated_case_for_determinants-necessary} If $r = s$ and $A
  \in M_{r,r}(\IQ G)$ is invertible, then the following equality is always true 
  \[
  \ln\bigl({\det}^{(2)}_{\caln(G)}(r_A^{(2)})\bigr) = \lim_{i \to \infty}
  \frac{\ln\bigl(\det(r_{A[i]}^{(2)})\bigr)}{[G:G_i]}.
  \]
  However, for applications to $L^2$-torsion we have to consider the case, where $r$ and
  $s$ may be different and the maps $(r_A^{(2)})^*r_A^{(2)}$ and $(r_{A[i]}^{(2)})^*r_{A[i]}^{(2)}$ may not be
  injective.
\end{remark}

\begin{remark}[$\IQ$ coefficients are necessary]
  \label{rem:IQ-coefficients_are_necessary}
  Conjecture~\ref{con:Approximation_conjecture_for_Fuglede-Kadison_determinants_with_finite_index} does not
  hold if one replaces $\IQ$ by $\IC$ by the following result appearing
  in~\cite[Example~13.69 on page~481]{Lueck(2002)}. There exists a sequence of integers $2
  \le n_1 < n_2 < n_3 < \cdots$ and a real number $s$ such that for $G = \IZ$ and $G_i =
  n_i \cdot \IZ$ and the $(1,1)$-matrix $A$ given by the element $z-\exp(2\pi i s)$ in
  $\IC[\IZ] = \IC[z,z^{-1}]$ we get for all $i \ge 1$
  \begin{eqnarray*}
    \ln\bigl({\det}^{(2)}_{\caln(G)}(r_A^{(2)})\bigr) & = & 0;
    \\
    \frac{\ln\bigl(\det(r_{A[i]}^{(2)})\bigr)}{[G:G_i]} & \le & - 1/2.
  \end{eqnarray*}
\end{remark}

\begin{remark}[Status of
  Conjecture~\ref{con:Approximation_conjecture_for_Fuglede-Kadison_determinants_with_finite_index}]
  \label{con:status_of_Approx_Conj_Fuglede-Kadison}
  Conjecture~\ref{con:Approximation_conjecture_for_Fuglede-Kadison_determinants_with_finite_index} has been
  proved for $G = \IZ$ by Schmidt~\cite{Schmidt(1995)}, see also~\cite[Lemma~13.53 on
  page~478]{Lueck(2002)}.  To the author's knowledge infinite virtually cyclic groups are the only infinite groups for
  which Conjecture~\ref{con:Approximation_conjecture_for_Fuglede-Kadison_determinants_with_finite_index} is
  known to be true.  One does know in general the inequality, see Theorem~\ref{the:inequality_det_det},
  \begin{eqnarray}
    \ln\bigl({\det}^{(2)}_{\caln(G)}(r_A^{(2)})\bigr)
    & \ge & 
    \limsup_{i \to \infty} \frac{\ln\bigl({\det}'(r_{A[i]}^{(2)})\bigr)}{[G:G_i]}.
    \label{ln(det)_ge_limsup}
  \end{eqnarray}

  For $G = \IZ^n$ the last inequality for the limit superior is known to be an equality by
  L\^e~\cite{Le(2013)}.  But nothing seems to be known beyond virtually finitely generated
  free abelian groups.
\end{remark}


\typeout{------------------------ Section 7: Torsion invariants --------------------}

\section{Torsion invariants}
\label{sec:torsion_invariants}

Before we start talking about torsion, we want to shed some light on the name. To the best of our knowledge
it comes from the visualization of a $3$-dimensional
lens space and its cells structure, where for different models the cells appear to be distorted in different ways.
The Reidemeister torsion, which was invented to classify lens spaces, measures this distorsion.


\subsection{$L^2$-torsion}
\label{sec:L2-torsion}

Let $D_*$ be a finite based free $\IZ$-chain complex, for instance the cellular chain
complex $C_*(Y) $ of a finite $CW$-complex $Y$.  The $\IC$-chain complex $D_* \otimes_{\IZ} \IC$ 
inherits from the $\IZ$-basis on $D_*$ and the standard Hilbert space structure on $\IC$ the
structure of a Hilbert space and the resulting $L^2$-chain complex is denoted by
$D_*^{(2)}$ with differentials $d_p^{(2)} := d_p \otimes_{\IZ} \id_{\IC}$. Define its \emph{$L^2$-torsion} by
\begin{eqnarray}
\rho^{(2)}(D_*^{(2)};\caln(\{1\})) := - \sum_{n \ge 1} (-1)^n \cdot \ln\bigl({\det}_{\caln(\{1\})}^{(2)}(d_n^{(2)})\bigr) \quad \in \IR.
\label{rho_caln(1)(2)(C_ast(2))_over_Z}
\end{eqnarray}
Notice that ${\det}_{\caln(\{1\})}^{(2)}(c_n^{(2)})$ is the same as $\det'(c_n^{(2)})$ which we have introduced
in~\eqref{detprime(f)}.

More generally, if $C_*$ is a finite based free $\IZ G$-chain complex, we obtain a 
finite Hilbert $\caln(G)$-chain complex $C_*^{(2)} := C_* \otimes_{\IZ G} L^2(G)$ and we define its 
\emph{$L^2$-torsion}
\begin{eqnarray}
\rho^{(2)}(C_*^{(2)};\caln(G)) := - \sum_{n \ge 1} (-1)^n \cdot \ln({\det}_{\caln(G)}^{(2)}\bigl(c_n^{(2)})\bigr)\quad \in \IR.
\label{rho_caln(1)(2)(C_ast(2))_over_ZG}
\end{eqnarray}

\begin{conjecture}[Approximation Conjecture for $L^2$-torsion]
\label{con:Approximation_conjecture_for_chain_complexes_with_finite_index}
Let $C_*$ be a finite based free $\IZ G$-chain complex. Denote by $C[i]_*$ the finite free $\IZ$-chain complex
given by $C[i]_* = C_* \otimes_{\IZ[G_i]} \IZ = C_* \otimes_{\IZ G} \IZ[G/G_i]$. 

Then we get
  \[
  \rho^{(2)}(C_*^{(2)};\caln(G))
  = 
  \lim_{i \to \infty}   \frac{\rho^{(2)}(C[i]_*^{(2)};\caln(\{1\}))}{[G:G_i]}.
  \]
\end{conjecture}

Obviously Conjecture~\ref{con:Approximation_conjecture_for_chain_complexes_with_finite_index}
is just the chain complex version of 
Conjecture~\ref{con:Approximation_conjecture_for_Fuglede-Kadison_determinants_with_finite_index}
and these two conjectures are equivalent for a given group $G$. 


\subsection{Analytic and topological $L^2$-torsion}
\label{sec:analytic_and_topological_L2-torsion}

Let $X$ be a closed Riemannian manifold. In the sequel we denote by $G \to \overline{X} \to X$ some $G$-covering.
Let $\rho_{\an}(X[i])$ be the analytic torsion in the sense of Ray and
Singer of the closed Riemannian manifold $X[i]$. Denote by $\rho^{(2)}_{\an}(\overline{X})$ the analytic
$L^2$-torsion of the Riemannian manifold $\overline{X}$ with isometric free cocompact $G$-action.

\begin{conjecture}[Approximation Conjecture for analytic torsion]
  \label{con:Approximation_for_analytic_torsion}
  Let $X$ be a closed Riemannian manifold. Then
  \[
  \rho^{(2)}_{\an}(\overline{X};\caln(G)) = \lim_{i \to \infty}
  \frac{\ln(\rho_{\an}(X[i]))}{[G:G_i]}.
  \]
\end{conjecture}

There are topological counterparts which we will denote by $\rho_{\topo}(X[i])$ and
$\rho^{(2)}_{\topo}(\overline{X})$ which agree with their analytic versions by results of
Cheeger~\cite{Cheeger(1979)} and M\"uller~\cite{Mueller(1978)} and
Burghelea-Friedlander-Kappeler-McDonald~\cite{Burghelea-Friedlander-Kappeler-McDonald(1996a)}.
So the conjecture above is equivalent to its topological counterpart.

\begin{conjecture}[Approximation Conjecture for topological torsion]
  \label{con:Approximation_for_topological_torsion}
  Let $X$ be a closed Riemannian manifold. Then
  \[
  \rho^{(2)}_{\topo}(\overline{X};\caln(G)) = \lim_{i \to \infty}
  \frac{\ln(\rho_{\topo}(X[i]))}{[G:G_i]}.
  \]
\end{conjecture}

\begin{remark}[Dependency on the triangulation and the Riemannian metric]
  \label{rem:Dependency_on_the_triangulation_and_the_Riemannian_metric}
  Let $X$ be a closed smooth manifold.  Fix a smooth triangulation. Since this induces a
  structure of a free finite $G$-$CW$-complex on $\overline{X}$, we get a $\IZ G$-basis
  for $C_*(\overline{X})$ and hence can consider
  $\rho^{(2)}(C_*^{(2)}(\overline{X});\caln(G))$. The cellular $\IZ G$-basis for
  $C_*(\overline{X})$ is not unique, but it is up to permutation of the basis elements and
  multiplying base elements with trivial units, i.e., elements of the shape $\pm g$ for
  $g \in G$. It turns out that $\rho^{(2)}(C_*^{(2)}(\overline{X});\caln(G))$ is independent
  of these choices after we have fixed a smooth triangulation of $X$.
  However, if we pass to a subdivision of the smooth triangulation of
  $X$, then $\rho^{(2)}(C_*^{(2)}(\overline{X});\caln(G))$ changes in general, 
unless $b_n^{(2)}(\overline{X};\caln(G))$ vanishes for all $n \ge 0$.

  Let $X$ be a closed smooth Riemannian manifold. Then
  $\rho^{(2)}_{\an}(\overline{X};\caln(G))$ and
  $\rho^{(2)}_{\topo}(\overline{X};\caln(G))$ are independent of the choice of
  smooth triangulation and hence depend only on the isometric diffeomorphism type of
  $X$. However, changing the Riemannian metric does in general change
  $\rho^{(2)}_{\an}(\overline{X};\caln(G))$ and
  $\rho^{(2)}_{\topo}(\overline{X};\caln(G))$. If we have
  $b_n^{(2)}(\overline{X};\caln(G)) = 0$ for all $n \ge 0$, then
  $\rho^{(2)}_{\an}(\overline{X};\caln(G))$ and
  $\rho^{(2)}_{\topo}(\overline{X};\caln(G))$ are independent of the Riemannian metric and
  depend only on the diffeomorphism type of $X$, actually, they depend only on the simple
  homotopy type of $X$. There is a lot of evidence that in this situation only the homotopy type of $X$ matters,
see for instance~\cite[Conjecture~3.94 on page~163]{Lueck(2002)} and the Determinant Conjecture~\ref{con:Determinant_Conjecture}.
\end{remark}

The next result is a special case of Theorem~\ref{the:comparing_analytic_and_chain_complexes}.

\begin{theorem}[Relating the Approximation Conjectures for Fuglede-Kadison determinant and
  torsion invariants]
  \label{the:comparing_torsion_and_FK-determinants}
  Suppose that $X$ is a closed Riemannian manifold such that
  $b_n^{(2)}(\overline{X},\caln(G))$ vanishes for all $n \ge 0$.  If $G$ satisfies
  Conjecture~\ref{con:Approximation_conjecture_for_Fuglede-Kadison_determinants_with_finite_index} for all
  matrices $A \in M_{r,s}(\IQ G)$ and all natural numbers $r,s$, then
  Conjecture~\ref{con:Approximation_for_analytic_torsion} and
  Conjecture~\ref{con:Approximation_for_topological_torsion} hold for $X$.
\end{theorem}

\begin{remark}[On the $L^2$-acyclicity assumption]
  \label{rem:On_the_L2-acyclicity_assumption_with_finite_index}
  Recall that in Theorem~\ref{the:comparing_torsion_and_FK-determinants} we require that
  $b_n^{(2)}(M;\caln(G)\bigr) = 0$ holds for $n \ge 0$.  This assumption is satisfied in
  many interesting cases.  It is possible that this assumption is not needed for
  Theorem~\ref{the:comparing_torsion_and_FK-determinants} to be true, but our proof does
  not work without it.
\end{remark}


\subsection{Integral torsion}
\label{sec:integral_torsion}

\begin{definition}[Integral torsion]\label{def:integral_torsion}
  Define for a finite $\IZ$-chain complex $D_*$ its \emph{integral torsion}
  \begin{eqnarray*}
    \rho^{\IZ}(D_*) 
    & := & 
    \sum_{n \ge 0} (-1)^n \cdot \ln\left(\bigl|\tors(H_{n}(D_*))\bigr|\right) 
    \quad \in \IR,
  \end{eqnarray*}
  where $\bigl|\tors(H_n(D_*))\bigr|$ is the order of the torsion subgroup of the finitely
  generated abelian group $H_n(D_*)$.

  Given a finite $CW$-complex $X$, define its \emph{integral torsion} $\rho^{\IZ}(X)$ by
  $\rho^{\IZ}(C_*(X))$, where $C_*(X)$ is its cellular $\IZ$-complex.
\end{definition}

\begin{remark}[Integral torsion and Milnor's torsion]
  \label{rem:Milnor_torsion}
  Let $C_*$ be a finite free $\IZ$-chain complex. Fix for each $n \ge 0$ a $\IZ$-basis for
  $C_n$ and for $H_n(C)/\tors(H_n(C))$. These induce $\IQ$-bases for $\IQ \otimes_{\IZ}
  C_n$ and $H_n\bigl(\IQ \otimes_{\IZ} C_*) \cong \IQ \otimes_{\IZ}
  \bigl(H_n(C)/\tors(H_n(C)\bigr)\bigr)$.  Then the torsion in the sense of
  Milnor~\cite[page~365]{Milnor(1966)} is $\rho^{\IZ}(C_*)$.
\end{remark}

The following two conjectures are motivated
by~\cite[Conjecture~1.3]{Bergeron-Venkatesh(2013)} and~\cite[Conjecture~11.3 on page~418
and Question~13.52 on page~478]{Lueck(2002)}.  They are true in special cases by
Theorem~\ref{the:Groups_containing_a_normal_infinite_nice_subgroups}.  The assumption that
$b_n^{(2)}(\overline{X};\caln(G))$ vanishes for all $n \ge 0$ ensures that the definition
of the topological $L^2$-torsion $\rho^{(2)}_{\topo}(\overline{X};\caln(G))$ makes sense
for $X$ also in the case of a connected finite $CW$-complex.

\begin{conjecture}[Approximation Conjecture for integral torsion]
  \label{con:Approximation_Conjecture_for_Milnor_torsion}
  Let $X$ be a finite connected $CW$-complex. Suppose that
  $b_n^{(2)}(\overline{X};\caln(G))$ vanishes for all $n \ge 0$. Then
  \[
  \rho^{(2)}_{\topo}(\overline{X};\caln(G)) = \lim_{i \to \infty}
  \frac{\rho^{\IZ}(X[i])}{[G:G_i]}.
  \]
\end{conjecture}

The chain complex version of Conjectures~\ref{con:Approximation_Conjecture_for_Milnor_torsion} is

\begin{conjecture}[Approximating Fuglede-Kadison determinants and $L^2$-torsion
  by homology]
  \label{con:Approximating_Fuglede-Kadison_determinants_and_L2-torsion_by_homology}\
 
  \begin{enumerate}

  \item \label{con:Approximating_Fuglede-Kadison_determinants_and_L2-torsion_by_homology:ZG-maps}
    Let $f \colon \IZ G^r \to \IZ G^r$ be a $\IZ G$-homomorphism such that
    $f^{(2)} \colon L^2(G)^r \to L^2(G)^r$ is a weak isomorphism of Hilbert
    $\caln(G)$-modules.  Let $f[i] := f \otimes_{\IZ G_i} \IZ = f \otimes_{\IZ G} \IZ[G/G_i] \colon \IZ[G/G_i]^r \to \IZ[G/G_i]^r$
   be the induced $\IZ$-homomorphism.
   Then
    \begin{eqnarray*}
      {\det}_{\caln(G)}^{(2)}(f^{(2)})
      & = & 
      \lim_{i \to \infty} \; \bigl|\tors\bigl(\coker(f[i])\bigr)\bigr|^{1/[G:G_i]};
    \end{eqnarray*}

  \item \label{con:Approximating_Fuglede-Kadison_determinants_and_L2-torsion_by_homology:chain_complexes}
    Let $C_*$ be a finite based free $\IZ G$-chain complex.  Suppose that
    $C_*^{(2)}$ is $L^2$-acyclic, i.e., $b_p^{(2)}\bigl(C_*^{(2)}\bigr) = 0$ for all $p \ge 0$. 
   Let $C[i]_* := C_* \otimes_{\IZ G_i} \IZ = C_* \otimes_{\IZ G} \IZ[G/G_i]$
   be the induced finite based free $\IZ$-chain complex. Then
    \begin{eqnarray*}
      \rho^{(2)}\bigl(C_*^{(2)}\bigr) 
      & = &
      \lim_{i \to \infty} \;\frac{\rho^{\IZ}(C[i]_*)}{[G:G_i]};
    \end{eqnarray*}

  \end{enumerate}
\end{conjecture}

In Conjecture~\ref{con:Approximation_Conjecture_for_Milnor_torsion}
and Conjecture~\ref{con:Approximating_Fuglede-Kadison_determinants_and_L2-torsion_by_homology}
it is necessary to demand that $f$ is a weak isomorphism and that $C_*$ and
$X$ are $L^2$-acyclic, otherwise there are counterexamples, see Remark~\ref{rem:condition_L2_acyclic_necessary}.

Here are some results about the conjecture above which will be proved in
Section~\ref{sec:Proof_of_Theorem_ref(the:about_approximating_by_IZ-torsion)}.

\begin{theorem}\label{the:about_approximating_by_IZ-torsion}\
 
  \begin{enumerate}

  \item \label{the:about_approximating_by_IZ-torsion:inequality} Let $f \colon
    \IZ G^r \to \IZ G^s$ be a $\IZ G$-homomorphism. Then
\begin{eqnarray*}
  \ln\bigl({\det}_{\caln(G)}^{(2)}(f^{(2)})\bigr)
  & \ge  & 
  \limsup_{i \to \infty} \frac{\bigl|\tors\bigl(\coker(f[i])\bigr)\big|}{[G:G_i]};
\end{eqnarray*}

\item \label{the:about_approximating_by_IZ-torsion:acyclic_in_each_degree}
  Suppose in the situation of
  assertion~\eqref{con:Approximating_Fuglede-Kadison_determinants_and_L2-torsion_by_homology:ZG-maps}
  of  Conjecture~\ref{con:Approximating_Fuglede-Kadison_determinants_and_L2-torsion_by_homology}
  that $f\otimes_{\IZ} \id_{\IQ} \colon \IQ[G]^r \to \IQ[G]^r$ is
  bijective. Then the conclusion appearing there is true.

  Suppose in the situation of
  assertion~\eqref{con:Approximating_Fuglede-Kadison_determinants_and_L2-torsion_by_homology:chain_complexes}
  of  Conjecture~\ref{con:Approximating_Fuglede-Kadison_determinants_and_L2-torsion_by_homology}
  that $H_p(C_*) \otimes_{\IZ} \IQ = 0$ for all $p \ge 0$. Then the
  conclusion appearing there is true;  

\item \label{the:about_approximating_by_IZ-torsion:G_is_Z} If $G$ is the
  infinite cyclic group $\IZ$, then
  Conjecture~\ref{con:Approximating_Fuglede-Kadison_determinants_and_L2-torsion_by_homology}
  is true.

\end{enumerate}

\end{theorem}

Assertion~\eqref{the:about_approximating_by_IZ-torsion:acyclic_in_each_degree} of 
Theorem~\ref{the:about_approximating_by_IZ-torsion} is generalized in
Lemma~\ref{lem:invariance_of_homological_conjecture_under_Q-homotopy_equivalence}.
Assertion~\eqref{the:about_approximating_by_IZ-torsion:G_is_Z} of 
Theorem~\ref{the:about_approximating_by_IZ-torsion} has already been 
proved by Bergeron-Venkatesh~\cite[Theorem~7.3]{Bergeron-Venkatesh(2013)}.
Applied to cyclic coverings of a knot complement this
reduces to a theorem of Silver-Williams~\cite[Theorem~2.1]{Silver-Williams(2002)}.
An extension of the results in this paper is given by
Raimbault~\cite{Raimbault(2012abelian)}.


\typeout{-------- Section 8: On the relation of $L^2$-torsion and integral torsion ------------------------}

\section{On the relation of $L^2$-torsion and integral torsion}
\label{sec:On_the_relation_of_L2-torsion_and_integral_torsion}

Let $C_*$ be a finite based free $\IZ$-chain complex, for instance the cellular chain
complex $C_*(X) $ of a finite $CW$-complex.  We have introduced in
Subsection~\ref{sec:L2-torsion} the $L^2$-chain complex $C_*^{(2)} = C_* \otimes_{\IZ} \IC $ 
with differentials $c_n^{(2)} := c_n \otimes_{\IZ} \id_{\IC}$  and its
$L^2$-torsion $\rho^{(2)}(C_*^{(2)};\caln(\{1\}))$.

Let $H_n^{(2)}(C_*^{(2)})$ be the
$L^2$-homology of $C_*^{(2)}$ with respect to the von Neumann algebra $\caln(\{1\}) = \IC$. The
underlying complex vector space is the homology $H_n(C_* \otimes_{\IZ} \IC$) of
$C_*\otimes_{\IZ} \IC$, but it comes now with the structure of a Hilbert
space.  For the reader's convenience we recall this Hilbert space
structure.  Let
\[\Delta_n^{(2)} = \bigl(c_n^{(2)}\bigr)^*\circ c_n^{(2)} + c_{n+1}^{(2)} \circ
\bigl(c_{n+1}^{(2)}\bigr)^* \colon C_n^{(2)} \to C_n^{(2)}\] be the associated
Laplacian.  Equip $\ker(\Delta_n^{(2)}) \subseteq C_n^{(2)}$ with the induced
Hilbert space structure.  Equip $H_n^{(2)}(C_*^{(2)})$ with the Hilbert space
structure for which the obvious $\IC$-isomorphism $\ker(\Delta_n^{(2)}) \to
H_n^{(2)}(C_*^{(2)})$ becomes an isometric isomorphism. This is the same as the 
Hilbert subquotient structure with respect to the inclusion $\ker\bigl(c_n^{(2)}\bigr) \to C_n^{(2)}$ and 
the projection $\ker\bigl(c_n^{(2)}\bigr) \to H_n^{(2)}(C_*^{(2)})$.

\begin{notation} \label{not_M_f}
  If $M$ is a finitely generated abelian group, define
  \begin{eqnarray*}
    M_f & := & M/ \tors(M).
  \end{eqnarray*}
\end{notation}

We have the canonical $\IC$-isomorphism
\begin{eqnarray}
\alpha_n \colon \bigl(H_n(C_*)_f\bigr)^{(2)} := \bigl(H_n(C_*)/\tors(H_n(C_*)\bigr) \otimes_{\IZ} \IC 
& \xrightarrow{\cong}  & 
H_n^{(2)}(C_*^{(2)}).
\label{iso_alpha}
\end{eqnarray}
Choose a $\IZ$-basis on $H_n(C_*)_f$. This  and  the
standard Hilbert space structure on $\IC$ induces a Hilbert space structure on
$\bigl(H_n(C_*)_f\bigr)^{(2)}$. Now we can consider the logarithm of the Fuglede-Kadison 
determinant
\begin{eqnarray}
R_n(C_*)
& = & 
\ln\left({\det}_{\caln(\{1\})}\bigl(\alpha_n \colon \bigl(H_n(C_*)_f)^{(2)}
\to  H_n^{(2)}(C_*^{(2)})\bigr)\right),
\label{nth-regulator}
\end{eqnarray}
which is sometimes called the \emph{$n$th regulator}.
It is independent of the choice of
the $\IZ$-basis of $H_n(C_*)_f$, since the absolute value of the determinant of an invertible
matrix over $\IZ$ is always $1$.  If $\{b_1, b_2, \ldots ,b_r\}$ is an integral basis
of $H_n(C_*)_f$ and we equip $H_n(C_*)_f \otimes_{\IZ} \IC$ with an inner product
$\langle -,-\rangle$ for which the map $\alpha_n$ of~\eqref{iso_alpha} becomes an isometry, then
\[
R_n(C_*)
= 
\frac{\ln\big({\det}_{\IC}(B)\bigr)}{2},
\]
where $B$ is the Gram-Schmidt matrix $\bigl(\langle b_i,b_j\rangle\bigr)_{i,j}$.

The next result is proved for instance in~\cite[Lemma~2.3]{Lueck(2013l2approxfib)}.

\begin{lemma} \label{lem:rho(2)-rhoZ}
Let $C_*$ be a finite based free $\IZ$-chain complex. Then
\begin{eqnarray*}
\rho^{\IZ}(C_*)  - \rho^{(2)}\bigl(C_*^{(2)};\caln(\{1\})\bigr) 
& = &
\sum_{n \ge 0} (-1)^n \cdot R_n(C_*).
\end{eqnarray*}
\end{lemma}

\begin{remark}[Comparing conjectures for $L^2$-torsion and integral torsion] 
\label{rem:Comparing_conjectures_for_L2-torsion_and_integral_torsion}
 Consider the following three statements:

  \begin{enumerate}

   \item \label{rem:Comparing_conjectures_for_L2-torsion_and_integral_torsion:L(2)}
  Every finite based free $\IZ G$-chain complex $C_*$ with $b_n^{(2)}\bigl(C_*^{(2)}\bigr) = 0$ for all $n \ge 0$
satisfies
\[
\rho^{(2)}\bigl(C_*^{(2)}\bigr) = \lim_{i \to \infty} \;\frac{\rho^{(2)}\bigl(C[i]_*^{(2)}\bigr)}{[G:G_i]};
\]

  \item\label {rem:Comparing_conjectures_for_L2-torsion_and_integral_torsion:Z}
Every finite based free $\IZ G$-chain complex $C_*$ with $b_n^{(2)}\bigl(C_*^{(2)}\bigr) = 0$ for all $n \ge 0$
satisfies assertion~\eqref{con:Approximating_Fuglede-Kadison_determinants_and_L2-torsion_by_homology:chain_complexes} 
  of Conjecture~\ref{con:Approximating_Fuglede-Kadison_determinants_and_L2-torsion_by_homology}, i.e.,
  \begin{eqnarray*}
      \rho^{(2)}\bigl(C_*^{(2)}\bigr) 
      & = &
      \lim_{i \to \infty} \;\frac{\rho^{\IZ}(C[i]_*)}{[G:G_i]};
    \end{eqnarray*}

  \item \label{rem:Comparing_conjectures_for_L2-torsion_and_integral_torsion:alpha}
    Every finite based free $\IZ G$-chain complex $C_*$ with $b_n^{(2)}\bigl(C_*^{(2)}\bigr) = 0$ for all $n \ge 0$ satisfies
    \begin{eqnarray*}
      \lim_{i \to \infty}\; \sum_{n \ge 0} (-1)^n \cdot  
      \frac{R_n(C[i]_*)}{[G:G_i]}
      & = & 0,
    \end{eqnarray*}
\end{enumerate}
By Lemma~\ref{lem:rho(2)-rhoZ} all of these statements are true if any  two of them hold.
\end{remark}

\begin{lemma} \label{lem:R_0_is_okay}
Let $X$ be an oriented closed smooth manifold of dimension $d$.
Fix a smooth triangulation.  Let $s_n$ be the number of $n$-simplices of
the triangulation of $X$. Then we get
\begin{eqnarray*}
R_d(C_*(X[i])) & = & \frac{\ln([G:G_i] \cdot s_d)}{2};
\\
R_0(C_*(X[i])) & = & -\frac{\ln([G:G_i] \cdot s_0)}{2};
\\
\lim_{i \to \infty} \frac{R_n(C_*(X[i]))}{[G:G_i]}  &= & 0\quad \text{for}\; n = 0,d.
\end{eqnarray*}
\end{lemma}
\begin{proof}
The  fundamental class $[X[i]]$ is a generator of the infinite cyclic group $H_d(X[i];\IZ)$, and
is represented by the cycle $\sigma[i]_d$ in $C_d(X[i])$ given by the
sum over all $d$-dimensional simplices. The number of $d$-simplices  in $X[i]$ is $[G:G_i] \cdot s_d$.
If we consider $\sigma[i]_d$ as an element in $C_d^{(2)}(X[i])$, it belongs to the kernel of
$\Delta[i]_d^{(2)}$ and has norm $\sqrt{[G:G_i] \cdot s_d}$.
Hence $R_d(C_*(X[i]))$, which is the logarithm of the norm of $\sigma[i]_d$ considered as an element
in $C_d^{(2)}(X[i])$, is $\frac{\ln([G:G_i] \cdot s_d)}{2}$.

Consider the element $\sigma[i]_0 \in C_0(X[i])$ given by the sum of the $0$-simplices
of $X[i]$.  The number of $0$-simplices in $X[i]$ is $[G:G_i] \cdot s_0$.
The element $\sigma[i]_0$ considered as an element in $C_0^{(2)}(X[i])$ has norm $\sqrt{[G:G_i] \cdot s_0}$
and lies in the kernel of $\bigl(c[i]_1^{(2)}\bigr)^*$ and hence in the kernel of $\Delta[i]_0^{(2)}$
since it is orthogonal to any element of the shape $e_1$ - $e_0$ for $0$-simplices $e_0$ and $e_1$
and hence to the image of $c_1^{(2)} \colon C_1^{(2)}(X[i]) \to C_0^{(2)}(X[i])$.
The augmentation map $C_0(X[i]) \to \IZ$ sending  a $0$-simplex to $1$
induces an isomorphism $H_0(C_*(X[i])) \xrightarrow{\cong} \IZ$.
This shows that $\sigma[i]_0$  represents $[G:G_i] \cdot s_0$ times the generator of $H_0(X[i];\IZ)$. Hence $R_0(C_*(X[i]))$,
which is the logarithm of the norm of $\frac{\sigma[i]_0}{[G:G_i] \cdot s_0}$, is $-\frac{\ln([G:G_i] \cdot s_0)}{2}$. 

In particular we get $\lim_{i \to \infty} \frac{R_n(C_*(X[i]))}{[G:G_i]} = 0$ for $n = 0,d$.
\end{proof}


\typeout{-------- Section 9: Elementary example about $L^2$-torsion and integral torsion  ----------------}

\section{Elementary example about $L^2$-torsion and integral torsion}
  \label{sec:elementary_example_about_L2-torsion_and_integral_torsion}

Consider integers $a$, $b$, $k$, $l$, and  $g \ge 1$, such that $(a,b) = (1)$ and $(k,l) = (1)$.

  Consider the following finite based free $\IZ$-chain complex $C_*$ which is
  concentrated in dimensions $0$, $1$, $2$ and $3$ and given there by
\[
0\cdots \to 0 \to \IZ \xrightarrow{c_3 = \left(\begin{matrix} -l \\ k\end{matrix}\right)} \IZ^2
\xrightarrow{c_2 = \left(\begin{matrix} gka & gla\\ gkb & glb\end{matrix}\right)}
\IZ^2 \xrightarrow{c_1 = \left(\begin{matrix} -b & a \end{matrix}\right)} \IZ
\to 0 \to \cdots.
\]
Notice that any matrix homomorphism $c_2 \colon \IZ^2 \to \IZ^2$ whose kernel has rank one is
of the shape above.

  One easily checks that
\begin{eqnarray*}
  \ker(c_3) & = & \{0\};
  \\
  \im(c_3) & = & \{ n \cdot \begin{pmatrix}-l \\ k \end{pmatrix}  \mid n \in \IZ\};
  \\
  \ker(c_2) & = & \{ n \cdot \begin{pmatrix}-l \\ k \end{pmatrix} \mid n \in \IZ\};
  \\
  \im(c_2) & = & \{n \cdot  \begin{pmatrix} ga \\ gb \end{pmatrix} \mid n \in \IZ\};
  \\
  \ker(c_1) & = & \{n \cdot \begin{pmatrix} a  \\ b \end{pmatrix}\mid n \in \IZ\};
  \\
  \im(c_1) & = & \IZ.
\end{eqnarray*}
This implies
\[ H_i(C_*) = 
\begin{cases}
  \IZ/g & i = 1;
  \\
  \{0\} & i \not= 1.
\end{cases}
\]
We conclude from~\cite[Lemma~3.15~(4) on page~129]{Lueck(2002)}
\begin{multline*}
  \ln\bigl({\det}^{(2)}_{\caln(\{1\})}\bigl(c_3^{(2)}\bigr)\bigr) = \frac{1}{2} \cdot
  \ln\bigl({\det}^{(2)}_{\caln(\{1\})}\bigl((c_3^{(2)})^* \circ c_3^{(2)}\bigr)\bigr)
  \\
  = \frac{1}{2} \cdot \ln\bigl({\det}^{(2)}_{\caln(\{1\})}\bigl((k^2 +l^2) \colon \IC \to
  \IC\bigr)) = \frac{\ln(k^2 + l^2)}{2}.
\end{multline*}
Analogously one shows
$$\ln\bigl({\det}^{(2)}_{\caln(\{1\})}\bigl(c_1^{(2)}\bigr)\bigr)  =  \frac{\ln(a^2 + b^2)}{2}.$$
The kernel of $c_2^{(2)}$ is the subvector space of $\IC^2$ generated by
$\frac{1}{\sqrt{k^2+l^2}} \cdot  \begin{pmatrix} -l \\ k \end{pmatrix} $ and the image of $c_2^{(2)}$ is the
subvector space of $\IC^2$ generated by $\frac{1}{\sqrt{a^2+b^2}} \cdot \begin{pmatrix} a \\ b \end{pmatrix} $. 
Hence the orthogonal complement $\ker\bigl(c_2^{(2)}\bigr)^{\perp}$ of the 
kernel of $c_2^{(2)}$ is the subvector space of $\IC^2$ generated by
$\frac{1}{\sqrt{k^2+l^2}} \cdot \begin{pmatrix} k \\ l \end{pmatrix} $. Since 
$\frac{1}{\sqrt{a^2+b^2}} \cdot \begin{pmatrix} a \\ b \end{pmatrix}$
and $\frac{1}{\sqrt{k^2+l^2}} \cdot  \begin{pmatrix} k \\ l \end{pmatrix} $ have norm $1$ and
\[
c_2^{(2)}\left(\frac{1}{\sqrt{k^2+l^2}} \cdot  \begin{pmatrix} k \\ l \end{pmatrix} \right)
= \left(g \cdot \sqrt{k^2+l^2} \cdot \sqrt{a^2 + b^2}\right) 
\cdot \left(\frac{1}{\sqrt{a^2+b^2}} \cdot  \begin{pmatrix} a \\ b \end{pmatrix} \right),
\]
we conclude from~\cite[Lemma~3.15~(3) on page~129]{Lueck(2002)}
\[
\ln\bigl({\det}^{(2)}_{\caln(\{1\})}\bigl(c_2^{(2)}\bigr)\bigr)  
=  \ln(g) +  \frac{\ln(a^2 + b^2) + \ln(k^2 +l^2)}{2}.
\]
Notice that Lemma~\ref{lem:rho(2)-rhoZ} predicts $\rho^{(2)}\bigl(C_*^{(2)}\bigr) = \rho^{\IZ}(C_*)$
which is consistent with the direct computation
\begin{eqnarray*}
  \rho^{(2)}\bigl(C_*^{(2)}\bigr)
  & = & 
  - \ln\bigl({\det}^{(2)}_{\caln(\{1\})}\bigl(c_3^{(2)}\bigr)\bigr) + 
  \ln\bigl({\det}^{(2)}_{\caln(\{1\})}\bigl(c_2^{(2)}\bigr)\bigr) 
  -\ln\bigl({\det}^{(2)}_{\caln(\{1\})}\bigl(c_1^{(2)}\bigr)\bigr)
  \\
  & = &
   \ln(g)
  \\
  & = & 
  \ln\bigl(\tors\bigl(H_1(C_*)\bigr)\bigr)
  \\
  & = & 
  \rho^{\IZ}(C_*).
\end{eqnarray*}

We also compute  the combinatorial Laplace operators of $C_*$. We get
for their matrices
\begin{eqnarray*}
\Delta_3 & = & (k^2 +l^2);
\\
\Delta_2 & = & \left(\begin{matrix}
g^2k^2a^2 + g^2k^2b^2  + l^2 & g^2kla^2 + g^2klb^2 - kl
\\
g^2kla^2 + g^2klb^2 - kl & g^2l^2a^2 + g^2l^2b^2  + k^2
\end{matrix}\right);
\\
\Delta_1 & = & \left(\begin{matrix}
g^2k^2a^2 + g^2l^2a^2 + b^2  & g^2k^2ab + g^2l^2ab - ab
\\
g^2k^2ab + g^2l^2ab - ab  & g^2k^2b^2 + g^2l^2b^2 + a^2
\end{matrix}\right);
\\
\Delta_0 & = & (a^2 +b^2).
\end{eqnarray*}
This implies
$${\det}_{\IZ}(\Delta_i) = {\det}_{\caln(\{1\})}^{(2)}\bigl(\Delta_i^{(2)}\bigr)
\; = \;
\begin{cases}
k^2+l^2 & i = 3;
\\
(a^2 + b^2) \cdot g^2 \cdot (k^2+l^2)^2 & i  = 2;
\\
(a^2 +b^2)^2 \cdot g^2 \cdot (k^2 +l^2) & i = 1;
\\
(a^2 + b^2) & i = 0.
\end{cases}$$
This is consistent with the formula
\[\rho^{(2)}\bigl(C_*^{(2)};\caln(\{1\})\bigr)
 = - \frac{1}{2} \cdot \sum_{i \ge 0} (-1)^i \cdot i \cdot \ln\bigl({\det}^{(2)}_{\caln(\{1\})}\bigl(\Delta_i^{(2)}\bigr)\bigr).
\]

\begin{remark}[No relationship between the differentials and homology in each degree]
 \label{no_relation_between_Laplace_and_homology}
We see that there is no relationship between 
$\ln \bigl({\det}_{\caln(\{1\})}\bigl(\Delta_i^{(2)}\bigr)\bigr) $ and $\ln\bigl(\tors\bigl(H_i(C_*)\bigr)\bigr)$
or between $\ln \bigl({\det}_{\caln(\{1\}}^{(2)}\bigl(c_i^{(2)}\bigr)\bigr)$ and $\ln\bigl(\tors\bigl(H_i(C_*)\bigr)\bigr)$
for each individual $i \in \IZ$, there is only a relationship after taking the alternating sum over $i \ge 0$.

This shows that a potential proof of Conjecture~\ref{con:Approximating_Fuglede-Kadison_determinants_and_L2-torsion_by_homology}
will require more input than one would expect
for a potential proof of 
Conjecture~\ref{con:Approximation_conjecture_for_chain_complexes_with_finite_index},
  Conjecture~\ref{con:Approximation_for_analytic_torsion}, or
   Conjecture~\ref{con:Approximation_for_topological_torsion}.
\end{remark}

\begin{remark}[$L^2$-acyclicity is necessary for the homological version]
\label{rem:condition_L2_acyclic_necessary}
  This example can also be used to show that the condition of $L^2$-acyclicity appearing in
  Conjecture~\ref{con:Approximating_Fuglede-Kadison_determinants_and_L2-torsion_by_homology}
  is necessary.  This is not a surprise since
  $\rho^{\IZ}(C[i]_*)$ depends only on the $\IZ$-chain homotopy type of
  $C[i]_*$ which is not true for $\rho^{(2)}(C_*^{(2)})$ unless $C_*^{(2)}$ is
  $L^2$-acyclic.

  Consider the $1$-dimensional $\IZ$-chain complex chain complex $D_*$
  whose only non-trivial differential $d_1$ is the differential $c_2$ in the chain complex $C_*$ above.
  Let $E_* := D_* \otimes_{\IZ} \IZ[\IZ]$. Put $E_*[n] = E_*\otimes_{\IZ[\IZ]} \IZ[\IZ/n]$.
  Then $E_*[n] = D_* \otimes_{\IZ} \IZ[\IZ/n]$. We conclude from
  the computations above 
  and~\cite[Theorem~3.14~(5) and~(6) on  page~128]{Lueck(2002)}
  \begin{eqnarray*}
    \rho^{(2)}\bigl(E_*^{(2)};\caln(\IZ)\bigr) 
    & = & 
    \ln(g) +  \frac{\ln(a^2 + b^2) + \ln(k^2 +l^2)}{2};
    \\
    \frac{\rho^{(2)}\bigl(E[n]_*^{(2)};\caln(\{1\})\bigr)}{n}
    & = & 
    \ln(g) + \frac{\ln(a^2 + b^2) + \ln(k^2 +l^2)}{2};
    \\
    \frac{\rho^{\IZ}(E[n]_*)}{n} 
    & = & 
    \ln(g).
  \end{eqnarray*}
  Hence we have
  \begin{eqnarray*}
    \rho^{(2)}\bigl(E_*^{(2)};\caln(\IZ)\bigr) 
    & = & 
    \lim_{n \to \infty}    \frac{\rho^{(2)}\bigl(E[n]_*^{(2)};\caln(\{1\})\bigr)}{n} 
  \end{eqnarray*}
  but 
   \begin{eqnarray*}
    \rho^{(2)}\bigl(E_*^{(2)};\caln(\IZ)\bigr) 
    & \not= & 
    \lim_{n \to \infty}    
    \frac{\rho^{\IZ}(E[n]_*)}{n}.
  \end{eqnarray*}
   Notice that the condition of $L^2$-acyclicity is not demanded in 
  Conjecture~\ref{con:Approximation_conjecture_for_chain_complexes_with_finite_index},
  Conjecture~\ref{con:Approximation_for_analytic_torsion},
   Conjecture~\ref{con:Approximation_for_topological_torsion},   
   Conjecture~\ref{con:Approximation_conjecture_for_L2-torsion_of_chain_complexes},
   and Conjecture~\ref{con:Approximation_conjecture_for_analytic_L2-torsion}.
\end{remark}


\typeout{------------------------ Section 10: Aspherical manifolds --------------------}

\section{Aspherical manifolds}
\label{sec:Aspherical_manifolds}

The following conjecture is in our view the most advanced and interesting one.  It
combines Conjecture~\ref{con:Approximation_Conjecture_for_Milnor_torsion}, that one can
approximate $L^2$-torsion by integral torsion in the $L^2$-acyclic case, with the
conjecture that for closed aspherical manifolds $X$ the $L^2$-cohomology of $\widetilde{X}$,
and asymptotically the homology of $X[i]$ are concentrated in the middle
dimension.

\begin{conjecture}[Homological growth and $L^2$-torsion for aspherical closed manifolds]
  \label{con:Homological_growth_and_L2-torsion_for_aspherical_manifolds}
  Let $M$ be an aspherical closed manifold of dimension $d$ and fundamental group $G =  \pi_1(M)$. 
  Let $\widetilde{M}$ be its universal covering.  Then

  \begin{enumerate}

  \item \label{con:Homological_growth_and_L2-torsion_for_aspherical_manifolds:Betti} 
    For any natural number $n$ with $2n \not= d$ we get 
    \[
    b_n^{(2)}(\widetilde{M}) = 0.
    \]
    \noindent
    If $d = 2n$, we have
    \[
    (-1)^{n} \cdot \chi(M)  = b_n^{(2)}(\widetilde{M}) \ge 0.
    \]
     \noindent
    If  $d = 2n$ and $M$ carries a Riemannian metric of negative sectional curvature, then
    \[
    (-1)^n \cdot \chi(M) = b_n^{(2)}(\widetilde{M}) > 0;
    \]

   \item \label{con:Homological_growth_and_L2-torsion_for_aspherical_manifolds:Betti_and_limit} 
    Let $(G_i)_{i \ge 0}$ be  any chain of normal subgroups $G_i \subseteq G$ of finite index $[G:G_i]$
    and trivial intersection $\bigcap_{i \ge 0} G_i = \{1\}$. Put $M[i] = G_i \backslash \widetilde{M}$.
   
    Then we get for any natural number $n$ and any field $F$
        \[
    b_n^{(2)}(\widetilde{M}) = \lim_{i \to \infty} \frac{b_n(M[i];F)}{[G:G_i]} = \lim_{i \to \infty} \frac{d\bigl(H_n(M[i];\IZ)\bigr)}{[G:G_i]};
    \]
    
    and for $n = 1$
    \begin{multline*}
    \quad \quad \quad \quad \quad \quad 
    b_1^{(2)}(\widetilde{M}) = \lim_{i \to \infty} \frac{b_1(M[i];F)}{[G:G_i]} = \lim_{i \to \infty} \frac{d\bigl(G_i/[G_i,G_i]\bigr)}{[G:G_i]} 
    \\
    = RG(G,(G_i)_{i \ge 0}) 
    = \begin{cases} 0 & d \not= 2; \\ -\chi(M) & d = 2; \end{cases}
  \end{multline*}

   \item \label{con:Homological_growth_and_L2-torsion_for_aspherical_manifolds:truncated_Euler_characteristic} 
   We get for $m \ge 0$

    \[
    \lim_{i \to \infty} \frac{\chi_m^{\trun}(M[i])}{[G:G_i]} 
    = \begin{cases}
    \chi(M) & \text{if} \; d \;\text{is even and}\; 2m \ge d;
    \\
    0 & \text{otherwise};
   \end{cases}
    \]

   \item \label{con:Homological_growth_and_L2-torsion_for_aspherical_manifolds:tors_parity} 
    If $d = 2n+1$ is odd, we have
    \[
     (-1)^n \cdot \rho^{(2)}_{\an}\bigl(\widetilde{M}\bigr) \ge 0;
    \]
    If $d = 2n+1$ is odd and $M$ carries a Riemannian metric with negative sectional curvature, we have
    \[
     (-1)^n \cdot \rho^{(2)}_{\an}\bigl(\widetilde{M}\bigr) > 0;
    \]

  \item \label{con:Homological_growth_and_L2-torsion_for_aspherical_manifolds:tors} 
    Let $(G_i)_{i \ge 0}$ be any chain of normal subgroups $G_i \subseteq G$ of finite index $[G:G_i]$
    and trivial intersection $\bigcap_{i \ge 0} G_i = \{1\}$. Put $M[i] = G_i \backslash \widetilde{M}$.
 
   Then we get for  any natural number $n$ with $2n +1 \not= d$
    \[
    \lim_{i \to \infty} \;\frac{\ln\big(\bigl|\tors\bigl(H_n(M[i])\bigr)\bigr|\bigr)}{[G:G_i]} = 0,
    \]
    and we get in the case $d = 2n+1$
    \[
    \lim_{i \to \infty} \;\frac{\ln\big(\bigl|\tors\bigl(H_n(M[i])\bigr)\bigr|\bigr)}{[G:G_i]} = (-1)^n \cdot
    \rho^{(2)}_{\an}\bigl(\widetilde{M}\bigr) \ge 0.
    \]
  \end{enumerate}
\end{conjecture}

Notice that in assertions~\eqref{con:Homological_growth_and_L2-torsion_for_aspherical_manifolds:Betti} 
and~\eqref{con:Homological_growth_and_L2-torsion_for_aspherical_manifolds:tors_parity} we are not demanding
that $G = \pi_1(M)$ is residually finite. This assumption only enters in 
assertions~\eqref{con:Homological_growth_and_L2-torsion_for_aspherical_manifolds:Betti_and_limit},%
\eqref{con:Homological_growth_and_L2-torsion_for_aspherical_manifolds:truncated_Euler_characteristic}, 
and~\eqref{con:Homological_growth_and_L2-torsion_for_aspherical_manifolds:tors}, where the chain $(G_i)_{i \ge 0}$
occurs. 

\begin{remark}[Rank growth versus torsion growth]
  \label{rem:Rank_growth_versus_torsion_growth}
  Let us summarize what
  Conjecture~\ref{con:Homological_growth_and_L2-torsion_for_aspherical_manifolds} means
  for an aspherical closed manifold $M$ of dimension $d$. It predicts that the rank of the
  singular homology grows in dimension $m$ sublinearly if $2m \not= d$, and grows linearly
  if $d = 2m$ and $M$ carries a Riemannian metric of negative sectional curvature. It also
  predicts that the cardinality of the torsion of the singular homology grows in dimension
  $m$ grows subexponentially if $2m +1 \not = d$ and grows exponentially if $d = 2m+1$ and
  $M$ carries a Riemannian metric of negative sectional curvature. Roughly speaking, the
  free part of the singular homology is asymptotically concentrated in dimension $m$ if $d
  = 2m$ and the torsion part is asymptotically concentrated in dimension $m$ if $d =
  2m+1$. A vague explanation for this phenomenon could be that Poincar\'e duality links
  the rank in dimensions $m$ and $d - m$, whereas the torsion is linked in dimensions $m$
  and $d-1-m$, and there must be some reason that except in the middle dimension the growth
  of the rank and the growth of the cardinality of the torsion block one another in dual dimensions.
\end{remark}

\begin{remark}[Finite Poincar\'e complexes] \label{rem:finite_Poincare_complexes}
One may replace in the formulation of Conjecture~\ref{con:Homological_growth_and_L2-torsion_for_aspherical_manifolds}
the aspherical closed manifold $M$ by an aspherical finite Poincar\'e complex.
In the formulation of the part of 
assertion~\eqref{con:Homological_growth_and_L2-torsion_for_aspherical_manifolds:Betti},
where negative sectional curvature is required, one has to add an assumption on $\pi_1(X)$, for instance that 
$\pi_1(X)$  is a CAT(-1)-group.
\end{remark}

Assertion~\eqref{con:Homological_growth_and_L2-torsion_for_aspherical_manifolds:Betti} of
Conjecture~\ref{con:Homological_growth_and_L2-torsion_for_aspherical_manifolds} in the case 
that $M$ carries a Riemannian metric with non-positive sectional curvature is the Singer Conjecture.
The Singer Conjecture and also the related Hopf Conjecture are discussed in detail in~\cite[Section~11]{Lueck(2002)}.

Assertion~\eqref{con:Homological_growth_and_L2-torsion_for_aspherical_manifolds:Betti_and_limit}  is closely related
to 
Conjecture~\ref{con:Approximation_in_zero_and_prime_characteristic}, 
Conjecture~\ref{con:Growth_of_number_of_generators_of_the_homology} and 
Question~\ref{que:Rank_gradient_cost_first_L2_Betti_number_and_approximation}.

Assertion~\eqref{con:Homological_growth_and_L2-torsion_for_aspherical_manifolds:truncated_Euler_characteristic}
is connected to Question~\ref{que:asymptotic_Morse_equality_general}.

The parity condition about the $L^2$-torsion appearing in 
assertion~\eqref{con:Homological_growth_and_L2-torsion_for_aspherical_manifolds:tors_parity}  of
Conjecture~\ref{con:Homological_growth_and_L2-torsion_for_aspherical_manifolds} is already
considered in~\cite[Conjecture~11.3 on page~418]{Lueck(2002)}.

Assertion~\eqref{con:Homological_growth_and_L2-torsion_for_aspherical_manifolds:tors} appearing in
Conjecture~\ref{con:Homological_growth_and_L2-torsion_for_aspherical_manifolds} in the case
that $M$ is a locally symmetric space, is discussed in Bergeron-Venkatesh~\cite{Bergeron-Venkatesh(2013)}, 
where also twisting with a finite-dimensional integral representation is considered.

Some evidence for Conjecture~\ref{con:Homological_growth_and_L2-torsion_for_aspherical_manifolds}
comes from the following result of
L\"uck~\cite[Corollary~1.13]{Lueck(2013l2approxfib)}.

\begin{theorem}
  \label{the:Groups_containing_a_normal_infinite_nice_subgroups}
  Let $M$ be an aspherical closed manifold with fundamental group $G = \pi_1(M)$.  Suppose
  that $M$ carries a non-trivial $S^1$-action or suppose that $G$ contains a non-trivial
  elementary amenable normal subgroup. Then we get for all $n \ge 0$ and fields $F$
  \begin{eqnarray*}
    \lim_{i \to \infty} \frac{b_n(M[i];F)}{[G:G_i]}  
    & = & 
    0;
    \\
      \lim_{i \to \infty} \frac{d\bigl(H_n(M[i];\IZ)\bigr)}{[G:G_i]}  
    & = & 
    0;
    \\
    \lim_{i \to \infty} \;\frac{\ln\big(\bigl|\tors\bigl(H_n(M[i])\bigr)\bigr|\bigr)}{[G:G_i]}
    & = & 
    0;
    \\
    \lim_{i\to \infty} \frac{\rho_{\an}\bigl(M[i];\caln(\{1\})\bigr)}{[G:G_i]}
    & = & 
    0;
    \\
    \lim_{i \to \infty} \frac{\rho^{\IZ}\bigl(M[i]\bigr)}{[G:G_i]}
    & = & 
    0;
    \\
    b_n^{(2)}(\widetilde{M}) 
    & = & 
    0;
    \\
    \rho^{(2)}_{\an}(\widetilde{M}) 
    & = & 
    0.
  \end{eqnarray*}
  In particular Conjecture~\ref{con:Approximation_in_zero_and_prime_characteristic},
  Conjecture~\ref{con:Growth_of_number_of_generators_of_the_homology},
  Conjecture~\ref{con:Approximation_for_analytic_torsion},
  Conjecture~\ref{con:Approximation_for_topological_torsion},
  Conjecture~\ref{con:Approximation_Conjecture_for_Milnor_torsion} and
  Conjecture~\ref{con:Homological_growth_and_L2-torsion_for_aspherical_manifolds} 
   with the exception of assertion~\eqref{con:Homological_growth_and_L2-torsion_for_aspherical_manifolds:truncated_Euler_characteristic}
   are known to be true  for $G = \pi_1(M)$ and $X =M$.
\end{theorem}

Estimates of the growth of the  torsion in the homology in terms of the volume of the underlying manifold 
and examples of aspherical manifolds, where this growth is subexponential,
are given in~\cite{Sauer(2016)}.

Sometimes one can express for certain  classes of closed Riemannian manifolds $M$
the $L^2$-torsion of the universal covering $\widetilde{M}$ by the volume
\[
\rho_{\an}^{(2)}(\widetilde{M}) = C_{\dim(M)}  \cdot \vol(M),
\]
where $C_{\dim(M)} \in \IR$ is a dimension constant depending only on the class but not
on the specific $M$.  This follows from the Proportionality Principle due to Gromov, see
for instance~\cite[Theorem~1.183 on page~201]{Lueck(2002)}.  Typical examples are
locally symmetric spaces of non-compact type, for instance hyperbolic manifolds, 
see~\cite[Theorem~5.12 on page~228]{Lueck(2002)}.  Since $\rho_{\an}^{(2)}(\widetilde{M})$ vanishes for
  even-dimensional closed Riemannian manifolds, only the case of odd dimensions is
  interesting. For locally symmetric spaces of non-compact type with  odd dimension $d$ one can show that
  $(-1)^{(d-1)/2} \cdot C_d \ge 0$ holds.  Thus one obtains some computational evidence for
  Conjecture~\ref{con:Homological_growth_and_L2-torsion_for_aspherical_manifolds}.

Here is a concrete and already very interesting special case.

\begin{example}[Hyperbolic $3$-manifolds]
\label{exa:hyperbolic_3-manifolds}
Suppose that $M$ is a closed hyperbolic $3$-manifold. Then 
$\rho_{\an}(\widetilde{M})$ is known to be $- \frac{1}{6\pi} \cdot \vol(M)$,
see~\cite[Theorem~4.3 on page~216]{Lueck(2002)}, and hence
Conjecture~\ref{con:Homological_growth_and_L2-torsion_for_aspherical_manifolds}
predicts
\[
\lim_{i \to \infty} \;\frac{\ln\big(\bigl|\tors\bigl(H_1(G_i\bigr)\bigr|\bigr)}{[G:G_i]} =  \frac{1}{6\pi} \cdot \vol(M).
\]
Since the volume is always positive, the equation above implies that $|\tors\bigl(H_1(G_i\bigr)|$
growth exponentially in $[G:G_i]$.  This is in contrast to the question
appearing in the survey paper by Aschenbrenner-Friedl-Wilton~\cite[Question~9.13]{Aschenbrenner-Friedl-Wilton(2015)}
whether a hyperbolic $3$-manifold of finite volume admits a finite covering $N \to M$ 
such that $\tors\bigl(H_1(N)\bigr)$ is non-trivial. However,
a positive  answer to this question  and evidence for 
Conjecture~\ref{con:Homological_growth_and_L2-torsion_for_aspherical_manifolds} for closed hyperbolic
$3$-manifolds is given in Sun~\cite[Corollary~1,6]{Sun(2015)}, where it is shown that
for any finitely generated abelian group $A$, and any closed hyperbolic 3-manifold $M$, there 
exists a finite cover $N$ of $M$, such that $A$ is a direct summand of $H_1(N;\IZ)$.

Bergeron-Sengun-Venkatesh~\cite{Bergeron-Sengun-Venkatesh(2016)} consider 
the equality above for arithmetic hyperbolic $3$-manifolds and relate it to a conjecture about classes in the second integral homology.

Some numerical evidence for the equality above is given in Sengun~\cite{Sengun(2011)}. 

The inequality 
\[
\limsup _{i \to \infty} \;\frac{\ln\big(\bigl|\tors\bigl(H_1(G_i\bigr)\bigr|\bigr)}{[G:G_i]} \le   \frac{1}{6\pi} \cdot \vol(M)
\]
is proved by Thang~\cite{Le(2014)} for a compact connected orientable irreducible $3$-manifold $M$ 
with infinite fundamental group and empty or toroidal boundary. 
\end{example}

\begin{remark}[Possible Scenarios] \label{rem:possible_scenarios}
Consider the situation of Conjecture~\ref{con:Homological_growth_and_L2-torsion_for_aspherical_manifolds}.
We can find for each $i \ge 0$, $n \ge 0$ and prime number $p$ integers $r[i,n] \ge 0$, $t[i,n,p]\ge 0 $, and
$n[i,n,p]_1, n[i,n,p]_2, \ldots, n[i,n,p]_{t[i,n,p]} \ge 1$ such that the set 
$\{p \; \text{prime} \mid t[i,n,p] >  0\}$ is finite and 
\[
H_n(M[i];\IZ) \cong \IZ^{r[i,n]} \oplus \bigoplus_{p \; \text{prime}} \;\bigoplus_{j =1}^{t[i,n,p]} \IZ/p^{n[i,n,p]_j}.
\]
Then
\begin{eqnarray*}
b_n(M[i];\IQ) & = & r[i,n];
\\
b_n(M[i];\IF_p) & = & r[i,n] + t[i,n,p];
\\
d\bigl(H_n(M[i];\IZ)\bigr) & = & r[i,n] + \max\{t[i,n,p]\mid p \; \text{prime}\};
\\
\ln\big(\bigl|\tors\bigl(H_n(M[i])\bigr)\bigr|\bigr)
& = & 
\sum_{p} \sum_{j =1}^{t[i,n,p]} n[i,n,p]_j \cdot \ln(p).
\end{eqnarray*}
Let us discuss first the case where $M$ possesses a Riemannian metric with negative sectional curvature
and $\dim(M)  = 2n+1$. Then Conjecture~\ref{con:Homological_growth_and_L2-torsion_for_aspherical_manifolds} predicts
\begin{eqnarray*}
\lim_{i \to \infty} \frac{r[i,n]}{[G:G_i]} & = & 0;
\\
\lim_{i \to \infty}  \frac{ \max\{t[i,n,p]\mid p \; \text{prime}\}}{[G:G_i]}   & =  & 0;
\\
\lim_{i \to \infty} \;\frac{\sum_{p} \sum_{j =1}^{t[i,n,p]} n[i,n,p]_j \cdot \ln(p)}{[G:G_i]} & > & 0.
\end{eqnarray*}
There are two scenarios which can explain these expected statements. 
Of course there are other scenarios as well, but the two below
illustrate nicely what may happen.

\begin{itemize}

\item The number $ \max\{t[i,n,p]\mid p \; \text{prime}\}$ grows sublinearly in comparison with
$[G:G_i]$. The number of primes $p$ for which $t[i,n,p] \ge 1$ grows linearly with $[G:G_i]$.
A concrete example is the case where
\[
H_1(M[i];\IZ)) \cong \IZ^{r[i,1]} \oplus \bigoplus_{p \in \calp[i]} \IZ/p,
\] 
where $\calp[i]$ is a set of primes satisfying $\lim_{i \to \infty} \frac{|\calp[i]|}{[G:G_i]} > 0$,
and $\lim_{i \to \infty} \frac{r[i,1]}{[G:G_i]}  =  0$ holds;

\item The number $\max\{t[i,n,p]\mid p \; \text{prime}\}$ grows sublinearly in comparison with
$[G:G_i]$. There is a prime $p$ such that the number $\sum_{j =1}^{t[i,n,p]} n[i,n,p]_j$ grows linearly with $[G:G_i]$. 
A concrete example is  the case, when
\[
H_1(M[i];\IZ)) \cong \IZ^{r[i,1]} \oplus   \IZ/p^{m[i,1]}
\]
for a prime $p$ such that $\lim_{i \to \infty} \frac{m[i,1]}{[G:G_i]} > 0$ and 
$\lim_{i \to \infty} \frac{r[i,1]}{[G:G_i]}  =  0$ holds.
\end{itemize}

Next we discuss the case where $M$ possesses a Riemannian metric with negative sectional curvature
and $\dim(M)  = 2n$. Then Conjecture~\ref{con:Homological_growth_and_L2-torsion_for_aspherical_manifolds} predicts
\begin{eqnarray*}
\lim_{i \to \infty} \frac{r[i,n]}{[G:G_i]} & > & 0;
\\
\lim_{i \to \infty} \;\frac{\sum_{p} \sum_{j =1}^{t[i,n,p]} n[i,n,p]_j \cdot \ln(p)}{[G:G_i]} & = & 0.
\end{eqnarray*}
\end{remark}


\typeout{--------------------------------- Section 11: Mapping tori ---------- --------------------}

\section{Mapping tori}
\label{sec:Mapping_tori}

A very interesting case is the example of a mapping torus $T_f$ of a self-map
$f \colon Z \to Z$ of a connected finite $CW$-complex $Z$
(which is not necessarily a homotopy equivalence). The canonical projection $q \colon T_f \to S^1$ induces
an epimorphism $\pr \colon G = \pi_1(T_f) \to \IZ$. Let $K$ be its kernel which can be identified
with the colimit of the direct system of groups indexed by $\IZ$.
\[
\cdots \xrightarrow{\pi_1(f)} \pi_1(Z)  \xrightarrow{\pi_1(f)} \pi_1(Z)  \xrightarrow{\pi_1(f)}  \cdots
\]
In particular the inclusion $Z \to T_f$ induces a homomorphism $j \colon \pi_1(Z) \to K$
which corresponds to the structure map at $0 \in \IZ$ in the description of $K$ as a colimit.
We obtain a short exact sequence 
$1 \to K \xrightarrow{j} G \xrightarrow{\pr} \IZ \to 1$. We use the Setup~\eqref{set:restricted} with $X = T_f$ and
$\overline{X} = \widetilde{T_f}$.  Put $K_i = j^{-1}(G_i)$. Let $d_i \in \IZ$ be the integer for which
$~pr(G_i) = d_i \cdot \IZ$. We obtain an induced exact sequence 
$1 \to K_i \xrightarrow{j_i} G_i \xrightarrow{\pr_i} d_i \cdot \IZ \to 1$. We have
\begin{eqnarray*}
[G:G_i] & = & [K:K_i] \cdot d_i.
\end{eqnarray*}
If $\pi_1(f)$ is an isomorphism, then $j \colon \pi_1(Z) \to K$ is an isomorphism.

Let $p_i \colon S^1 \to S^1$ be the $d_i$-sheeted covering given by $z \mapsto z^{d_i}$.
It is the covering associated to $d_i \cdot \IZ \subseteq \IZ$.
Let $\overline{p_i} \colon T_f[i]' \to T_f$ be the $d_i$-sheeted covering given by the pullback
\[\xymatrix{T_f[i]' \ar[r]^{q[i]'} \ar[d]_{\overline{p_i}}
& S^1 \ar[d]^{p_i}
\\
T_f \ar[r]_q
& S^1
}
\]
It is the covering associated to $\pr^{-1}(d_i \cdot \IZ) \subseteq G = \pi_1(T_f)$.
Let $q_i \colon T_f[i] \to T_f[i]'$ be the $[K:K_i]$-sheeted covering which is associated
to $G_i \subseteq \pr^{-1}(d_i \cdot \IZ) = \pi_1(T_f[i]')$.  The composite $T_f[i]
\xrightarrow{q_i} T_f'[i] \xrightarrow{\overline{p_i}} T_f$ is the $[G:G_i]$-sheeted
covering associated to $G_i \subseteq G = \pi_1(T_f)$.  Let $\overline{q_i} \colon Z[i]\to Z$ 
be the $[K:K_i]$-sheeted covering given by the pullback
\[
\xymatrix{Z[i] \ar[d]^{\overline{q_i}} \ar[r] 
&
T_f[i] \ar[d]^{q_i} 
\\
Z \ar[r]^{i} 
&
T_f[i]'
}
\]
It is the covering given by $K/K_i \times_{\pi_1(Z)/j^{-1}(K_i)} \widetilde{Z}$.

Since $T_f[i]'$ is obtained from the $d_i$-fold mapping telescope of $f$ by identifying the left and the right end by the identity, 
there is an  obvious map $T_f[i]' \to T_{f^{d_i}}$ which turns out to be a homotopy equivalence.
Hence we can choose a  homotopy equivalence
\[
u \colon T_{f^{d_i}} \xrightarrow{\simeq} T_f[i]'.
\]
Define the homotopy equivalence
$\overline{u} \colon \overline{T_{f^{d_i}}} \xrightarrow{\simeq} T_f[i]$ by the pullback
\[
\xymatrix{
\overline{T_{f^{d_i}}} \ar[r]^{\overline{u}} \ar[d]^{\overline{q_i}}
&
T_f[i] \ar[d]^{q_i}
\\
T_{f^{d_i}} \ar[r]^{u}
&
T_f[i]'
}
\]

There is a finite $CW$-structure on $T_{f^{d_i}}$ such that the number of $n$-cells $c_n(T_{f^{d_i}})$ is
$c_n(Z) + c_{n-1}(Z)$, where $c_n(Z)$ is the number of $n$-cells in $Z$. Since
$\overline{q_i} \colon \overline{T_{f^{d_i}}} \to T_{f^{d_i}}$ is a $[K:K_i]$-sheeted covering,
there is a finite $CW$-structure on $\overline{T_{f^{d_i}}}$ such that the number of $n$-cells $c_n(T_{f^{d_i}})$ is
$[K:K_i] \cdot (c_n(Z) + c_{n-1}(Z))$. This implies
\begin{eqnarray*}
\frac{c_n(T_{f^{d_i}})}{[G:G_i]}  & = & \frac{c_n(Z) + c_{n-1}(Z)}{d_i}.
\end{eqnarray*}
Hence we get for any field $F$
\begin{eqnarray*}
\frac{b_n(T_f[i];F)}{[G:G_i]}   & \le  & \frac{c_n(Z) + c_{n-1}(Z)}{d_i};
\\
d(\pi_1(T_f[i])) & \le & \frac{c_1(Z) + c_0(Z)}{d_i};
\\
\frac{{\rm def}(\pi_1(T_f[i]))}{[G:G_i]} & \le & \frac{c_2(Z) +  c_1(Z) + c_0(Z)}{d_i};
\\
\frac{\chi_n^{\trun}(T_f[i])}{[G:G_i]} & \le & \frac{\sum_{i = 0}^n c_i(Z)}{d_i}.
\end{eqnarray*}


\subsection{The case $\bigcap_{i \ge 0} d_i \cdot \IZ = \{1\}$}
\label{subsec:The_case_bigcap_d_i_cdotZ_is_1}

Suppose that $\bigcap_{i \ge 0} d_i \cdot \IZ = \{1\}$, or, equivalently, $\lim_{i \to \infty} d_i = \infty$ holds. 
Then we conclude for any field $F$
\begin{eqnarray*}
\lim_{i \to \infty}  \frac{b_n(T_f[i];F)}{[G:G_i]}   & =  & 0;
\\
\lim_{i \to \infty}  \frac{d(\pi_1(T_f[i]))}{[G:G_i]} & = & 0;
\\
\lim_{i \to \infty}  \frac{{\rm def}(\pi_1(T_f[i]))}{[G:G_i]} & = & 0;
\\
\lim_{i \to \infty}  \frac{\chi_n^{\trun}(T_f[i])}{[G:G_i]} & =  & 0.
\end{eqnarray*}
Since
\begin{eqnarray}
b_n^{(2)}(\widetilde{T_f}) & = & 0
\label{L2-Betti_numbers_of_tori_vanish}
\end{eqnarray}
holds for $n \ge 0$ by~\cite[Theorem~2.1]{Lueck(1994b)}, this gives evidence 
for Conjecture~\ref{con:Approximation_in_zero_and_prime_characteristic}
and Conjecture~\ref{con:Growth_of_number_of_generators_of_the_homology},
and positive answers to Questions~\ref{que:Rank_gradient_cost_and_first_L2_Betti_number}
and Question~\ref{que:asymptotic_Morse_equality}, provided that $Z$ is aspherical.


\subsection{The case $\bigcap_{i \ge 0} d_i \cdot \IZ \not= \{1\}$}
\label{subsec:The_case_bigcap_d_i_cdotZ_is_not_1}
Next we consider the hard case $\bigcap_{i \ge 0} d_i \cdot \IZ \not= \{1\}$. Then there exists an  integer $i_0$ such that $d_i = d_{i_0}$
for all $i \ge i_0$. We can assume without loss of generality that $d_i = 1$ holds for all $i \ge 0$, otherwise
replace $T_f$ by $T_f[i_0]$, $G$ by $G_{i_0}$, $\IZ$ by $n_{i_0} \cdot \IZ$, and $(G_i)_{i \ge 0}$ by $(G_i)_{i \ge i_0}$.

We conclude from Theorem~\ref{the:approx_Betti_char_zero},
Theorem~\ref{the:dim_approximation_over_fields} and~\eqref{L2-Betti_numbers_of_tori_vanish} 
\begin{eqnarray}
  \lim_{i \to \infty}  \frac{b_n(T_f[i];F)}{[G:G_i]}   & =  & 0.
  \label{lim_Betti_numbers_mapping_torus_is_zero}
\end{eqnarray}
provided that $F$ has characteristic zero. We get the same
conclusion~\eqref{lim_Betti_numbers_mapping_torus_is_zero} for any field $F$ provided that
$G$ is torsion-free elementary amenable and residually finite by
Theorem~\ref{the:dim_approximation_over_fields} 
and~\eqref{L2-Betti_numbers_of_tori_vanish}, since Theorem~\ref{the:dim_approximation_over_fields} implies
for torsion-free elementary $G$ that $\lim_{i \to \infty} \frac{b_n(T_f[i];F)}{[G:G_i]}$
exists and is independent of the chain $(G_i)_{i \ge 0}$ and because of
Subsection~\ref{subsec:The_case_bigcap_d_i_cdotZ_is_1} there exists an  appropriate chain,
for instance $(G_i \cap \pr^{-1}(2^i \cdot \IZ))_{i \ge 0}$,  with $\lim_{i \to \infty} \frac{b_n(T_f[i];F)}{[G:G_i]} = 0$. We do not
know whether~\eqref{lim_Betti_numbers_mapping_torus_is_zero} holds for arbitrary fields
and arbitrary residually finite groups $G$, 
as predicted by Conjecture~\ref{con:Approximation_in_zero_and_prime_characteristic}.

We do not know whether the rank gradient $RG(G,(G_i)_{i_\ge 0})$ is zero for any chain $(G_i)_{i \ge 0}$
as predicted by Conjecture~\ref{que:Rank_gradient_cost_and_first_L2_Betti_number} 
in view of~\eqref{lim_Betti_numbers_mapping_torus_is_zero}, but at least for chains
with $\bigcap_{i \ge 0} d_i \cdot \IZ = \{1\}$ this follows from Subsection~\ref{subsec:The_case_bigcap_d_i_cdotZ_is_1}.
This illustrates why it would be very interesting to know 
whether the rank gradient $RG(G,(G_i)_{i \ge 0})$ is independent of the chain $(G_i)_{i \ge 0}$.
The same remark applies to the more general Question~\ref{que:asymptotic_Morse_equality}.

One can express $b_n(T_f[i];F)$ in terms of $f$ as follows.
Obviously $K_i$ is a normal subgroup of $G$, the automorphism of $K$ induced by $f$
sends $K_i$ to $K_i$ and we have $[G:G_i] = [K:K_i]$ for all $i \ge 0$. Put $f[0] = f \colon Z = Z[0] \to Z = Z[0]$.
We can choose for each $i \ge 1$ self-homotopy equivalences
$f[i] \colon Z[i] \to Z[i]$ for which the following diagram with the obvious coverings as vertical maps
\[
\xymatrix{Z[i] \ar[r]^{f[i]} \ar[d]
&
Z[i]\ar[d]
\\
Z[i-1] \ar[r]^{f[i-1]} 
&
Z[i-1]
}
\]
commutes.  Then $T_f[i]$ is $T_{f[i]}$. We have the Wang sequence of $R$-modules for any commutative  ring $R$
\begin{multline}
\cdots \to H_n(Z[i];R) \xrightarrow{\id - H_n(f[i];R)} H_n(Z[i];R) \to H_n(T_f[i]) 
\\
\to H_{n-1}(Z[i];R) \xrightarrow{\id - H_{n-1}(f[i];R)} H_{n-1}(Z[i];R) \to\cdots
\label{Wang_sequence}
\end{multline}
This implies
\begin{multline}
b_n(T_f[i];F) 
\\
= 
\dim_F\bigl(\coker\bigl(\id - H_n(f[i];F) \bigr)\bigr) + \dim_F\bigl(\ker\bigl(\id - H_{n-1}(f[i];F) \bigr)\bigr).
\label{Betti_numbers_and_Wang}
\end{multline}


\subsection{Self-homeomorphism of a surface}
\label{subsec:Selfhomeomorphism_of_a_surface}

Now assume that $Z$ is a closed orientable surface of genus $g$ and $f \colon Z \to Z$ is an orientation preserving
selfhomeomorphism.

If $g = 0$, we get $T_f = S^1 \times S^2$ and in this case everything can be computed directly.

If $g = 1$, then $\pi_1(T_f)$ is poly-$\IZ$ and $T_f$ is aspherical, and hence we know already that
Conjecture~\ref{con:Approximation_in_zero_and_prime_characteristic} and
Conjecture~\ref{con:Growth_of_number_of_generators_of_the_homology}
are true, the 
answers to Questions~\ref{que:Rank_gradient_cost_and_first_L2_Betti_number}
and Question~\ref{que:asymptotic_Morse_equality} are positive, and 
Conjecture~\ref{con:Homological_growth_and_L2-torsion_for_aspherical_manifolds}
is true by applying Remark~\ref{rem:known_cases}, Lemma~\ref{lem:consequences_of_slow_growth}, 
Example~\ref{exa:Examples_of_groups_with_slow_growth_in_dimensions_le_d}
and Theorem~\ref{the:Groups_containing_a_normal_infinite_nice_subgroups}.

So the interesting (and open) case is $g \ge 2$. In this situation equality~\eqref{Betti_numbers_and_Wang}
becomes
\begin{eqnarray*}
b_n(T_f[i];F) & = & 
\begin{cases}
\dim_F\bigl(\coker\bigl(\id - H_1(f[i];F) \bigr)\bigr)  + 1 & n = 1;
\\
\dim_F\bigl(\ker\bigl(\id - H_1(f[i];F) \bigr)\bigr) + 1 & n = 2;
\\
1 & n = 0,3;
\\ 0 & n \ge 4.
\end{cases}
\end{eqnarray*}

We know for all $n \ge 0$ 
\begin{eqnarray}
  \lim_{i \to \infty}  \frac{b_n(T_f[i];F)}{[G:G_i]}   & =  & 0,
  \label{lim_Betti_numbers_mapping_torus_for_surface_is_zero}
\end{eqnarray}
provided that $F$ has characteristic zero. Notice that~\eqref{lim_Betti_numbers_mapping_torus_for_surface_is_zero}
for a field of characteristic zero is equivalent to 
\begin{eqnarray*}
  \lim_{i \to \infty}  \frac{\dim_{\IQ}\bigl(\coker(\id - H_1(f[i];\IZ)) \otimes_{\IZ} \IQ\bigr)}{[G:G_i]}   & =  & 0.
\end{eqnarray*}
Next we consider the case that $F$ is a field of prime characteristic $p$.
Then we do know~\eqref{lim_Betti_numbers_mapping_torus_for_surface_is_zero}
in the situation of
Subsection~\ref{subsec:The_case_bigcap_d_i_cdotZ_is_1} but not in the situation
of Subsection~\ref{subsec:The_case_bigcap_d_i_cdotZ_is_not_1}.
Recall that  Conjecture~\ref{con:Approximation_in_zero_and_prime_characteristic} 
predicts~\eqref{lim_Betti_numbers_mapping_torus_for_surface_is_zero} 
in view of~\eqref{L2-Betti_numbers_of_tori_vanish} also in this case.
In order to prove~\eqref{lim_Betti_numbers_mapping_torus_is_zero}
also for a field $F$ of prime characteristic $p$ for all $n \ge 0$
in the situation of Subsection~\ref{subsec:The_case_bigcap_d_i_cdotZ_is_1}, it suffices to show
\begin{eqnarray*}
  \lim_{i \to \infty}  \frac{\dim_{\IF_p}\bigl(\tors(H_1(T_f[i];\IZ)) \otimes_{\IZ} \IF_p\bigr)}{[G:G_i]}   & =  & 0;
\end{eqnarray*}
or equivalently 
\begin{eqnarray*}
  \lim_{i \to \infty}  \frac{\dim_{\IF_p}\bigl(\tors\bigl(\coker(\id - H_1(f[i];\IZ))\bigr) \otimes_{\IZ} \IF_p\bigr)}{[G:G_i]}   & =  & 0.
\end{eqnarray*}
So one needs to understand more about the maps $\id - H_1(f[i];\IZ) \colon H_1(Z[i];\IZ) \to H_1(Z[i];\IZ)$ for $i \ge 0$. 

The status of
Conjecture~\ref{con:Homological_growth_and_L2-torsion_for_aspherical_manifolds}~%
\eqref{con:Homological_growth_and_L2-torsion_for_aspherical_manifolds:tors}  is even
more mysterious.  Suppose that $f \colon Z \to Z$ is an orientation preserving irreducible  selfhomeomorphism of a closed
orientable surface $Z$ of genus $g \ge 2$.
If $f$ is periodic, $T_f$ is finitely covered by $S^1 \times Z$ and  
Conjecture~\ref{con:Homological_growth_and_L2-torsion_for_aspherical_manifolds} is known to be true.
Therefore we consider from now on the case, where $f$ is not periodic.
Then $f$ is  pseudo-Anosov, see~\cite[Theorem 6.3]{Casson-Bleiler(1988)},
and $T_f$ carries the structure of a hyperbolic $3$-manifold 
by~\cite[Theorem 3.6 on page 47, Theorem 3.9 on page 50]{McMullen(1996)}. Hence
Conjecture~\ref{con:Homological_growth_and_L2-torsion_for_aspherical_manifolds} predicts, see 
Example~\ref{exa:hyperbolic_3-manifolds},
\begin{eqnarray*}
\lim_{i \to \infty} \;\frac{\ln\big(\bigl|\tors(H_1(T_f[i];\IZ))\bigr|\bigr)}{[G:G_i]}  
> 0,
\\
\lim_{i \to \infty}  \frac{\dim_{\IF_p}\bigl(\tors(H_1(T_f[i];\IZ)) \otimes_{\IZ} \IF_p\bigr)}{[G:G_i]}   & =  & 0,
\\
\lim_{i \to \infty}  \frac{d\bigl(\tors(H_1(T_f[i];\IZ)) \bigr)}{[G:G_i]}   & =  & 0,
\end{eqnarray*}
or, because of the Wang sequence~\eqref{Wang_sequence} equivalently, 
\begin{eqnarray*}
\lim_{i \to \infty} \;\frac{\ln\big(\bigl|\tors\bigl(\coker(\id - H_1(f[i];\IZ))\bigr)\bigr|\bigr)}{[G:G_i]} 
& > & 
0,
\\
 \lim_{i \to \infty}  \frac{\dim_{\IF_p}\bigl(\tors\bigl(\coker(\id - H_1(f[i];\IZ))\bigr) \otimes_{\IZ} \IF_p\bigr)}{[G:G_i]}   
& =  & 
0,
\\
 \lim_{i \to \infty}  \frac{d\bigl(\tors\bigl(\coker(\id - H_1(f[i];\IZ))\bigr)\bigr)}{[G:G_i]}   
& =  & 
0.
\end{eqnarray*}


\typeout{-------------------- Section 12: Dropping the finite index condition --------------------------}

\section{Dropping the finite index condition}
\label{sec:Dropping_the_finite_index_condition}

From now on we want to drop the condition that the index of the subgroups $G_i$ in $G$ is finite
and that the index set for the chain is given by the natural numbers. So we will consider
for the remainder of this paper the following  more general  situation:

\begin{setup}[Inverse system]\label{set:inverse_systems}
  Let $G$ be a group together with an inverse system $\{G_i \mid i \in I\}$ of normal subgroups of $G$
  directed by inclusion over the directed set $I$ such that $\bigcap_{i \in I} G_i = \{1\}$.
\end{setup}

If $I$ is given by the natural numbers, this boils down to a nested sequence of
normal subgroups
\[G =  G_0 \supset G_1 \supseteq G_2 \supseteq \cdots
\]
satisfying $\bigcap_{n \ge 1} G_n = \{1\}$. If we additionally assume that $[G:G_i]$ is finite,
we are back in the previous special situation~\eqref{normal_chain}. Some of the following conjectures
reduce to previous conjectures in this special case. The reason is that for a finite group $H$ 
and a based free finite $\IZ H$-chain complex $D_*$ we
have
\begin{eqnarray*}
b_p^{(2)}(D_*^{(2)};\caln(H)) & = & \frac{b_p^{(2)}(D_*^{(2)};\caln(\{1\}))}{|H|};
\\
\rho^{(2)}(D_*^{(2)};\caln(H)) & = & \frac{\rho^{(2)}(D_*^{(2)};\caln(\{1\}))}{|H|}.
\end{eqnarray*}


\typeout{--- Section 13: Review of the  Determinant Conjecture and the Approximation Conjecture for L2-Betti numbers --}

\section{Review of the Determinant Conjecture and the Approximation Conjecture for $L^2$-Betti numbers}
\label{sec:Review_of_the_Determinant_Conjecture_and_the_Approximation_Conjecture_for_L2-Betti_numbers}

We begin with the Determinant Conjecture (see~\cite[Conjecture~13.2 on page~454]{Lueck(2002)}). 

\begin{conjecture}[Determinant Conjecture for a group $G$]
  \label{con:Determinant_Conjecture}
  For any matrix $A \in M_{r,s}(\IZ G)$, the Fuglede-Kadison determinant of the
  morphism of Hilbert modules $r_A^{(2)}\colon L^2(G)^r \to L^2(G)^s$ given by
  right multiplication with $A$ satisfies
\[{\det}_{\caln(G)}^{(2)}\bigl(r_A^{(2)}\bigr)\ge 1.
\]
\end{conjecture}

\begin{remark}[Status of the Determinant Conjecture]
  \label{rem:status_of_Determinant_Conjecture}
  We will often have to assume that the Determinant
  Conjecture~\ref{con:Determinant_Conjecture} is true. This is an acceptable
  condition since it is known for a large class of groups.  Namely, the
  following is known (see~\cite[Theorem~5]{Elek-Szabo(2005)},
  \cite[Section~13.2 on pages~459~ff]{Lueck(2002)}, \cite[Theorem~1.21]{Schick(2001b)}).  Let $\calf$ be the
  class of groups for which the Determinant
  Conjecture~\ref{con:Determinant_Conjecture} is true.  Then:
  \begin{enumerate}

  \item \label{rem:status_of_Determinant_Conjecture:amenable_quotient}
    Amenable quotient\\
    Let $H \subset G$ be a normal subgroup. Suppose that $H \in \calf$ and the
    quotient $G/H$ is amenable. Then $G \in \calf$;

  \item \label{rem:status_of_Determinant_Conjecture:direct_limit}
    Colimits\\
    If $G = \colim_{i \in I} G_i$ is the colimit of the directed system $\{G_i
    \mid i \in I\}$ of groups indexed by the directed set $I$ (with not
    necessarily injective structure maps) and each $G_i$ belongs to $\calf$,
    then $G$ belongs to $\calf$;

  \item \label{rem:status_of_Determinant_Conjecture:inverse_limit}
    Inverse limits\\
    If $G = \lim_{i \in I} G_i$ is the limit of the inverse system $\{G_i \mid i
    \in I\}$ of groups indexed by the directed set $I$ and each $G_i$ belongs to
    $\calf$, then $G$ belongs to $\calf$;

  \item \label{rem:status_of_Determinant_Conjecture:subgroups}
    Subgroups\\
    If $H$ is isomorphic to a subgroup of a group $G$ with $G \in \calf$, then
    $H \in \calf$;

  \item \label{rem:status_of_Determinant_Conjecture:quotient_with_finite_kernel}
    Quotients with finite kernel\\
    Let $1 \to K \to G \to Q \to 1$ be an exact sequence of groups. If $K$ is
    finite and $G$ belongs to $\calf$, then $Q$ belongs to $\calf$;

  \item \label{rem:status_of_Determinant_Conjecture:sofic_groups} Sofic groups
    belong to $\calf$.
  \end{enumerate}

  The class of sofic groups is very large.  It is closed under direct and free
  products, taking subgroups, taking inverse and direct limits over directed index sets, and is closed
  under extensions with amenable groups as quotients and a sofic group as
  kernel.  In particular it contains all residually amenable groups.  One
  expects that there exists non-sofic groups but no example is known.  More
  information about sofic groups can be found for instance
  in~\cite{Elek-Szabo(2006)} and~\cite{Pestov(2008)}.

\end{remark}

\begin{notation}[Inverse systems and matrices]\label{not:inverse_systems_and_matrices}
  Let $R$ be a ring with $\IZ \subseteq R \subseteq \IC$. Given a matrix $A \in
  M_{r,s}(RG)$, let $A[i]\in M_{r,s}(R[G/G_i])$ be the matrix obtained from
  $A$ by applying elementwise the ring homomorphism $RG \to R[G/G_i]$
  induced by the projection $G \to G/G_i$.  Let $r_A \colon RG^r \to RG^s$ and
  $r_{A[i]} \colon R[G/G_i]^r \to R[G/G_i]^s$ be the $RG$- and
  $R[G/G_i]$-homomorphisms given by right multiplication with $A$ and $A[i]$.
  Let $r_A^{(2)} \colon L^2(G)^r \to L^2(G)^s$ and $r_{A[i]}^{(2)} \colon
  L^2(G/G_i)^r \to L^2(G/G_i)^s$ be the morphisms of Hilbert $\caln(G)$- and
  Hilbert $\caln(G/G_i)$-modules given by right multiplication with $A$ and
  $A[i]$.
\end{notation}

Next we deal with the Approximation Conjecture for $L^2$-Betti numbers
(see~\cite[Conjecture~1.10]{Schick(2001b)}, 
\cite[Conjecture~13.1 on page~453]{Lueck(2002)}).

\begin{conjecture}[Approximation Conjecture for $L^2$-Betti numbers]
  \label{con:Approximation_conjecture_for_L2-Betti_numbers}
  A group $G$ together with an inverse system $\{G_i \mid i \in I\}$ as in
  Setup~\ref{set:inverse_systems} satisfies the \emph{Approximation
    Conjecture for $L^2$-Betti numbers} if one of the following equivalent
  conditions holds:

  \begin{enumerate}

  \item Matrix version\\[1mm]
    Let $A \in M_{r,s}(\IQ G)$ be a matrix. Then
    \begin{eqnarray*}
      \lefteqn{\dim_{\caln(G)}\bigl(\ker\bigl(r_A^{(2)}\colon L^2(G)^r \to L^2(G)^s
        \bigr)\bigr)}
      & &
      \\ & \hspace{14mm} =  &
      \lim_{i \in I} \;\dim_{\caln(G/G_i)}\big(\ker
      \big(r_{A[i]}^{(2)}\colon L^2(G/G_i)^r \to L^2(G/G_i)^s \bigr)\bigr);
    \end{eqnarray*}

  \item $CW$-complex version\\[1mm]
    Let $X$ be a $G$-$CW$-complex of finite type. Then $X[i] := G_i\backslash X$ is a
    $G/G_i$-$CW$-complex of finite type and
    \begin{eqnarray*} b_p^{(2)}(X;\caln(G)) & = & \lim_{i \in I}
      \;b_p^{(2)}(X[i];\caln(G/G_i)).
    \end{eqnarray*}

  \end{enumerate}
\end{conjecture}

The two conditions appearing in
Conjecture~\ref{con:Approximation_conjecture_for_L2-Betti_numbers} are
equivalent by~\cite[Lemma~13.4 on page~455]{Lueck(2002)}.

We will frequently make the following assumption:

\begin{assumption}[Determinant Conjecture]
  \label{ass:Determinant_Conjecture}
  For each $i \in I$ the quotient $G/G_i$ satisfies the Determinant 
  Conjecture~\ref{con:Determinant_Conjecture}.
\end{assumption}

\begin{theorem}[The Determinant Conjecture implies the Approximation Conjecture
  for $L^2$-Betti numbers]
  \label{the:The_Determinant_Conjecture_implies_the_Approximation_Conjecture_for_L2-Betti_numbers}
  If Assumption~\ref{ass:Determinant_Conjecture} holds, then the conclusion of
  the Approximation
  Conjecture~\ref{con:Approximation_conjecture_for_L2-Betti_numbers} holds for
  $\{G_i \mid i \in I\}$.
\end{theorem}
\begin{proof}
  See~\cite[Theorem~13.3 (1) on page~454]{Lueck(2002)}  and~\cite{Schick(2001b)}.
\end{proof}

  Suppose that each quotient $G/G_i$ is finite. Then
  Assumption~\ref{ass:Determinant_Conjecture} is fulfilled by Remark~\ref{rem:status_of_Determinant_Conjecture},
  and  we   recover Theorem~\ref{the:approx_Betti_char_zero} from
  Theorem~\ref{the:The_Determinant_Conjecture_implies_the_Approximation_Conjecture_for_L2-Betti_numbers}.

For more information about the Approximation Conjecture and the Determinant
Conjecture we refer to~\cite[Chapter~13 on pages~453~ff]{Lueck(2002)} and~\cite{Schick(2001b)}.


\typeout{----- Section 14: The Approximation Conjecture for Fuglede-Kadison determinants  and L^2-torsion ----}

\section{The Approximation Conjecture for Fuglede-Kadison determinants 
and $L^2$-torsion}
\label{sec:The_Approximation_Conjecture_for_Fuglede-Kadison_determinants_and_L2-torsion}

Next we turn to Fuglede-Kadison determinants and $L^2$-torsion.


\subsection{The chain complex version}
\label{subsec:The_chain_complex_version}

\begin{conjecture}[Approximation Conjecture for Fuglede-Kadison determinants]
\label{con:Approximation_conjecture_for_Fuglede-Kadison_determinants_with_arbitrary_index}
A group $G$ together with an inverse system $\{G_i \mid i \in I\}$ as in
Setup~\ref{set:inverse_systems} satisfies 
the \emph{Approximation Conjecture for Fuglede-Kadison determinants}
if for any matrix $A \in M_{r,s}(\IQ G)$ 
we get for the Fuglede-Kadison determinant
\begin{eqnarray*}
\lefteqn{{\det}_{\caln(G)}\bigl(r_A^{(2)}\colon L^2(G)^r \to L^2(G)^s\bigr)}
& &
\\ & \hspace{14mm} =  &
\lim_{i \in I}\; {\det}_{\caln(G/G_i)}\big(r_{A[i]}^{(2)}\colon L^2(G/G_i)^r \to L^2(G/G_i)^s\bigr),
\end{eqnarray*}
where the existence of the limit above is part of the claim.
\end{conjecture}

\begin{notation}[Inverse systems and chain complexes]
\label{not:inverse_system_and_chain_complexes}
Let $C_*$ be a finite based free $\IQ G$-chain complex.  In the sequel we denote
by $C[i]_*$ the $\IQ [G/G_i]$-chain complex $\IQ [G/G_i]\otimes_{\IQ G} C_* $, by
$C^{(2)}_*$ the finite Hilbert $\caln(G)$-chain complex $L^2(G) \otimes_{\IQ G} C_*$, 
and by $C[i]^{(2)}_*$ the finite Hilbert $\caln(G/G_i)$-chain complex
$L^2(G/G_i) \otimes_{\IQ[G/G_i]} C[i]_*  $. The $\IQ G$-basis for
$C_*$ induces a $\IQ[G/G_i]$-basis for $C[i]_*$ and Hilbert space structures on
$C^{(2)}_*$ and $C[i]_*^{(2)}$ using the standard Hilbert structure on $L^2(G)$
and $L^2(G/G_i)$.  We emphasize that in the sequel after fixing a $\IQ G$-basis
for $C_*$ the $\IQ[G/G_i]$-basis for $C_*[i]$ and the Hilbert structures on
$C_*^{(2)}$ and $C[i]_*^{(2)}$ have to be chosen in this particular way.

Denote by
\begin{eqnarray}
  \rho^{(2)}\big(C_*^{(2)}\bigr) 
   & := & 
   - \sum_{p \ge 0} (-1)^p \cdot \ln\bigl({\det}_{\caln(G)}^{(2)}\bigl(c_p^{(2)}\bigr)\bigr);
  \label{L2-torsion_for_C_over_cakln(G)}
  \\
  \rho^{(2)}\big(C[i]_*^{(2)}\bigr) 
  & := & 
  - \sum_{p \ge 0} (-1)^p \cdot \ln\bigl({\det}_{\caln(G/G_i)}^{(2)}\bigl(c[i]_p^{(2)}\bigr)\bigr),
  \label{L2-torsion_for_C[i]_over_caln(G/G_i)}
\end{eqnarray}
their \emph{$L^2$-torsion} over $\caln(G)$ and $\caln(G/G_i)$ respectively.
\end{notation}

We have the following chain complex version of 
Conjecture~\ref{con:Approximation_conjecture_for_Fuglede-Kadison_determinants_with_arbitrary_index}
which is obviously equivalent to 
Conjecture~\ref{con:Approximation_conjecture_for_Fuglede-Kadison_determinants_with_arbitrary_index}

\begin{conjecture}[Approximation Conjecture for $L^2$-torsion of chain complexes]
\label{con:Approximation_conjecture_for_L2-torsion_of_chain_complexes}
A group $G$ together with an inverse system $\{G_i \mid i \in I\}$ as in
Setup~\ref{set:inverse_systems}
satisfies the \emph{Approximation Conjecture for $L^2$-torsion of chain complexes}
if for any finite based free $\IQ G$-chain complex $C_*$ we have
\[
\rho^{(2)}\bigl(C_*^{(2)}\bigr) = \lim_{i \in I} \;\rho^{(2)}\bigl(C[i]_*^{(2)}\bigr).
\]
\end{conjecture}


\subsection{$L^2$-torsion}
\label{subsec:L2-torsion}

Let $\overline{M}$ be a Riemannian manifold without boundary that 
comes with a proper free cocompact isometric $G$-action.
Denote by $M[i]$ the Riemannian manifold obtained
from $\overline{M}$ by dividing out the $G_i$-action. The Riemannian metric on
$M[i]$ is induced by the one on $M$. There is an obvious proper free 
cocompact isometric $G/G_i$-action on $M[i]$ induced by the given $G$-action on
$\overline{M}$. Notice that $M = \overline{M}/G$ is a closed Riemannian manifold
and we get a $G$-covering $\overline{M} \to M$ and a $G/G_i$-covering $M[i] \to M$
which are compatible with the Riemannian metrics. Denote by
\begin{eqnarray}
  \rho^{(2)}_{\an}\big(\overline{M};\caln(G)\bigr) & \in & \IR;
  \label{L2-torsion_for_M_over_cakln(G)}
  \\
  \rho^{(2)}_{\an}\big(M[i];\caln(G/G_i)\bigr) & \in & \IR,
  \label{L2-torsion_for_M[i]_over_caln(G/G_i)}
\end{eqnarray}
their \emph{analytic $L^2$-torsion} over $\caln(G)$ and $\caln(G/G_i)$
respectively.

\begin{conjecture}[Approximation Conjecture for analytic $L^2$-torsion]
\label{con:Approximation_conjecture_for_analytic_L2-torsion}
Consider a group $G$ together with an inverse system $\{G_i \mid i \in I\}$ as in
Setup~\ref{set:inverse_systems}.  Let $\overline{M}$ be a Riemannian manifold without boundary that 
comes with a proper free cocompact isometric $G$-action. Then
\[\rho^{(2)}_{\an}\big(\overline{M};\caln(G)\bigr) 
= \lim_{i \in I} \;\rho^{(2)}_{\an}\big(M[i];\caln(G/G_i)\bigr).
\]
\end{conjecture}

\begin{remark} \label{rem:main_theorem_answers_questions}
  The conjectures above imply a positive
  answer to~\cite[Question~21]{Deninger(2009Mahler)} and~\cite[Question~13.52 on
  page~478 and Question~13.73 on page~483]{Lueck(2002)}.  They also would
  settle~\cite[Problem~4.4 and Problem~6.4]{Kitano-Morifuji(2008)}
  and~\cite[Conjecture~3.5] {Kitano-Morifuji-Takasawa(2004surface_bundle)}. One may wonder
  whether it is related to the Volume Conjecture due to
  Kashaev~\cite{Kashaev(1997)} and H. and J. Murakami~\cite[Conjecture~5.1 on
  page~102]{Murakami-Murakami(2001)}.
\end{remark}

We will prove in 
Section~\ref{sec:The_L2de_Rham_isomorphism_ad_the_proof_of_theorem_ref(the:comparing_analytic_and_chain_complexes)}
the following result
which in the weakly acyclic case reduces Conjecture~\ref{con:Approximation_conjecture_for_analytic_L2-torsion}
to Conjecture~\ref{con:Approximation_conjecture_for_Fuglede-Kadison_determinants_with_arbitrary_index}.

\begin{theorem}\label{the:comparing_analytic_and_chain_complexes}
  Consider a group $G$ together with an inverse system $\{G_i \mid i \in I\}$ as
  in Setup~\ref{set:inverse_systems}.  Let $\overline{M}$ be a Riemannian manifold without
  boundary that comes with a proper free cocompact isometric $G$-action. Suppose
  that $b_p^{(2)}(\overline{M};\caln(G)) = 0$ for all $p \ge 0$. Assume that the
  Approximation Conjecture for $L^2$-torsion of chain
  complexes~\ref{con:Approximation_conjecture_for_L2-torsion_of_chain_complexes}
  (or, equivalently,
  Conjecture~\ref{con:Approximation_conjecture_for_Fuglede-Kadison_determinants_with_arbitrary_index})
  holds for $G$.

  Then
  Conjecture~\ref{con:Approximation_conjecture_for_analytic_L2-torsion}
  holds for $M$, i.e.,
  \[
  \rho^{(2)}_{\an}\big(\overline{M};\caln(G)\bigr) = \lim_{i \in I}
  \;\rho^{(2)}_{\an}\big(M[i];\caln(G/G_i)\bigr).
  \]
\end{theorem}

It is conceivable that Theorem~\ref{the:comparing_analytic_and_chain_complexes}
remains to true if we drop the assumption that $b_p^{(2)}(\overline{M};\caln(G))$ vanishes for all $p
\ge 0$, but our present proof works only under this assumption
(see Remark~\ref{rem:On_the_L2-acyclicity_assumption}).

A more general notion of $L^2$-torsion called \emph{universal $L^2$-torsion} and the relevant algebraic $K$-groups,
where it takes values in, 
are investigated in~\cite{Friedl-Lueck(2016l2_universal),Linnell-Lueck(2016)}.


\subsection{An inequality}
\label{subsec:an_inequality}

We always  have the following inequality.

\begin{theorem}[Inequality]
\label{the:inequality_det_det}
Consider a group $G$ together with an inverse system $\{G_i \mid i \in I\}$ as in
Setup~\ref{set:inverse_systems}. Suppose that 
Assumption~\ref{ass:Determinant_Conjecture} holds.
Consider a matrix $A \in M_{r,s}(\IQ G)$ with coefficients in $\IQ G$.

Then we get the inequality
\begin{multline*}
{\det}_{\caln(G)}^{(2)}\big(r_A^{(2)}\colon L^2(G)^r \to L^2(G)^s \bigr)
\\ 
\ge
\limsup _{i \in I} {\det}_{\caln(G/G_i)}^{(2)}\big(r_{A[i]}^{(2)}\colon L^2(G/G_i)^r \to L^2(G/G_i)^s \bigr).
\end{multline*}
\end{theorem}

The proof of Theorem~\ref{the:inequality_det_det} will be given in 
Section~\ref{sec:Proof_of_some_Theorems_for_Fuglede-Kadison_determinants}.


\subsection{Matrices invertible in $L^1(G)$}
\label{subsec:Matrices_invertible_in_L1(G)}

\begin{theorem}[Invertible matrices over $L^1(G)$]
\label{the:invertible_matrices_over_l1(G)}
Consider a group $G$ together with an inverse system $\{G_i \mid i \in I\}$ as in
Setup~\ref{set:inverse_systems}. 
Consider an invertible matrix $A \in \GL_d(L^1(G))$ with coefficients in $L^1(G)$. 
The projection $G \to G/G_i$ induces a ring
homomorphism $L^1(G) \to L^1(G/G_i)$.  Thus we obtain for each $i \in I$ an invertible matrix
$A[i]  \in \GL_d(L^1(G/G_i))$. 

Then the Approximation Conjecture for Fuglede-Kadison 
determinants~\ref{con:Approximation_conjecture_for_Fuglede-Kadison_determinants_with_arbitrary_index}
holds for $A$, i.e.,
\[
{\det}_{\caln(G)}^{(2)}\big(r_A^{(2)}\colon L^2(G)^d\to L^2(G)^d \bigr)
= \lim_{i \in I} {\det}_{\caln(G/G_i)}^{(2)}\big(r_{A[i]}^{(2)}\colon L^2(G/G_i)^d \to L^2(G/G_i)^d \bigr).
\]
\end{theorem}

Theorem~\ref{the:invertible_matrices_over_l1(G)} has already been proved by
  Deninger~\cite[Theorem~17]{Deninger(2009Mahler)} in the case $d = 1$. Notice that
  Deninger~\cite[page~46]{Deninger(2009Mahler)} uses a different definition of
  Fuglede-Kadison determinant which agrees with ours for injective operators
  by~\cite[Lemma~3.15 (5) on page~129]{Lueck(2002)}.

\begin{corollary}\label{cor:l1-chain_equivalence}
Consider a group $G$ together with an inverse system $\{G_i \mid i \in I\}$ as in
Setup~\ref{set:inverse_systems}. 

\begin{enumerate}
\item \label{cor:l1-chain_equivalence:acyclic}
Let $C_*$ be a finite based free $L^1(G)$-chain complex which is acyclic. Then
\[
\rho^{(2)}\bigl(C_*^{(2)}\bigr)  = \lim_{i \in I} \rho^{(2)}\bigl(C_*[i]^{(2)}\bigr);
\] 

\item \label{cor:l1-chain_equivalence:equivalence}
Let $C_*$ and $D_*$ be finite based free $L^1(G)$-chain complexes. Suppose that they are
$L^1(G)$-chain homotopy equivalent. Then
\[\rho^{(2)}\bigl(C_*^{(2)}\bigr)- \rho^{(2)}\bigl(D_*^{(2)}\bigr) 
= \lim_{i \in I} \left(\rho^{(2)}\bigl(C_*[i]^{(2)}\bigr) - \rho^{(2)}\bigl(D_*[i]^{(2)}\bigr)\right).
\] 
\end{enumerate}
\end{corollary}

The proofs of Theorem~\ref{the:invertible_matrices_over_l1(G)} and Corollary~\ref{cor:l1-chain_equivalence}
will be given in Section~\ref{sec:Proof_of_some_Theorems_for_Fuglede-Kadison_determinants}.


\typeout{---   Section 15: The L2-de Rham isomorphism proof of Theorem ref(the:comparing_analytic_and_chain_complexes)---}

\section{The $L^2$-de Rham isomorphism and the proof of Theorem~\ref{the:comparing_analytic_and_chain_complexes}}
\label{sec:The_L2de_Rham_isomorphism_ad_the_proof_of_theorem_ref(the:comparing_analytic_and_chain_complexes)}

In this section we investigate the $L^2$-de Rham isomorphism in order to give the proof of
Theorem~\ref{the:comparing_analytic_and_chain_complexes}.

Let $M$ be a closed Riemannian manifold. Fix a smooth triangulation $K$ of $M$.
Consider a (discrete) group $G$ and a $G$-covering $\pr\colon \overline{M} \to
M$. The smooth triangulation $K$ of $M$
lifts to $G$-equivariant smooth triangulation $\overline{K}$ of $\overline{M}$.  Denote by
$\pr\colon \overline{K} \to K$ the associated $G$-covering. Equip $\overline{M}$
with the Riemannian metric for which $\pr \colon \overline{M} \to M$ becomes a
local isometry.

In the sequel we will consider the de Rham isomorphism
\begin{eqnarray}
  \Int^p  \colon \calh^p_{(2)}(\overline{M}) & \xrightarrow{\cong} & H^p_{(2)}(\overline{K}).
  \label{de_Rham_isomorphism}
\end{eqnarray}
from the space of harmonic $L^2$-integrable $p$-forms on $\overline{M}$ to the
$L^2$-cohomology of the free simplicial $G$-complex $\overline{K}$. It is
essentially given by integrating a $p$-form over a $p$-simplex and is an
isomorphism of finitely generated Hilbert $\caln(G)$-modules. For more details
we refer to~\cite{Dodziuk(1977)} or~\cite[Theorem~1.59 on page~52]{Lueck(2002)}.

There is the de Rham cochain map (for large enough fixed $k$)
(see~\cite{Dodziuk(1977)} or~\cite[(1.77) on page~61]{Lueck(2002)})
\begin{eqnarray}
  A^* \colon H^{k-*}\Omega^p(\overline{M}) & \to & C^*_{(2)}(\overline{K})
  \label{de_Rham_chain_map}
\end{eqnarray}
where $H^{k-*}\Omega^p(\overline{M}) $ denotes the Sobolev space of $p$-forms on
$\overline{M}$.

\begin{lemma} \label{lem:bound_for_de_Rham} Assume that for every simplex
  $\sigma$ of $K$ we can find a neighborhood $V_{\sigma}$ together with a
  diffeomorphism $\eta_{\sigma} \colon \IR^m \to V_{\sigma}$.  (This can be
  arranged by possibly passing to a $d$-fold barycentric subdivision of $K$.)  Fix
  an integer $p$ with $0 \le p \le \dim(M)$.

  Then there exist constants $L_1,L_2 > 0$, which depend on data coming from $M$
  and $K$, but  do not depend on $G$ and $\pr \colon \overline{M} \to M$, such
  that for every $\overline{\omega} \in H^{k-p}\Omega^p(\overline{M})$ we have
\[L_1\cdot ||\overline{\omega}||_{H^{k-p}} \le ||A^p(\overline{\omega})||_{L^2} 
\le L_2 \cdot ||\overline{\omega}||_{H^{k-p}},
\]
and we get for the operator norm of the operator  
$\Int^p$ of~\eqref{de_Rham_isomorphism}
\[
L_1 \le ||\Int^p|| \le L_2.
\]
  
\end{lemma}
\begin{proof}
  Dodziuk~\cite[Lemma~3.2]{Dodziuk(1977)} proves for a given $G$-covering 
  $\pr   \colon \overline{M} \to M$ and $p \ge 0$ that the map $A^p \colon H^{k-p}
  \Omega^p(\overline{M}) \xrightarrow{\cong} C^p_{(2)}(\overline{K})$ is
  bounded, i.e., there exists a constant $L_2$ such that for $\overline{\omega}
  \in H^{k-p}\Omega^p(\overline{M})$ we have
\[
||A^p(\overline{\omega})||_{L^2} \le L_2 \cdot ||\overline{\omega}||_{H^{k-p}}.
\]
Next we will analyze Dodziuk's proof and explain why the constant $L_2$ depends
only on data coming from $M$ and $K$ and does not depend on $G$ and $\pr \colon
\overline{M} \to M$.

For every $p$-simplex $\sigma$ in $K$ we choose a relatively compact
neighborhood $U_{\sigma}$ of $\sigma$ with $U_{\sigma} \subseteq V_{\sigma}$.
Choose $N$ to be an integer such that every point $x \in M$ belongs to at most
$N$ of the sets $U_{\sigma}$, where $\sigma$ runs through the $p$-simplices of
$K$, e.g., take $N$ to be the number of $p$-simplices of $K$.
 We can apply~\cite[Lemma~3.1]{Dodziuk(1977)} to $\sigma \subseteq
V_{\sigma}$ and obtain a constant $C_{\sigma}> 0$ such that for any $p$-form
$\omega$ in $H^{k-p} \Omega^p(M)$ and any $p$-simplex $\sigma$ of $K$
\begin{eqnarray*}
  \sup_{x\in \sigma} |\omega(x)| 
  & \le &
  C_{\sigma}\cdot \bigl(||\omega||^{U_{\sigma}}_{H^{k-p}} 
  + ||\omega||^{U_{\sigma}}_{L^2}\bigr)
\end{eqnarray*}
holds.  The real number $|\omega(x)|$ is the norm of $\omega(x)$ as an element in
$\Lambda^pT_{x}^*M$, and $||\omega||^{U_{\sigma}}_{H^{k-p}}$ and
$||\omega||^{U_{\sigma}}_{L^2}$ are the Sobolev norm and $L^2$-norm of $\omega$
restricted to $U_{\sigma}$.

Let $C$ be the maximum of the numbers $C_{\sigma}$, where $\sigma$ runs through
all $p$-simplices of $K$.  Let $E$ be the maximum over the volumes of the
$p$-simplices of $K$.  Obviously the numbers $C$, $E$ and $N$ depend only on data
coming from $M$ and $K$, but do not depend on $G$ and $\pr \colon \overline{M}
\to M$.

Since $V_{\sigma}$ is contractible, the restriction of $\pr \colon \overline{M}
\to M$ to $\pr^{-1}(V_{\sigma})$ is trivial and hence $\pr^{-1}(V_{\sigma})$ is
$G$-diffeomorphic to $G \times V_{\sigma}$. Hence there are for every
$p$-simplex $\overline{\sigma} \in \overline{M}$ open neighborhoods
$U_{\overline{\sigma}}$ and $V_{\overline{\sigma}}$ which are uniquely
determined by the property that they are mapped diffeomorphically under 
$\pr \colon \overline{M} \to M$ onto $U_{\pr(\overline{\sigma})}$ and
$V_{\pr(\overline{\sigma})}$.  Notice for the sequel that
$\pr|_{V_{\overline{\sigma}}} \colon V_{\overline{\sigma}} \to
V_{\pr(\overline{\sigma})}$ and $\pr|_{U_{\overline{\sigma}}} \colon
U_{\overline{\sigma}} \to U_{\pr(\overline{\sigma})}$ are isometric
diffeomorphisms. Hence we get for every $p$-form $\overline{\omega}$ in
$H^{k-p} \Omega^p(\overline{M})$ and any $p$-simplex $\overline{\sigma}$ of
$\overline{K}$
\begin{eqnarray*}
  \sup_{\overline{x}\in \overline{\sigma}} |\omega(\overline{x})| 
  & \le &
  C \cdot \bigl(||\overline{\omega}||^{U_{\overline{\sigma}}}_{H^{k-p}} 
  + ||\overline{\omega}||^{U_{\overline{\sigma}}}_{L^2}\bigr).
\end{eqnarray*}
Here the real number $|\overline{\omega}(\overline{x})|$ is the norm of
$\overline{\omega}(\overline{x})$ as an element in
$\Lambda^pT_{\overline{x}}^*\overline{M}$, and
$||\overline{\omega}||^{V_{\overline{\sigma}}}_{H^{k-p}}$ and
$||\omega||^{V_{\overline{\sigma}}}_{L^2}$ are the Sobolev norm and $L^2$-norm
of the restriction of $\overline{\omega}$ to $V_{\overline{\sigma}}$.  One easily
checks that every point $\overline{x} \in \overline{M}$ belongs to at most $N$
of the sets $U_{\overline{\sigma}}$, where $\overline{\sigma}$ runs through the
$p$-simplices of $\overline{K}$ and that the volume of every $p$-simplex of
$\overline{K}$ is bounded by $E$.  Put
\[L_2 := \sqrt{4 \cdot C^2 \cdot E^2 \cdot N}.
\]
Then $L_2$ depends only on data coming from $M$ and $K$, but does not depend on
$G$ and $\pr \colon \overline{M} \to M$.

Next we perform essentially the same calculation as
in~\cite[Lemma~3.2]{Dodziuk(1977)}. We estimate for a $p$-simplex
$\overline{\sigma}$ of $\overline{K}$ and an element $\overline{\omega} \in H^{k-p}
\Omega^p(\overline{M})$
\begin{eqnarray*}
  \left(\int_{\overline{\sigma}} \overline{\omega}\right)^2  
  & \le & 
  \left(\sup_{\overline{x} \in \overline{\sigma}} \bigl\{|\overline{\omega}(\overline{x})|\bigr\} 
\cdot \vol(\overline{\sigma})\right)^2
  \\
  & \le &
  \left(C \cdot \bigl(||\overline{\omega}||^{U_{\overline{\sigma}}}_{H^{k-p}} 
    + ||\overline{\omega}||^{U_{\overline{\sigma}}}_{L^2}\bigr) \cdot  \vol(\overline{\sigma})\right)^2
  \\
  & \le & 
  C^2 \cdot 2 \cdot \left(\bigl(||\overline{\omega}||^{U_{\overline{\sigma}}}_{H^{k-p}}\bigr)^2
    + \bigl(\bigl||\overline{\omega}||^{U_{\overline{\sigma}}}_{L^2}\bigr)^2\right) \cdot E^2.
\end{eqnarray*}
This implies for $\overline{\omega} \in H^{k-p} \Omega^p(\overline{M})$, where
$\overline{\sigma}$ runs through the $p$-simplices of $\overline{K}$,
\begin{eqnarray*}
  \sum_{\overline{\sigma}}\left(\int_{\overline{\sigma}} \overline{\omega}\right)^2 
  & \le &
  \sum_{\overline{\sigma}} C^2 \cdot 2 \cdot 
  \left(\bigl(||\overline{\omega}||^{U_{\overline{\sigma}}}_{H^{k-p}}\bigr)^2
    + \bigl(\bigl||\overline{\omega}||^{U_{\overline{\sigma}}}_{L^2}\bigr)^2\right) \cdot E^2
  \\
  & \le &
  2 \cdot C^2 \cdot E^2 \cdot \left(\sum_{\overline{\sigma}} 
\left(\bigl(||\overline{\omega}||^{U_{\overline{\sigma}}}_{H^{k-p}}\bigr)^2
      + \bigl(||\overline{\omega}||^{U_{\overline{\sigma}}}_{L^2}\bigr)^2\right)\right)
  \\
  & \le &
  2 \cdot C^2 \cdot E^2 \cdot N \cdot  \left(\bigl(||\overline{\omega}||_{H^{k-p}}\bigr)^2
    + \bigl(\bigl||\overline{\omega}||_{L^2}\bigr)^2\right)
  \\
  & \le & 2 \cdot C^2 \cdot E^2 \cdot N \cdot 2 \cdot ||\overline{\omega}||_{H^{k-p}}^2
  \\
  & = & 
  L_2^2 \cdot ||\overline{\omega}||_{H^{k-p}}^2.
\end{eqnarray*}
We conclude that the de Rham map
\[
A^p \colon H^{k-p}\Omega^p(\overline{M}) \to C^p_{(2)}(\overline{K})
\]
is a bounded operator whose norm is less than or equal to $L_2$.

Dodziuk~\cite[Lemma~3.7]{Dodziuk(1977)} (see also~\cite[(1.78) on
page~61]{Lueck(2002)}) constructs a bounded $G$-equivariant operator
\begin{eqnarray}
  W^p \colon C^p_{(2)}(\overline{K})  & \to & H^{k-p}\Omega^p(\overline{M}),
  \label{Whitney_map}
\end{eqnarray}
and gives an upper bound for its operator norm by a number
\[
\left|\bigl\{p\text{-simplices of}\, K\bigr\}\right| \cdot
\max\big\{||W^p\sigma||_{H^{k-p}}\mid \sigma \, p\text{-simplex of}\, K\bigr\}.
\]
Define
\[
L_1 := \frac{1}{\left|\bigl\{p\text{-simplices of}\, K\bigr\}\right| 
  \cdot \max\big\{||W\sigma||_{H^{k-p}}\mid \sigma \, p\text{-simplex of}\,
  K\bigr\}}.
\] 
Obviously $L_1$ depends only on data coming from $M$ and $K$, but
not on $G$ and $\pr \colon \overline{M} \to M$.  The maps $A^p$
of~\eqref{de_Rham_chain_map} and $W^p$ of~\eqref{Whitney_map} induce bounded
$G$-operators (see~\cite[Corollary on page~162 and Corollary on
page~163]{Dodziuk(1977)})
\begin{eqnarray*}
  H^p_{(2)}(A^*) \colon H^p_{(2)}\bigl(H^{k-*}\Omega^*(\overline{M})\bigr) 
  & \to & 
  H^p_{(2)}\bigl(C^*_{(2)}(\overline{K})\bigr);
  \\
  H^p_{(2)}(W^*) \colon H^p_{(2)}\bigl(C^*_{(2)}(\overline{K})\bigr) 
  & \to &
  H^p_{(2)}\bigl(H^{k-*}\Omega^*(\overline{M})\bigr),
\end{eqnarray*}
such that we obtain for their operator norms
\begin{eqnarray*}
  \bigl|\bigl|H^p_{(2)}(A^*)\bigr|\bigr|  & \le & L_2;
  \\
  \bigl|\bigl|H^p_{(2)}(W^*)\bigr|\bigr|  & \le & \frac{1}{L_1}.
\end{eqnarray*}
Since $W^p \circ A^p = \id$ and $H^p_{(2)}(A^p)$ is an isomorphism
(see~\cite[(3.6), Lemma~3.8 and Lemma~3.10]{Dodziuk(1977)} and~\cite[(1.79)
and~(1.80) on page~61]{Lueck(2002)}, the map $H^p_{(2)}(W^*)$ is the inverse of
$H^p_{(2)}(A^*)$. This implies
\[L_1 \le \bigl|\bigl|H^p_{(2)}(A^*)\bigr|\bigr|   \le  L_2.
\]

Since the canonical inclusion
\[
i^p \colon \calh^p_{(2)} \xrightarrow{\cong} H^p_{(2)}\bigl(H^{k-*}\Omega^*(\overline{M})\bigr)
\]
is an isometric $G$-isomorphism (see~\cite[Lemma~1.75 on page~59]{Lueck(2002)})
and the map $\Int^p \colon \calh^p_{(2)} \to H^p_{(2)}(\overline{K})$ is the
composite $H^p_{(2)}(A^*)\circ i^p$, Lemma~\ref{lem:bound_for_de_Rham} follows.
\end{proof}

Now we are ready to give

\begin{proof}[Proof of Theorem~\ref{the:comparing_analytic_and_chain_complexes}]
 Let $\overline{M}$ be a Riemannian manifold without boundary that comes with a proper free cocompact isometric
$G$-action such that $b_p^{(2)}(\overline{M};\caln(G))$ vanishes for all $p \ge 0$.
Fix a smooth triangulation $K$ of 
$M = G \backslash \overline{M}$.  By possibly subdividing it we can arrange that
Lemma~\ref{lem:bound_for_de_Rham} will apply to $M$ and $K$.
The triangulation $K$ lifts to a $G$-equivariant smooth triangulation $\overline{K}$ of $\overline{M}$ and
to a $G/G_i$-equivariant smooth triangulation $\overline{K}[i]$ of $M[i] :=G_i \backslash \overline{M}$.

We assume that the Approximation Conjecture for $L^2$-torsion of chain
complexes~\ref{con:Approximation_conjecture_for_L2-torsion_of_chain_complexes}
is true. Hence we get
\begin{eqnarray}
  \rho^{(2)}\bigl(\overline{K};\caln(G)\bigr) & = & 
  \lim_{i \in I} \; \rho^{(2)}(\overline{K}[i];\caln(G/G_i)\bigr).
\end{eqnarray}
In the sequel we will use the theorem of Burghelea, Friedlander, Kappeler and
McDonald \cite{Burghelea-Friedlander-Kappeler-McDonald(1996a)} that 
the topological and the analytic $L^2$-torsion agree. Since $M$ and hence $\overline{K}$ is $L^2$-acyclic, we get from the
definitions
\begin{eqnarray*}
  \rho^{(2)}_{\an}\bigl(\overline{M};\caln(G)\bigr) 
  & = & 
  \rho^{(2)}\bigl(\overline{K};\caln(G)\bigr);
  \\
  \rho^{(2)}_{\an}\bigl(M[i];\caln(G/G_i)\bigr) 
  & = & 
  \rho^{(2)}\bigl(\overline{K}[i];\caln(G/G_i)\bigr)
  - \sum_{p \ge 0} (-1)^p \cdot {\det}_{\caln(G/G_i)}^{(2)}(\Int[i]^p),
\end{eqnarray*}
where $\Int[i]^p \colon \calh^p_{(2)}(M[i]) \xrightarrow{\cong}
H^p_{(2)}(\overline{K}[i])$ is the $L^2$-de Rham isomorphism
of~\eqref{de_Rham_isomorphism}.  Hence it suffices to show for $p \ge 0$
\begin{eqnarray}
  \lim_{i \in I} \;\ln\bigl({\det}_{\caln(G/G_i)}^{(2)}(\Int[i]^p)\bigr) & = & 0.
  \label{lim_det_de_Rham_is_zero}
\end{eqnarray}
We obtain from Lemma~\ref{lem:bound_for_de_Rham} constants $L_1> 0$ and $L_2> 0$ which
are independent of $i \in I$ such that for every $i \in I$
\[
L_1 \le ||\Int[i]^p|| \le L_2
\]
holds for the operator norm of $\Int[i]^p$.  Since
\[
b_p^{(2)}\bigl(M[i];\caln(G/G_i)\bigr) = 
\dim_{\caln(G/G_i)}\bigl( \calh^p_{(2)}(M[i])\bigr) 
=
\dim_{\caln(G/G_i)}\bigl(H^p_{(2)}(K[i])\bigr),
\]
we conclude
\begin{eqnarray*}
\ln(L_1) \cdot b_p^{(2)}\bigl(M[i];\caln(G/G_i)\bigr) 
& \le &
\ln\bigl({\det}_{\caln(G/G_i)}^{(2)}(\Int[i]^p)\bigr) 
\\
& \le &
\ln(L_2) \cdot
b_p^{(2)}(M[i];\caln(G/G_i)\bigr).
\end{eqnarray*}
Since the Approximation
Conjecture~\ref{con:Approximation_conjecture_for_L2-Betti_numbers} holds for $\overline{M}$ by
Theorem~\ref{the:The_Determinant_Conjecture_implies_the_Approximation_Conjecture_for_L2-Betti_numbers}
and we have $b_p^{(2)}(\overline{M};\caln(G)\bigr) = 0$ for $p \ge 0$ by assumption, we have
\begin{eqnarray*}
  \lim_{i \in I}  \; b_p^{(2)}\bigl(M[i];\caln(G/G_i)\bigr) & = & 0.
\end{eqnarray*}
Now~\eqref{lim_det_de_Rham_is_zero} follows. This finishes the proof of
Theorem~\ref{the:comparing_analytic_and_chain_complexes}.
\end{proof}

\begin{remark}[On the $L^2$-acyclicity assumption]
\label{rem:On_the_L2-acyclicity_assumption}
Recall that in Theorem~\ref{the:comparing_analytic_and_chain_complexes}
we require that $b_p^{(2)}(\overline{M};\caln(G)\bigr) = 0$ holds for $p \ge 0$.  This
assumption is satisfied in many interesting cases.  It is possible that this
assumption is not needed for Theorem~\ref{the:comparing_analytic_and_chain_complexes}
to be true, but our proof does not
work without it. We can drop this assumption if we can
generalize~\eqref{lim_det_de_Rham_is_zero} to
\begin{eqnarray*}
    \lim_{i \in I} \; \ln\bigl({\det}_{\caln(G/G_i)}(\Int[i]^p)\bigr) & = & \ln\bigl({\det}_{\caln(G)}(\Int^p)\bigr).
  \end{eqnarray*}
\end{remark}


\typeout{------   Section 16: A general strategy to prove the Approximation Conjecture for Fuglede-Kadison determinants ---------}

\section{A strategy to prove the Approximation Conjecture for Fuglede-Kadison 
determinants~\ref{con:Approximation_conjecture_for_Fuglede-Kadison_determinants_with_arbitrary_index}}
\label{sec:A_general_strategy_to_prove_the_Approximation_Conjecture_for_Fuglede-Kadison_determinants}


\subsection{The general setup}
\label{subsec:The_general_setup}

Throughout this section we will consider the following data:

\begin{itemize}

\item $G$ is a group, $B$ is a matrix in $M_d(\caln(G))$ and
$\tr \colon M_d(\caln(G)) \to \IC$ is a faithful finite normal trace.

\item $I$ is a directed set. For each $i \in I$ we have a group $Q_i$, a matrix
$B[i] \in M_d(\caln(Q_i))$ and a faithful finite normal trace $\tr_i\colon M_d(\caln(Q_i)) \to
\IC$  such that $\tr_i(I_d) = d$ holds for the unit matrix $I_d \in M_d(\caln(Q_i))$.

\end{itemize}

Faithful finite normal trace $\tr$ 
means that $\tr$ is $\IC$-linear, satisfies
$\tr(B_1B_2) = \tr_i(B_2B_1)$, sends $B^*B$ to a real number $\tr(B^*B) \ge
0$ such that $\tr(B^*B) = 0 \Leftrightarrow B = 0$, and for $f \in \caln(G)$, which is the supremum with respect to the usual
ordering $\le$ of positive elements of some monotone increasing net $\{f_j \mid
j \in J\}$ of positive elements in $\caln(G)$, we get $\tr(f) = \sup\{\tr(f_j)
\mid j \in J\}$. The trace $\tr$ or $\tr_i$ respectively may or may not be the von Neumann trace
(see~\cite[Definition~1.2 on page~15]{Lueck(2002)}) 
$\tr_{\caln(G)}$ or $\tr_{\caln(Q_i)}$ respectively.

Let $F\colon [0,\infty) \to [0,\infty)$ be the spectral density function of
$r_B^{(2)}\colon L^2(G)^d \to L^2(G)^d$ with respect to $\tr$ as defined in~\cite[Definition~2.1 on page~73]{Lueck(2002)}.
Let $F[i]\colon [0,\infty) \to [0,\infty)$ be the spectral density function of 
$r_{B[i]}^{(2)}\colon L^2(Q_i)^d \to L^2(Q_i)^d$  with
respect to the trace $\tr_i$. If $r_B^{(2)}$ is
positive, we get $F(\lambda) = \tr(E_{\lambda})$ for $\{E_{\lambda} \mid \lambda
\in \IR\}$ the family of spectral projections of $r_B^{(2)}$
(see~\cite[Lemma~2.3 on page~74 and Lemma~2.11~(11) on page~77]{Lueck(2002)}).
The analogous statement holds for $F[i]$.

Recall that for a directed set $I$ and a net $(x_i)_{i \in I}$ of real numbers
one defines
\begin{eqnarray}
\liminf_{i \in I} x_i%
 & := & \sup\{\inf\{x_j\mid j \in I, j \ge
i\}\mid i \in I\};
\label{liminf}
\\
\limsup_{i \in I} x_i%
& := & \inf\{\sup\{x_j\mid j \in I, j \ge i\}\mid i \in I\}.
\label{limsup}
\end{eqnarray}


\subsection{The general strategy}
\label{subsec:The_general_strategy}

To any monotone non-decreasing function
$f\colon [0,\infty) \to [0,\infty)$ we can assign a density function,
i.e., a monotone non-decreasing right-continuous function,
\begin{eqnarray*}
f^+\colon [0,\infty) \to [0,\infty), & & \hspace{5mm}\lambda \mapsto 
\lim_{\epsilon \to 0+} f(\lambda + \epsilon).
\end{eqnarray*}
Put
\begin{eqnarray*}
\underline{F}(\lambda) & := & \liminf_{i \in I} F[i](\lambda);
\\
\overline{F}(\lambda) & := & \limsup_{i \in I} F[i](\lambda).
\end{eqnarray*}

Let $\det$ and $\det_i$ be the Fuglede-Kadison determinant with respect to $\tr$ and $\tr_i$
(compare~\cite[Definition~3.11 on page~127]{Lueck(2002)}).
If $\tr$ is the von Neumann trace $\tr_{\caln(G)}$, then $\det$ is the Fuglede-Kadison determinant
$\det_{\caln(G)}$ as defined in~\cite[Definition~3.11 on page~127]{Lueck(2002)}.  We want to prove

\begin{theorem} \label{the:approxi_for_spectral_density}
Consider the following conditions, where $K > 0$ and $\kappa > 0$ are some fixed real numbers:

\renewcommand{\theenumi}{\roman{enumi}}

\begin{enumerate}

\item \label{the:approxi_for_spectral_density:uniform_bound_K}
The operator norms satisfy $||r_B^{(2)}|| \le K$ and $||r_{B[i]}^{(2)}|| \le K$ for all $i \in I$;

\item \label{the:approxi_for_spectral_density:traces_and_polynomials}
For every polynomial $p$ with real coefficients we have
\[\tr(p(B)) =  \lim_{i \in I} \;\tr_i(p(B[i]));
\]

\item \label{the:approxi_for_spectral_density:det}
We have ${\det}_i\bigl(r_{B[i]}^{(2)} \colon L^2(Q_i)^d \to L^2(Q_i)^d\bigr) \ge \kappa$ 
for each $i \in I$,

\item \label{the:approxi_for_spectral_density:positivity}
Suppose $r_B^{(2)}\colon L^2(G)^d \to L^2(G)^d$ and
$r_{B[i]}^{(2)}\colon L^2(Q_i)^d \to L^2(Q_i)^d$ for $i \in I$ are positive;

\item \label{the:approxi_for_spectral_density:integrability}
The \emph{uniform integrability condition} is satisfied, i.e., 
there exists $\epsilon > 0$ such that 
\begin{eqnarray*}
\int_{0+}^{\epsilon} \left. \sup \left\{\frac{F[i](\lambda) - F[i](0)}{\lambda} 
\, \right|\,i \in I \right\}\; d\lambda 
& <  & \infty.
\end{eqnarray*}

\end{enumerate}

\renewcommand{\theenumi}{\arabic{enumi}}
Then:

\begin{enumerate}

\item \label{the:approxi_for_spectral_density:inequality}
If conditions~\eqref{the:approxi_for_spectral_density:uniform_bound_K},~%
\eqref{the:approxi_for_spectral_density:traces_and_polynomials}~%
\eqref{the:approxi_for_spectral_density:det},
and~\eqref{the:approxi_for_spectral_density:positivity} are satisfied, then
\[
\det\bigl(r_B^{(2)} \colon L^2(G)^d \to L^2(G)^d\bigr) 
\ge  \limsup_{i \in I} {\det}_i\bigl(r_{B[i]}^{(2)}\colon L^2(Q_i)^d \to L^2(Q_i)^d\bigr);
\]

\item \label{the:approxi_for_spectral_density:equality}
If conditions~\eqref{the:approxi_for_spectral_density:uniform_bound_K},~%
\eqref{the:approxi_for_spectral_density:traces_and_polynomials}~%
\eqref{the:approxi_for_spectral_density:det},~%
\eqref{the:approxi_for_spectral_density:positivity} 
and~\eqref{the:approxi_for_spectral_density:integrability} are satisfied, then
\[
\det\bigl(r_B^{(2)} \colon L^2(G)^d \to L^2(G)^d\bigr) 
= \lim_{i \in I} {\det}_i\bigl(r_{B[i]}^{(2)}\colon L^2(Q_i)^d \to L^2(Q_i)^d\bigr).
\]
\end{enumerate}
\end{theorem}
\begin{proof}
Completely analogously to the proof of~\cite[Theorem~13.19 on page~461]{Lueck(2002)} we prove
\begin{eqnarray}
F(\lambda) 
& = & 
\underline{F}^+(\lambda) \;= \; \overline{F}^+(\lambda) \quad \text{for} \; \lambda \in \IR;
\label{F[i]s_underlineFplus_isoverlineFplus}
\\
F(0) 
& = & 
\lim_{i \in I} F[i](0);
\label{F(0)_is_lim_i_F[i](0)}
\\
\kappa &\le &
\det\bigl(r_B^{(2)} \colon L^2(G)^d \to L^2(G)^d\bigr).
\label{ln(det(R_B)_greater_or_equal_zero}
\end{eqnarray}
The proof of~\cite[(13.22) and~(13.23) on page~462]{Lueck(2002)} carries directly over and yields
\begin{eqnarray}
\quad \quad \ln(\det(r_B^{(2)})) 
& = & 
\ln(K)\cdot (F(K) - F(0)) - \int_{0+}^K
\frac{F(\lambda) - F(0)}{\lambda} \, d\lambda; \label{det(r_B)_is_int}
\\
\quad \quad \quad \ln({\det}_i(r_{B[i]}^{(2)})) 
& = & 
\ln(K)\cdot (F[i](K) - F[i](0)) - \int_{0+}^K
\frac{F[i](\lambda) - F[i](0)}{\lambda} \, d\lambda. 
\label{det(r_B[i])_is_int}
\end{eqnarray}

We conclude from~\eqref{ln(det(R_B)_greater_or_equal_zero} and~\eqref{det(r_B)_is_int}
\begin{eqnarray}
0 & \le & \int_{0+}^K\frac{F(\lambda) - F(0)}{\lambda} d\lambda < + \infty.
\label{0_le_int_F/lambda}
\end{eqnarray}

Since  $\underline{F}$ and  $\overline{F}$  are monotone  increasing
bounded functions,  there are only countably  many elements $\lambda \in    (0,\infty)$   
such    that    $\underline{F}(\lambda)   \not= \underline{F}^+(\lambda)$  
or     $\overline{F}(\lambda)     \not= \overline{F}^+(\lambda)$  hold. 
We conclude from~\eqref{F[i]s_underlineFplus_isoverlineFplus} that  there  is a  countable set  $S
\subseteq  [0,\infty)$  such  that  for all  $\lambda  \in  [0,\infty)
\setminus S$  the limit $\lim_{i \to \infty}  F[i](\lambda)$ exists and
is equal to $F(\lambda)$. Since $S$ has measure zero and we have~\eqref{F(0)_is_lim_i_F[i](0)},
we get  almost everywhere for $\lambda \in [0,\infty)$
\begin{eqnarray}
\lim_{i \in I} \frac{F[i](\lambda) - F[i](0)}{\lambda} 
& = & 
\frac{F(\lambda) - F(0)}{\lambda}
\label{lim_isF_almost_everywhere}
\end{eqnarray}

Analogously to the proof of~\cite[(13.28) on page~463]{Lueck(2002)} one shows
\begin{eqnarray*}
\int_{0+}^K\frac{\lim_{i \in I} \bigl(F[i](\lambda) - F[i](0)\bigr)}{\lambda} d\lambda
& \le  & 
\liminf_{i \in I}  \int_{0+}^K \frac{F[i](\lambda) - F[i](0)}{\lambda} d\lambda.
\end{eqnarray*}
This implies
\begin{eqnarray}
\int_{0+}^K\frac{F(\lambda) - F(0)}{\lambda} d\lambda
& \le  & 
\liminf_{i \in I}  \int_{0+}^K \frac{F[i](\lambda) - F[i](0)}{\lambda} d\lambda.
\label{int_F/lambda_le_int_liminf}
\end{eqnarray}

Now assertion~\eqref{the:approxi_for_spectral_density:inequality} follows 
from~\eqref{det(r_B)_is_int},~\eqref{det(r_B[i])_is_int}, and~\eqref{int_F/lambda_le_int_liminf}.

Next we prove assertion~\eqref{the:approxi_for_spectral_density:equality}. 
We can apply Lebesgue's Dominated Convergence Theorem to~\eqref{lim_isF_almost_everywhere}
because of the assumption~\eqref{the:approxi_for_spectral_density:integrability} and obtain
\begin{eqnarray*}
\int_{0+}^K\frac{\lim_{i \in I} \bigl(F[i](\lambda) - F[i](0)\bigr)}{\lambda} d\lambda
& =  & 
\lim_{i \in I}  \int_{0+}^K \frac{F[i](\lambda) - F[i](0)}{\lambda} d\lambda.
\end{eqnarray*}
Now assertion~\eqref{the:approxi_for_spectral_density:equality} follows 
from~\eqref{det(r_B)_is_int}, and~\eqref{det(r_B[i])_is_int}. 
This finishes the proof of Theorem~\ref{the:approxi_for_spectral_density}.
\end{proof}


\subsection{The uniform integrability condition is not automatically satisfied}
\label{subsec:The_uniform_integrability_condition_is_not_automatically_satisfied}

The main difficulty to apply Theorem~\ref{the:approxi_for_spectral_density} to
the situations of interest is the verification of  the uniform integrability
condition~\eqref{the:approxi_for_spectral_density:integrability} appearing in
Theorem~\ref{the:approxi_for_spectral_density}.  We will illustrate by an
example that one needs extra input to ensure this condition since there are
examples where this condition is violated but all properties of the spectral
density functions which are known so far are satisfied.

Define the following sequence of functions $f_n \colon [0,1] \to [0,1]$
\[
f_n(\lambda) = 
\begin{cases}
\lambda 
& 
0 \le \lambda \le e^{-3n};
\\
\frac{(e^{-2n} -\lambda) \cdot e^{-3n}  + (\lambda- e^{-3n}) \cdot \bigl(\frac{1}
{-\ln(e^{-2n})} + e^{-2n}\bigr)}{e^{-2n}- e^{-3n}}
& 
e^{-3n} \le \lambda \le e^{-2n};
\\
\frac{1}{-\ln(\lambda)} + \lambda 
& 
e^{-2n} \le \lambda \le e^{-n};
\\
\frac{1}{-\ln(e^{-n})} + e^{-n} 
& 
e^{-n} \le \lambda \le \frac{1}{-\ln(e^{-n})} + e^{-n};
\\
\lambda 
&  
\frac{1}{-\ln(e^{-n})} + e^{-n} \le \lambda \le 1.
\end{cases}
\]

\begin{lemma}
\label{lem:addendum}\
\begin{enumerate}

\item \label{lem:addendum:(1)}
The function $f_n(\lambda)$ is monotone non-decreasing and continuous for $n \ge 1$;

\item \label{lem:addendum:(2)}
$f_n(0) = 0$ and $f_n(1) = 1$ for all $n \ge 1$;

\item \label{lem:addendum:(3)}
$\lim_{n \to \infty} f_n(\lambda) = \lambda$ for $\lambda \in [0,1]$;

\item \label{lem:addendum:(4)}
We have for all $n \ge 1$ and $\lambda \in [0,1)$
\[
\lambda \le f_n(\lambda) \le \frac{1}{-\ln(\lambda)} + \lambda  \le \frac{2}{-\ln(\lambda)}; 
\]
\item \label{lem:addendum:(5)}
We have for  $\lambda \in [0,e^{-1}]$
\[
\sup\bigl\{f_n(\lambda) \mid n \ge 0\} = \frac{1}{-\ln(\lambda)} + \lambda;
\] 
\item \label{lem:addendum:(6)}
We have
\[
\int_{0+}^1 \frac{\sup\{f_n(\lambda) \mid n \ge 0\}}{\lambda} \; d\lambda = \infty;
\]
\item  \label{lem:addendum:(7)}
We get for all $n \ge 1$
\[
\int_{0+}^1 \frac{f_n(\lambda)}{\lambda} \; d\lambda
\ge 
\ln(2) + 1;
\]
\item \label{lem:addendum:(8)}
We have
\[
\int_{0+}^1 \lim_{n \to \infty} \frac{ f_n(\lambda)}{\lambda} d \lambda 
< 
\liminf_{n \to \infty}
\int_{0+}^1 \frac{f_n(\lambda)}{\lambda} d \lambda
\le
\limsup_{n \to \infty}
\int_{0+}^1 \frac{f_n(\lambda)}{\lambda} d \lambda;
\]
\item \label{lem:addendum:(9)} 
We get for all $n \ge 1$
\[\int_{0+}^1 \frac{f_n(\lambda)}{\lambda} \,d\lambda 
\le 4.
\]
\end{enumerate}
\end{lemma}
\begin{proof}~%
\eqref{lem:addendum:(1)}
One easily checks that the definition of $f_n$ makes sense,
in particular at the values $\lambda = e^{-3n}, e^{-2n}, e^{-n}$.

The first derivative of $f_n(\lambda)$ exists with the exception of $\lambda = e^{-3n}, e^{-2n}, e^{-n},
\frac{1}{-\ln(e^{-n})} + e^{-n}$ and is given by
\[
f_n'(\lambda) = 
\begin{cases}
1 
& 
0 \le \lambda < e^{-3n};
\\
1  +  \frac{1}{2n \cdot (e^{-2n}- e^{-3n})}
& 
e^{-3n} < \lambda <  e^{-2n};
\\
\frac{1}{\lambda \cdot \ln(\lambda)^2} + 1
& 
e^{-2n} <\lambda < e^{-n};
\\
0
& 
e^{-n} \le \lambda \le \frac{1}{-\ln(e^{-n})} + e^{-n};
\\
1
&  
\frac{1}{-\ln(e^{-n})} + e^{-n} <  \lambda \le 1.
\end{cases}
\]
Hence for $n \ge 1$ the function $f_n$ is smooth with non-negative derivative
on the open intervals $(0,e^{-3n})$,
$(e^{-3n},e^{-2n})$, $(e^{-2n},e^{-n})$, and $(e^{-n},1)$,
and is continuous on the closed intervals $[0,e^{-3n}]$,
$[e^{-3n},e^{-2n}]$, $[e^{-2n},e^{-n}]$, and $[e^{-n},1]$.
Hence each $f_n$ is continuous and monotone non-decreasing.
\\[1mm]~%
\eqref{lem:addendum:(2)} This is obvious.
\\[1mm]~%
\eqref{lem:addendum:(3)} This follows since
$\lim_{n \to \infty} \frac{1}{-\ln(e^{-n})} + e^{-n} = 0$,
$f_n(\lambda) = \lambda$ for $\frac{1}{-\ln(e^{-n})} + e^{-n} \le \lambda \le 1$
and $f_n(0) = 0$ for all $n \ge 1$.
\\[1mm]~%
\eqref{lem:addendum:(4)}
We conclude $\lambda \le f_n(\lambda) \le \frac{1}{-\ln(\lambda)} + \lambda$ for
$\lambda \in [0,1)$ by inspecting the definitions since
$\lambda \le \frac{1}{-\ln(\lambda)} + \lambda$ holds and
$\frac{1}{-\ln(\lambda)} + \lambda$ is monotone non-decreasing.
We have $\lambda \le \frac{1}{-\ln(\lambda)}$ for $\lambda \in (0,1)$.
\\[1mm]~%
\eqref{lem:addendum:(5)}
From assertion~\eqref{lem:addendum:(4)} we conclude
$\sup\bigl\{f_n(\lambda) \mid n \ge 0\} \le \frac{1}{-\ln(\lambda)} + \lambda$ for  
$\lambda \in [0,1)$.
Since for $\lambda$ with $0 < \lambda \le e^{-1}$ we can find
$n \ge 1$ with $e^{-2n} \le \lambda \le e^{-n}$ and hence 
$f_n(\lambda) = \frac{1}{-\ln(\lambda)} + \lambda$ holds for that $n$, we conclude
$\sup\bigl\{f_n(\lambda) \mid n \ge 0\} = \frac{1}{-\ln(\lambda)} + \lambda$ for  
$\lambda \in [0,e^{-1}]$.
\\[1mm]~%
\eqref{lem:addendum:(6)}
We compute using assertion~\eqref{lem:addendum:(5)} for every $\epsilon \in (0,e^{-1})$
\begin{eqnarray*}
\int_{0+}^1 \frac{\sup\bigl\{f_n(\lambda) \mid n \ge 0\}}{\lambda} \; d\lambda 
& \ge &
\int_{\epsilon}^{e^{-1}} \frac{\sup\bigl\{f_n(\lambda) \mid n \ge 0\}}{\lambda} \; d\lambda 
\\
& = & 
\int_{\epsilon}^{e^{-1}} 1  + \frac{1}{\lambda\cdot (-\ln(\lambda))} \; d\lambda 
\\
& = & e^{-1} - \epsilon + \left[-\ln(-\ln(\lambda))\right]_{\epsilon}^{e^{-1}}
\\
& = & 
e^{-1} - \epsilon  - \ln(-\ln(e^{-1}) + \ln(-\ln(\epsilon))
\\
& = & 
e^{-1} - \epsilon + \ln(-\ln(\epsilon))
\\
& \ge &
\ln(-\ln(\epsilon)).
\end{eqnarray*}
Since
$\lim_{\epsilon \to 0+} \ln(-\ln(\epsilon))  =  \infty$, 
assertion~\eqref{lem:addendum:(6)} follows.
\\[1mm]~%
\eqref{lem:addendum:(7)}
We estimate for given $n \ge 1$ using the conclusion $f_n(\lambda) - \lambda \ge 0$ for $\lambda \in [0,1]$ 
from assertion~\eqref{lem:addendum:(4)}
\begin{eqnarray*}
\int_{0+}^1 \frac{f_n(\lambda)}{\lambda} \; d\lambda 
& = & 
\int_{0+}^1 \frac{f_n(\lambda) - \lambda }{\lambda} \; d\lambda 
 + 1
\\
& \ge &
\int_{e^{-2n}}^{e^{-n}} \frac{f_n(\lambda)- \lambda}{\lambda} \; d\lambda + 1
\\
& = & 
\int_{e^{-2n}}^{e^{-n}} \frac{1}{\lambda \cdot (-\ln(\lambda))} \; d \lambda + 1
\\
& = & 
\left[-\ln(-\ln(\lambda))\right]_{e^{-2n}}^{e^{-n}} + 1
\\
& = & 
-\ln(-\ln(e^{-n}) + \ln(-\ln(e^{-2n})) + 1
\\
& = & 
-\ln(n) + \ln(2n) + 1
\\
& = & 
\ln(2) + 1.
\end{eqnarray*}
\eqref{lem:addendum:(8)}
This follows from assertion~\eqref{lem:addendum:(3)},~\eqref{lem:addendum:(7)}.
\\[1mm]~%
\eqref{lem:addendum:(9)}
We estimate
\begin{eqnarray*}
\lefteqn{\int_{0+}^1 \frac{f_n(\lambda)}{\lambda} \,d\lambda}
& & 
\\
& = &
\int_{0+}^{e^{-3n}} \frac{f_n(\lambda)}{\lambda} \,d\lambda
+ \int_{e^{-3n}}^{e^{-2n}} \frac{f_n(\lambda)}{\lambda} \, d\lambda
+ \int_{e^{-2n}}^{e^{-n}} \frac{f_n(\lambda)}{\lambda} \, d\lambda
\\
& &  \hspace{30mm}
+ \int_{e^{-n}}^{\frac{1}{-\ln(e^{-n})} + e^{-n}} \frac{f_n(\lambda)}{\lambda} \,d\lambda
+ \int_{\frac{1}{-\ln(e^{-n})} +  e^{-n}}^1 \frac{f_n(\lambda)}{\lambda} \,d\lambda
\\
& = & 
\int_{0+}^{e^{-3n}} \frac{\lambda}{\lambda} \,d\lambda
+ \int_{e^{-3n}}^{e^{-2n}} \frac{(e^{-2n} -\lambda) \cdot e^{-3n}  + (\lambda- e^{-3n}) \cdot \bigl(\frac{1}
{-\ln(e^{-2n})} + e^{-2n}\bigr)}{e^{-2n}- e^{-3n}}\cdot \frac{1}{\lambda} \, d\lambda
\\
& &  \hspace{5mm}
+ \int_{e^{-2n}}^{e^{-n}} \frac{\frac{1}{-\ln(\lambda)} + \lambda}{\lambda} \, d\lambda
+ \int_{e^{-n}}^{\frac{1}{n} + e^{-n}} 
\frac{\frac{1}{-\ln(e^{-n})} + e^{-n}}{\lambda} \,d\lambda
+ \int_{\frac{1}{n} +  e^{-n}}^1 \frac{\lambda}{\lambda} \,d\lambda
\\
& = & 
\int_{0+}^{e^{-3n}} 1  \,d\lambda
+ \int_{e^{-3n}}^{e^{-2n}} 1  +  \frac{1}{2n \cdot (e^{-2n}- e^{-3n})} + \frac{e^{-3n} }{2n \cdot (e^{-2n}- e^{-3n})}\cdot \frac{1}{\lambda}  \, d\lambda
\\
& &  \hspace{5mm}
+ \int_{e^{-2n}}^{e^{-n}} 1 - \frac{1}{\ln(\lambda) \cdot \lambda} \, d\lambda
+ \int_{e^{-n}}^{\frac{1}{n} + e^{-n}} 
\frac{\frac{1}{n} + e^{-n}}{\lambda} \,d\lambda
+ \int_{\frac{1}{n} +  e^{-n}}^1 1 \,d\lambda
\\
& = & 
e^{-3n}
+ \int_{e^{-3n}}^{e^{-2n}} 1  +  \frac{1}{2n \cdot (e^{-2n}- e^{-3n})}  \, d\lambda 
+ \int_{e^{-3n}}^{e^{-2n}} \frac{e^{-3n} }{2n \cdot (e^{-2n}- e^{-3n})}\cdot \frac{1}{\lambda}  \, d\lambda
\\
& &  \hspace{5mm}
+ \int_{e^{-2n}}^{e^{-n}} 1 \, d\lambda -  \int_{e^{-2n}}^{e^{-n}} \frac{1}{\ln(\lambda) \cdot \lambda} \, d\lambda
+ \int_{e^{-n}}^{\frac{1}{n} + e^{-n}} 
\frac{\frac{1}{n} + e^{-n}}{\lambda} \,d\lambda
+ 1 - \frac{1}{n} -  e^{-n}
\\
& = & 
e^{-3n} +  \bigl(e^{-2n}- e^{-3n}\bigr) \cdot \left(1 + \frac{1}{2n \cdot (e^{-2n}- e^{-3n})}\right) 
+ \left[\frac{e^{-3n} \cdot \ln(\lambda)}{2n \cdot (e^{-2n}- e^{-3n})} \right]_{e^{-3n}}^{e^{-2n}} 
\\
& &  \hspace{5mm}
+ e^{-n} - e^{-2n} -  \left[\ln(-\ln(\lambda))\right]_{e^{-2n}}^{e^{-n}} 
+ 
\left[\left(\frac{1}{n} + e^{-n}\right) \cdot \ln(\lambda)\right]_{e^{-n}}^{\frac{1}{n} + e^{-n}} 
+ 1 - \frac{1}{n} -  e^{-n}
\\
& = & 
1 -   \frac{3}{2n}+ \frac{e^{-3n} \cdot (-2n + 3n)}{2n \cdot (e^{-2n}- e^{-3n})} 
\\
& &  \hspace{5mm}
 -  \bigl(\ln(n) - \ln(2n)\bigr)
+  \left(\frac{1}{n} + e^{-n}\right) \cdot \left(\ln\left(\frac{1}{n} + e^{-n} \right) - \ln(e^{-n})\right)
\\
& = & 
1 -   \frac{3}{2n} + \frac{1}{2 \cdot (e^{n}- 1)}   +   \ln(2) 
+  \left(\frac{1}{n} + e^{-n}\right) \cdot \ln\left(\frac{e^n}{n} + 1\right) 
\\
& \le  & 
1  + \frac{1}{2 \cdot (e - 1)}   +   \ln(2) 
+   \frac{2}{n} \cdot \ln(2 e^n) 
\\
& =  & 
1  + \frac{1}{2 \cdot (e - 1)}   +   \ln(2)  +   2 \cdot \ln(2) 
\\ & \le & 4.
\end{eqnarray*}
This finishes the proof of Lemma~\ref{lem:addendum}.
\end{proof}

\begin{remark}[Exotic behavior at zero] \label{rem:exotic_behavior_at_zero}
  The sequence of functions $(f_n)_{n \ge 0}$ has an exotic behavior close to
  zero in a small range depending on $n$.  There are no $C> 0$ and $\epsilon > 0$ 
  such that $f_n'(\lambda) \le C$ holds for all $n \ge 1$ and all 
  $\lambda \in (0,\epsilon)$ for which the derivative exists.

  This exotic behavior is responsible for the violation of the the uniform
  integrability condition, see Lemma~\ref{lem:addendum}~\eqref{lem:addendum:(6)}.
  It is very unlikely that such a sequence $(f_n)_{n \ge 0}$ actually occurs as
  the sequence of spectral density functions of the manifolds $G_i\backslash M$
  for some smooth manifold $M$ with proper free cocompact $G$-action and
  $G$-invariant Riemannian metric. The example above shows that we need to have
  more information on such sequences of spectral density functions.
\end{remark}


\subsection{Uniform estimate on spectral density functions}
\label{subsec:Uniform_estimate_on_spectral_density_functions}

The crudest way to ensure the uniform integrability condition~\eqref{the:approxi_for_spectral_density:integrability}
appearing in Theorem~\ref{the:approxi_for_spectral_density} 
is to assume a uniform gap in the spectrum, namely we have the obvious 

\begin{lemma}\label{lem:uniform_gap}
Suppose that \emph{the uniform gap in the spectrum at zero condition} is satisfied, i.e., 
there exists $\epsilon > 0$ such that for all $i \in I$ and $\lambda \in [0,\epsilon]$ we have
$F[i](\lambda) = F[i](0)$.

Then the uniform integrability condition~\eqref{the:approxi_for_spectral_density:integrability}
appearing in Theorem~\ref{the:approxi_for_spectral_density} is satisfied.
\end{lemma}

However, this is an unrealistic condition in our situation for closed aspherical manifolds
because of the following remark. 

\begin{remark}[The Zero-in-the-Spectrum Conjecture]
  \label{rem:The_Zero-in-the-Spectrum_Conjecture}
  Let $M$ be an aspherical closed manifold. If one wants to use
  Lemma~\ref{lem:uniform_gap} in connection with
  Theorem~\ref{the:approxi_for_spectral_density} to prove
  Conjecture~\ref{con:Approximation_conjecture_for_Fuglede-Kadison_determinants_with_finite_index}
  or more generally Conjecture~\ref{con:Approximation_conjecture_for_analytic_L2-torsion} 
  for $\widetilde{M}$, one has to face the fact that the assumption
  that one has in each dimension a uniform gap in the spectrum at zero implies
  that $b_p^{(2)}(\widetilde{M})$ vanishes and the $p$-th Novikov-Shubin invariant satisfies 
  $\alpha_p(\widetilde{M}) = \infty^*$ for all  $p  \ge 0$. In other words, $M$ must be a counterexample to the
  Zero-in-the-Spectrum Conjecture which is discussed in detail
  in~\cite[Chapter~12 on pages~437~ff]{Lueck(2002)}.  Such counterexample is not
  known to exist and it is evident that it is hard to find one.  Therefore the
  uniform gap in the spectrum at zero condition is not useful in this  setting.

  There are examples where Lemma~\ref{lem:uniform_gap}  does apply when one allows to twist with
  representations in favorable case, see for 
  instance~\cite{Bergeron-Venkatesh(2013),Marshall-Mueller(2013),Mueller-Pfaff(2013locsym),Mueller-Pfaff(2013asymp)}.

\end{remark}

Here is a more promising version.

\begin{theorem}[The uniform logarithmic estimate]
\label{the:the_uniform_logarithmic_estimate}
Suppose that there exists constants $C > 0 $, $0 < \epsilon < 1$ and $\delta > 0$ 
independent of $i$ such that for all $\lambda$ with $0 < \lambda \le  \epsilon$
and all $i \in I$ we have 
\begin{eqnarray*}
  F[i](\lambda) - F[i](0) \le & \frac{C}{(-\ln(\lambda))^{1 + \delta}}.
  \label{ln(lambda)(-1-1delta_n)_bound}
\end{eqnarray*}

Then the uniform integrability condition~\eqref{the:approxi_for_spectral_density:integrability}
appearing in Theorem~\ref{the:approxi_for_spectral_density} is satisfied.
\end{theorem}
\begin{proof}
This follows from  the following calculation.
\begin{eqnarray*}
  \int_{+0}^{\epsilon} \frac{C}{\lambda \cdot (-\ln(\lambda))^{1 + \delta}} d \lambda 
   & \stackrel{\lambda = \exp(\mu)}{=} & 
  \int_{-\infty}^{\ln(\epsilon)} \frac{C}{\exp(\mu) \cdot (-\ln(\exp(\mu)))^{1 + \delta}}  \cdot \exp(\mu) \, d\mu
  \\
  & = & 
  \int_{-\infty}^{\ln(\epsilon)} \frac{C}{(-\mu)^{1 + \delta}}  d\mu
  \\
  & \stackrel{\mu = -\nu}{=} & 
  \int_{-\ln(\epsilon)}^{\infty} \frac{C}{\nu^{1 + \delta}}  d\nu
  \\
  & = & 
  \lim_{x \to \infty} \int_{-\ln(\epsilon)}^{x} \frac{C}{\nu^{1 + \delta}}  d\nu
  \\
  & = & 
  \lim_{x \to \infty}  - \delta \cdot \left(x^{-\delta} - (-\ln(\epsilon))^{-\delta}\right)
  \\
  & = & 
  \delta \cdot  (-\ln(\epsilon))^{-\delta}
  \\
  & < & \infty,
\end{eqnarray*}
\end{proof}

The $p$-th spectral density function $F_p(X[i])$ of $X[i]$ is defined as the spectral density function $F\bigl(c_p^{(2)}(X[i])\bigr))$
of the $p$-th differential of $C_*^{(2)}(X[i])$.
If we now consider the Setup~\ref{set:restricted}, we get 
 from~\cite[Theorem~0.3]{Lueck(1994c)} for all $p \ge 0$ at least the inequality
\begin{eqnarray}
  \frac{F_p(X[i])(\lambda) - F_p(X[i])(0)}{[G:G_i]} & \le & \frac{C}{-\ln(\lambda)}.
  \label{ln(lambda)(-1)_bound}
\end{eqnarray}
But this is not enough, since one cannot take $\delta = 0$ in Theorem~\ref{the:the_uniform_logarithmic_estimate},
namely, we have for every $C > 0$ and $0 < \epsilon < 1$
\begin{eqnarray*}
\int_{+0}^{\epsilon} \frac{C}{\lambda \cdot (-\ln(\lambda))} d \lambda 
& = &
\lim_{x \to 0+} \int_x^{\epsilon} \frac{C}{\lambda \cdot (-\ln(\lambda))} d \lambda 
\\
& = &
\lim_{x \to 0+}  C \cdot \left(- \ln(-\ln(\epsilon)) + \ln(-\ln(x))\right) 
\\
& = &
\infty.
\end{eqnarray*}

Condition~\eqref{ln(lambda)(-1-1delta_n)_bound} is implied by the stronger condition that
there exist for each $p \ge 0$ constants $C_p > 0$, $\epsilon_p > 0$ and $\alpha_p> 0$
independent of $i$ such that we have  for all $\lambda \in [0,\epsilon_p)$ and all $i = 1,2, \ldots$
\begin{eqnarray}
  \frac{F_p(X[i])(\lambda) - F_p(X[i])(0)}{[G:G_i]} & \le & C \cdot \lambda^{\alpha_p}.
  \label{lambda(alpha_n)_bound}
\end{eqnarray}

Condition~\eqref{lambda(alpha_n)_bound} is known for $p = 0$,
see~\cite[Theorem~1.1]{Koch-Lueck(2014)}.  However, there is an unpublished manuscript by
Grabowski and Virag~\cite{Grabowski-Virag(2013)}, where they show that there exists an
explicit element $a$ in the integral group ring of the wreath product $\IZ^3 \wr \IZ$ such
that the spectral density function of the associated $\caln(\IZ^3\wr \IZ)$-map $r_a \colon
\caln(\IZ^3 \wr \IZ) \to \caln(\IZ^3 \wr \IZ)$ does not satisfy
condition~\eqref{lambda(alpha_n)_bound}. This implies that there exists a closed
Riemannian manifold $M$ with fundamental group $G = \pi_1(X) = \IZ^3 \wr \IZ$ such that
for some $p$ condition~\eqref{lambda(alpha_n)_bound} is not satisfied for $X =
\widetilde{M}$ and the $p$-th Novikov-Shubin invariant of $\widetilde{M}$ is zero,
disproving a conjecture of Lott and L\"uck, see~\cite[Conjecture~7.1
and~7.2]{Lott-Lueck(1995)}.  It may still be possible that
Condition~\ref{lambda(alpha_n)_bound} and the conjecture of Lott and L\"uck hold for an
aspherical closed manifold $M$.

There is no counterexample known to the condition appearing in Theorem~\ref{the:the_uniform_logarithmic_estimate}
but the constant $\delta$ has to be depend on the group $G$, see Grabowski~\cite{Grabowski(2015large)}.


\typeout{-------   Section 17: Proof of some Theorems for Fuglede-Kadison determinants ------------------}

\section{Proof of Theorem~\ref{the:inequality_det_det}, Theorem~\ref{the:invertible_matrices_over_l1(G)},
and Corollary~\ref{cor:l1-chain_equivalence}}
\label{sec:Proof_of_some_Theorems_for_Fuglede-Kadison_determinants}

In this section we derive the proofs of Theorem~\ref{the:inequality_det_det},~%
\ref{the:invertible_matrices_over_l1(G)}    from
Theorem~\ref{the:approxi_for_spectral_density}. This needs some preparation.

Define for a matrix $B \in M_{r,s}(L^1(G))$ the real number
\begin{eqnarray}
K^G(B) & := & rs \cdot \max\{||b_{i,j}||_{L^1} \mid 1 \le i \le r, 1 \le j \le s\},
\label{KG(B)}
\end{eqnarray}
where for $a = \sum_{g \in G} \lambda_g \cdot g$ its $L^1$-norm $||a||_{L^1}$ is defined by $\sum_{g \in G} |\lambda_g|$.
The proof of the following result is analogous to the proof of~\cite[Lemma~13.33 on page~466]{Lueck(2002)}.

\begin{lemma} \label{lem_KG(B)-estimate} 
We get for $B \in M_{r,s}(L^1(G))$
\[
||r_B^{(2)} \colon L^2(G)^r \to L^2(G)^s|| \le K^G(B).
\]
\end{lemma}

Consider the setup~\ref{set:inverse_systems}. 

\begin{lemma} \label{lem_trace_and_limit}
Consider $B \in M_{d}(L^1(G))$. Let $B[i] \in M_d(L^1(G/G_i))$ be obtained from $B$ by applying
the map $L^1(G) \to L^1(G/G_i)$ induced by the projection $\psi_i \colon G \to G/G_i$. Then
\[
\tr_{\caln(G)}(B) = \lim_{i \in I} \tr_{\caln(G/G_i)}(B[i]).
\]
\end{lemma}
\begin{proof}
Suppose that $B$ looks like $\left(\sum_{g \in G} \lambda_g(r,s) \cdot g\right)_{r,s}$. 
Denote in the sequel by $e$ the unit element in $G$ or $G/G_i$. Then
\begin{eqnarray*}
\tr_{\caln(G)}(B) & = & \sum_{r = 1 }^d \lambda_e(r,r);
\\
\tr_{\caln(G/G_i)}(B[i]) & = & \sum_{r = 1 }^d \sum_{g \in G, \psi_i(g) = e} \lambda_g(r,r).
\end{eqnarray*}
Consider $\epsilon > 0$. We can choose a finite subset $S \subseteq G$ with $e \in S$ such that
$\sum_{g \in G, g \notin S} |\lambda_g(r,r)| < \epsilon/d$ holds for all $r \in \{1,2, \ldots , d\}$. 
Since $\bigcap_{i \in I} G_i = \{1\}$, there exists an index $i_S$ such that $\psi_i(g) = e \Rightarrow g = e$ holds for all
$g \in S$ and $i \ge i_S$. This implies for $i \ge i_S$
\begin{eqnarray*}
\bigl|\tr_{\caln(G)}(B) - \tr_{\caln(G/G_i)}(B[i])\bigr|
& = & 
\biggl|\sum_{r = 1 }^d \lambda_e(r,r) - \sum_{r = 1 }^d \sum_{g \in G, \psi_i(g) = e} \lambda_g(r,r)\biggr|
\\
& \le & 
\sum_{r = 1 }^d \; \biggl|\lambda_e(r,r) - \sum_{g \in G, \psi_i(g) = e} \lambda_g(r,r)\biggr|
\\
& = & 
\sum_{r = 1 }^d  \; \biggl|\sum_{g \in G, g \notin S,\psi_i(g) = e} \; \lambda_g(r,r)\biggr|
\\
&\le  & 
\sum_{r = 1 }^d  \;\sum_{g \in G, g \notin S,\psi_i(g) = e}  \biggl|\lambda_g(r,r)\biggr|
\\
&\le  & 
\sum_{r = 1 }^d  \; \sum_{g \in G, g \notin S}  \biggl|\lambda_g(r,r)\biggr|
\\
&\le  & 
\sum_{r = 1 }^d  \epsilon/d
\\
& = & 
\epsilon.
\end{eqnarray*}
\end{proof}

Next we give the proof of Theorem~\ref{the:inequality_det_det}.

\begin{proof}[Proof of Theorem~\ref{the:inequality_det_det}]
We have to show for  $A \in M_{r,s}(\IQ G)$.
\begin{multline}
{\det}_{\caln(G)}\big(r_A^{(2)}\colon L^2(G)^r \to L^2(G)^s \bigr)
\\ 
\ge
\limsup _{i \in I} {\det}_{\caln(G/G_i)}\big(r_{A[i]}^{(2)}\colon L^2(G/G_i)^r \to L^2(G/G_i)^s \bigr).
\label{claim_to_prove}
\end{multline}
We first deal with the special case that $A \in M_{r,s}(\IZ G)$.

We will apply Theorem~\ref{the:approxi_for_spectral_density} to the following
special situation:
\begin{itemize}
\item $B = A^*A$;
\item $Q_i = G/G_i$;
\item $B[i] = A[i]^*A[i]$;
\item$\tr$ is the von Neumann trace $\tr_{\caln(G)} \colon \caln(G) \to \IC$;
\item$\tr_i$ is the von Neumann trace $\tr_{\caln(G/G_i)} \colon \caln(G/G_i) \to \IC$.
\end{itemize}
We have to check that the conditions of
Theorem~\ref{the:approxi_for_spectral_density}~\eqref{the:approxi_for_spectral_density:inequality} are satisfied.  We obtain
Condition~\eqref{the:approxi_for_spectral_density:uniform_bound_K} appearing in
Theorem~\ref{the:approxi_for_spectral_density} 
from Lemma~\ref{lem_KG(B)-estimate} 
since the projection $L^1(G) \to L^1(G/G_i)$ has operator
norm at most $1$ and hence we get for the number $K^G(B)$ defined 
in~\eqref{KG(B)}
\[
K^{Q_i}(B[i]) \le K^G(B).
\]
Condition~\eqref{the:approxi_for_spectral_density:traces_and_polynomials}
appearing in Theorem~\ref{the:approxi_for_spectral_density} follows
from Lemma~\ref{lem_trace_and_limit}.
Condition~\eqref{the:approxi_for_spectral_density:det} appearing in
Theorem~\ref{the:approxi_for_spectral_density} follows from the
Assumption~\ref{ass:Determinant_Conjecture} that each quotient $G/G_i$
satisfies the Determinant Conjecture~\ref{con:Determinant_Conjecture}.
Condition~\eqref{the:approxi_for_spectral_density:positivity} appearing in
Theorem~\ref{the:approxi_for_spectral_density} is satisfied because of $B = A^*A$ and $B[i]=A[i]^*A[i]$. 
Hence we conclude from
Theorem~\ref{the:approxi_for_spectral_density}~\eqref{the:approxi_for_spectral_density:inequality}
\[
{\det}_{\caln(G)}\bigl(r_B^{(2)} \colon L^2(G)^r \to L^2(G)^r\bigr) 
\ge  \limsup_{i \in I} \;{\det}_{\caln(G)}\bigl(r_{B[i]}^{(2)}\colon L^2(G/G_i)^r \to L^2(G/G_i)^r\bigr).
\]

Since we get from~\cite[Lemma~3.15~(4) on page~129]{Lueck(2002)}
\begin{eqnarray*} {\det}_{\caln(G)}\bigl(r_A^{(2)}\bigr) & = &
  \sqrt{{\det}_{\caln(G)}\bigl(r_B^{(2)}\bigr)};
  \\
  {\det}_{\caln(G/G_i})\bigl(r_{A[i]}^{(2)}\bigr) & = &
  \sqrt{{\det}_{\caln(G/G_i}\bigl(r_{B[i]}^{(2)}\bigr)},
\end{eqnarray*}
equation~\eqref{claim_to_prove} follows for $A \in M_{r,s}(\IZ G)$.

Next we reduce the general case $A \in M_{r,s}(\IQ G)$ to the case above.

Consider any  real  number $m > 0$, any group $H$ and any morphism
$f \colon L^2(H)^r \to L^2(H)^s$.
We conclude  from~\cite[Lemma~1.18 on page~24 and Theorem~3.14~(1) on page~128 and 
Lemma~3.15~(3), (4) and~(7) on page~129]{Lueck(2002)} 
\begin{eqnarray*}
\lefteqn{{\det}_{\caln(H)}\bigl(f \circ m\id_{L^2(H)^r}\bigr)^2}
& &
\\ 
& = & 
{\det}_{\caln(H)}\left(\bigl(f \circ m\id_{L^2(H)^r}\bigr)^* \circ \bigl(f \circ m\id_{L^2(H)^r}\bigr)\right)
\\
& = & 
{\det}_{\caln(H)}\bigl(f^* \circ f \circ m^2\id_{L^2(H)^r}\bigr)
\\
& = & 
{\det}_{\caln(H)}\left(\bigl(\left.f^* \circ f \circ m^2\id_{L^2(H)^r}\bigr)
\right|_{\ker\bigl(f^* \circ f \circ m^2\id_{L^2(H)^r}\bigr)^{\perp}}\right)
\\
& = & 
{\det}_{\caln(H)}\left(\left.\bigl(f^* \circ f \bigr)\right|_{\ker(f^*f)^{\perp}} \circ m^2\id_{\ker(f^*f)^{\perp}}\right)
\\
& = & 
{\det}_{\caln(H)}\left(\left.\bigl(f^* \circ f \bigr)\right|_{\ker(f^*f)^{\perp}}\right) 
\cdot  {\det}_{\caln(H)}\left(m^2\id_{\ker(f^*f)^{\perp}}\right)
\\
& = & 
{\det}_{\caln(H)}\bigl(f^* \circ f \bigr)
\cdot  m^{2 \cdot \dim_{\caln(H)}\ker(f^*f)^{\perp}}
\\
& = & 
{\det}_{\caln(H)}(f)^2 \cdot m^{2r-2\dim_{\caln(H)}(\ker(f^*f))}
\\
& = & 
\left({\det}_{\caln(H)}(f) \cdot  m^{r-\dim_{\caln(H)}(\ker(f))}\right)^2.
\end{eqnarray*}
Thus we have shown
\begin{eqnarray}
{\det}_{\caln(H)}\bigl(f \circ m\id_{L^2(H)^r}\bigr)
& = & 
{\det}_{\caln(H)}\bigl(f) \cdot m^{r - \dim_{\caln(H)}(\ker(f))}.
\label{nice_estimate_for_morphisms}
\end{eqnarray}

Let $m\ge 1 $ be an integer such that $mI_r\cdot A$ belongs
to $M_{r,s}(\IZ G)$, where $m I_r$ is obtained from the identity matrix by
multiplying all entries with $m$.  If we apply~\eqref{nice_estimate_for_morphisms} 
in the case $H = G$ and $H = G/G_i$ to
$f = r_A^{(2)}$ and $f = r_{A[i]}^{(2)}$, we obtain
\begin{eqnarray*}
{\det}_{\caln(G)}\bigl(r_{mI_r\cdot A}^{(2)}\bigr) 
& = & 
{\det}_{\caln(G)}\bigl(r_{A}^{(2)}\bigr)  \cdot m^{r-\dim_{\caln(G)}(\ker(r_A^{(2)}))};
\\
{\det}_{\caln(G/G_i)}\bigl(r_{mI_r\cdot A[i]}^{(2)}\bigr) 
& = & 
{\det}_{\caln(G/G_i)}\bigl(r_{A[i]}^{(2)}\bigr)  \cdot m^{r -\dim_{\caln(G/G_i)}(\ker(r_{A[i]}^{(2)}))}.
\end{eqnarray*}
Since ${\det}_{\caln(G)}\bigl(r_{mI_r\cdot A}^{(2)}\bigr)  \ge 1$ follows from~\eqref{claim_to_prove}
and the assumption that we have ${\det}_{\caln(G/G_i)}\bigl(r_{mI_r\cdot A[i]}^{(2)}\bigr)  \ge 1$ for $i \in I$,
we get ${\det}_{\caln(G)}\bigl(r_{A}^{(2)}\bigr) > 0$. We conclude
\begin{multline}
\frac{{\det}_{\caln(G/G_i)}\bigl(r_{A[i]}^{(2)}\bigr)}{{\det}_{\caln(G)}\bigl(r_{A}^{(2)}\bigr)}
\\ = 
\frac{{\det}_{\caln(G/G_i)}\bigl(r_{mI_rA[i]}^{(2)}\bigr)}{{\det}_{\caln(G)}\bigl(r_{mI_rA}^{(2)}\bigr)}
\cdot
m^{|\dim_{\caln(G)}(\ker(r_{A}^{(2)})) -  \dim_{\caln(G/G_i)}(\ker(r_{A[i]}^{(2)}))|}.
\label{conclusion_I}
\end{multline}
We derive
\begin{eqnarray}
\lim_{i \in I} \; \left(\dim_{\caln(G)}\bigl(\ker\bigl(r_{A}^{(2)}\bigr)\bigr) - 
\dim_{\caln(G/G_i)}\bigl(\ker\bigl(r_{A[i]}^{(2)}\bigr)\bigr)\right) 
& = & 0
\label{conclusion_II}
\end{eqnarray}
from Theorem~\ref{the:The_Determinant_Conjecture_implies_the_Approximation_Conjecture_for_L2-Betti_numbers}.
Since~\eqref{claim_to_prove} holds for $mI_rA$, it holds also for $A$
by~\eqref{conclusion_I} and~\eqref{conclusion_II}. This finishes the proof of
Theorem~\ref{the:inequality_det_det}.
\end{proof}

Next we give the proof of Theorem~\ref{the:invertible_matrices_over_l1(G)}. 

\begin{proof}[Proof of Theorem~\ref{the:invertible_matrices_over_l1(G)}] 
We will apply Theorem~\ref{the:approxi_for_spectral_density} to the following
special situation:
\begin{itemize}
\item $B = A^*A$;
\item $Q_i = G/G_i$;
\item $B[i] = A[i]^*A[i]$;
\item$\tr$ is the von Neumann trace $\tr_{\caln(G)} \colon \caln(G) \to \IC$;
\item$\tr_i$ is the von Neumann trace $\tr_{\caln(G/G_i)} \colon \caln(G/G_i) \to \IC$.
\end{itemize}
We have to check that the conditions of
Theorem~\ref{the:approxi_for_spectral_density}~\eqref{the:approxi_for_spectral_density:equality} are satisfied.  
Condition~\eqref{the:approxi_for_spectral_density:traces_and_polynomials}
appearing in Theorem~\ref{the:approxi_for_spectral_density} follows
from Lemma~\ref{lem_trace_and_limit}.
Condition~\eqref{the:approxi_for_spectral_density:positivity} appearing in
Theorem~\ref{the:approxi_for_spectral_density} is satisfied because of $B =
A^*A$ and $B[i]=A[i]^*A[i]$. 

Let $B^{-1}$ be the inverse of $B$ in $\GL_d(L^1(G))$.  Put $K:= \max\{K^G(B),K^{G/G_i}(B[i])\}$. 
Since the projection $L^1(G) \to L^1(G/G_i)$ has operator
norm at most $1$, we get for the numbers $K^G(B)$ and $K^G(B^{-1})$  defined 
in~\eqref{KG(B)} and all $i \in I$
\begin{eqnarray*}
K^{G/G_i}(B[i]) & \le & K;
\\
K^{G/G_i}(B[i]^{-1}) & \le & K.
\end{eqnarray*}
hold.  We conclude $||r_{B[i]^{-1}}^{(2)}|| \le K$ and $||r_{B[i]^{-1}}^{(2)}|| \le K$ for all $i \in I$  from Lemma~\ref{lem_KG(B)-estimate}.
In particular condition~\eqref{the:approxi_for_spectral_density:uniform_bound_K} appearing in
Theorem~\ref{the:approxi_for_spectral_density} is satisfied.
Since $r_{B[i]^{-1}}^{(2)}$ is the inverse of $r_{B[i]}^{(2)}$, we conclude 
from~\cite[Lemma~2.13~(2) on page~78, Theorem~3.14~(1) on page~128 and Lemma~3.15~(6) on page~129]{Lueck(2002)}
\begin{eqnarray}
  F[i](\lambda) & = & 0 \quad \text{for all} \; \lambda < K^{-1}\;\text{and}\;  i \in I;
  \label{gap_at_zero}
  \\
   {\det}_{\caln(G/G_i)}^{(2)}\bigl(r^{(2)}_{B[i]}\bigr) & \ge & d \cdot \ln(K) \quad \text{for all}\;   i \in I.
   \label{lower_bound_for_det}
\end{eqnarray}
Hence condition~\eqref{the:approxi_for_spectral_density:det} is satisfied if we take 
$\kappa :=  d \cdot \ln(K)$.
We conclude from~\eqref{gap_at_zero}
that also condition~\eqref{the:approxi_for_spectral_density:integrability} is satisfied.
We conclude from Theorem~\ref{the:approxi_for_spectral_density}~\eqref{the:approxi_for_spectral_density:equality}
\[
{\det}_{\caln(G)}^{(2)}\bigl(r_B^{(2)} \colon L^2(G)^d \to L^2(G)^d\bigr) 
= \lim_{i \in I} \;{\det}_{\caln(G/G_i)}^{2}\bigl(r_{B[i]}^{(2)}\colon L^2(G/G_i)^d \to L^2(G/G_i)^d\bigr).
\]
Since we get from~\cite[Lemma~3.15~(4) on page~129]{Lueck(2002)}
\begin{eqnarray*} {\det}_{\caln(G)}^{(2)}\bigl(r_A^{(2)}\bigr) & = &
  \sqrt{{\det}_{\caln(G)}^{(2)}\bigl(r_B^{(2)}\bigr)};
  \\
  {\det}_{\caln(G/G_i)}^{(2)}\bigl(r_{A[i]}^{(2)}\bigr) & = &
  \sqrt{{\det}_{\caln(G/G_i)}^{(2)}\bigl(r_{B[i]}^{(2)}\bigr)},
\end{eqnarray*}
Theorem~\ref{the:invertible_matrices_over_l1(G)} follows.
\end{proof}

Next we prove Corollary~\ref{cor:l1-chain_equivalence}
\begin{proof}[Proof of
  Corollary~\ref{cor:l1-chain_equivalence}]~\eqref{cor:l1-chain_equivalence:acyclic}
  Since $C_*$ is acyclic over $L^1(G)$ and finitely generated free, we
  can choose an $L^1(G)$-chain contraction $\gamma \colon C_* \to C_{*+1}$. 
  Then $(c+ \gamma)_{\odd} \colon C_{\odd} \xrightarrow{\cong} C_{\ev}$ 
  is an isomorphism of finitely generated
  based free $L^1(G)$-modules. It induces an isomorphism of
  finitely generated Hilbert $\caln(G)$-modules 
  $(c+  \gamma)_{\odd}^{(2)} \colon C_{\odd}^{(2)} \xrightarrow{\cong} C_{\ev}^{(2)}$. 
  We conclude from~\cite[Lemma~3.41 on page~146]{Lueck(2002)}
    \[
    \rho^{(2)}\bigl(C_*^{(2)}\bigr) := \ln\left({\det}_{\caln(G)}^{(2)}\bigl((c+
      \gamma)_{\odd}^{(2)} \colon C_{\odd}^{(2)} \xrightarrow{\cong}
        C_{\ev}^{(2)}\bigr)\right).
      \]
  Analogously we prove for each $i \in I$
  \[
    \rho^{(2)}\bigl(C[i]_*^{(2)}\bigr) := \ln\left({\det}_{\caln(G/G_i)}^{(2)}\bigl((c[i]+
      \gamma[i])_{\odd}^{(2)} \colon C[i]_{\odd}^{(2)} \xrightarrow{\cong}
        C[i]_{\ev}^{(2)}\bigr)\right).
      \]
Now assertion~\eqref{cor:l1-chain_equivalence:acyclic} follows from Theorem~\ref{the:invertible_matrices_over_l1(G)}.
      \\[1mm]~\eqref{cor:l1-chain_equivalence:equivalence} We begin with the case of an isomorphism
     $f_* \colon C_* \xrightarrow{\cong} D_*$ of finitely generated based free $L^1(G)$-chain complexes. 
     We conclude from~\cite[Lemma~3.41 on page~146]{Lueck(2002)} for all $i \in I$
     \begin{eqnarray*}
     \rho^{(2)}\bigl(D_*^{(2)}\bigr) - \rho^{(2)}\bigl(C_*^{(2)}\bigr) & = & \sum_{p \ge 0} (-1)^p \cdot \ln\bigl({\det}_{\caln(G)}^{(2)}(f_p^{(2)})\bigr);
     \\
     \rho^{(2)}\bigl(D[i]_*^{(2)}\bigr) - \rho^{(2)}\bigl(C[i]_*^{(2)}\bigr) & = & \sum_{p \ge 0} (-1)^p \cdot \ln\bigl({\det}_{\caln(G/G_i)}^{(2)}(f[i]_p^{(2)})\bigr).
     \end{eqnarray*}
     Now the claim follows in this special case from Theorem~\ref{the:invertible_matrices_over_l1(G)}.

    Finally we consider an $L^1(G)$-chain homotopy equivalence   $f_* \colon C_* \xrightarrow{\simeq} D_*$.
    Let $\cyl(f_*)$ be its mapping cylinder and $\cone(f_*)$ be its mapping cone.
    Let $\cone(C_*)$ be the mapping cone of $C_*$. We obtain based exact sequences of $L^1(G)$-chain complexes
    $$0 \to C_* \to \cyl(f_*) \to \cone(f_*) \to 0$$
    and
    $$0 \to D_* \to \cyl(f_*) \to \cone(C_*) \to 0.$$
     Since $f_*$ is a $L^1(G)$-chain homotopy equivalence, $\cone(f_*)$ is contractible.
    Since $\cone(C_*)$ is contractible, we can find isomorphisms of $L^1(G)$-chain complexes
    (cf.~\cite[Lemma~3.42 on page~148]{Lueck(2002)})
    \begin{eqnarray*}
    u_* \colon C_* \oplus \cone(f_*) & \xrightarrow{\cong} &\cyl(f_*);
    \\
    v_* \colon D_* \oplus \cone(C_*) & \xrightarrow{\cong} & \cyl(f_*).
   \end{eqnarray*}
   Since we have already treated the case of a chain isomorphism, we conclude
   \begin{multline*}
   \rho^{(2)}\left(\bigl(C_* \oplus \cone(f_*)\bigr)^{(2)}\right) - \rho^{(2)}\bigl(\cyl(f_*)^{(2)}\bigr)
   \\ = \;
   \lim_{i \in I} \rho^{(2)}\left(\bigl(C[i]_* \oplus \cone(f[i]_*)\bigr)^{(2)}\right) - \rho^{(2)}\bigl(\cyl(f[i]_*)^{(2)}\bigr);
  \end{multline*}
  and 
   \begin{multline*}
 \rho^{(2)}\left(\bigl(D_* \oplus \cone(C_*)\bigr)^{(2)}\right) - \rho^{(2)}\bigl(\cyl(f_*)^{(2)}\bigr)
   \\ =  \;
   \lim_{i \in I} \rho^{(2)}\left(\bigl(D[i]_* \oplus \cone(C[i]_*)\bigr)^{(2)}\right) - \rho^{(2)}\bigl(\cyl(f[i]_*)^{(2)}\bigr).
 \end{multline*}
This implies
  \begin{multline*}
   \rho^{(2)}\bigl(C_*^{(2)}\bigr) + \rho^{(2)}\bigl(\cone(f_*^{(2)})\bigr) - \rho^{(2)}\bigl(D_*^{(2)}\bigr) - \rho^{(2)}\bigl(\cone(C_*)^{(2)}\bigr) 
   \\
   = \;
   \lim_{i \in I} \left(\rho^{(2)}\bigl(C[i]_*^{(2)}\bigr) + \rho^{(2)}\bigl(\cone(f[i]_*)^{(2)}\bigr) 
   - \rho^{(2)}\bigl(D[i]_*^{(2)}\bigr) - \rho^{(2)}\bigl(\cone(C[i]_*)^{(2)}\bigr)\right).
\end{multline*}
We conclude from assertion~\eqref{cor:l1-chain_equivalence:acyclic} 
\begin{eqnarray*}
\rho^{(2)}\bigl(\cone(f_*^{(2)})\bigr) & = &   \lim_{i \in I} \rho^{(2)}\bigl(\cone(f[i]_*)^{(2)}\bigr);
\\
\rho^{(2)}\bigl(\cone(C_*^{(2)})\bigr) & = &   \lim_{i \in I} \rho^{(2)}\bigl(\cone(C[i]_*)^{(2)}\bigr).
\end{eqnarray*}
This implies
\begin{eqnarray*}
   \rho^{(2)}\bigl(C_*^{(2)}\bigr)  - \rho^{(2)}\bigl(D_*^{(2)}\bigr) 
   & = & 
   \lim_{i \in I} \left(\rho^{(2)}\bigl(C[i]_*^{(2)}\bigr)  - \rho^{(2)}\bigl(D[i]_*^{(2)}\bigr) \right).
\end{eqnarray*}
This finishes the proof of  Corollary~\ref{cor:l1-chain_equivalence}.
\end{proof}


\typeout{-------   Section 18:  Proof of some the Theorem about_approximating_by_Z-torsion -------------}

\section{Proof of Theorem~\ref{the:about_approximating_by_IZ-torsion}}
\label{sec:Proof_of_Theorem_ref(the:about_approximating_by_IZ-torsion)}

Next we want to prove Theorem~\ref{the:about_approximating_by_IZ-torsion}. First
we deal with homotopy invariance  and with the relationship
between $L^2$-torsion and integral torsion.

\begin{lemma}
  \label{lem:homotopy_invariance}
  Let $G$ be a group for which the Determinant
  Conjecture~\ref{con:Determinant_Conjecture} is true.  Let $f_* \colon D_* \to
  E_*$ be a $\IZ G$-chain homotopy equivalence equivalence of finite based free
  $\IZ G$-chain complexes.  Suppose that $D_*^{(2)}$ or $E_*^{(2)}$ is
  $L^2$-acyclic. Then both $D_*^{(2)}$ and $E_*^{(2)}$ are $L^2$-acyclic and
\[\rho^{(2)}\bigl(D_*^{(2)}\bigr) = \rho^{(2)}\bigl(E_*^{(2)}\bigr).
\]
\end{lemma}
\begin{proof}
  This follows from~\cite[Theorem~3.93~(1) on page~161 and Lemma~13.6 on
  page~456]{Lueck(2002)}.
\end{proof}

\begin{notation}\label{not:closure_of_submodule}
  Let $A$ be a finitely generated free abelian group and let $B \subseteq A$ be a
  subgroup.  Define the \emph{closure} of $B$ in $A$ to be the subgroup
  \begin{eqnarray*}
    \overline{B}  
    & = & \{x \in A \mid n\cdot x \in B \; \text{for some non-zero integer}\; n\}.
  \end{eqnarray*}
\end{notation}

Notice that $A/\overline{B}$ and $M_f:= M/\tors(M)$ are finitely generated free and we have
$\overline{\ker(f)} = \ker(f)$ for a homomorphism $f \colon A_0 \to A_1$ of
finitely generated free abelian groups. The proof of the next result 
can be found in~\cite[Lemma~2.11]{Lueck(2013l2approxfib)}.

\begin{lemma} \label{lem:Fuglede-Kadison_and_tors_for_G_is_trivial} 
  Let $u  \colon \IZ^r \to \IZ^s$ be a homomorphism of abelian groups. Let 
   $j \colon \ker(u) \to \IZ^r$ be the inclusion
   and $\pr \colon \IZ^s \to \coker(u)_f$ be the canonical
  projection.  Choose $\IZ$-basis for $\ker(u)$ and $\coker(u)_f$.

  Then ${\det}_{\caln(\{1\})}\bigl(j^{(2)}\bigr)$ and
  ${\det}_{\caln(\{1\})}\bigl(\pr^{(2)}\bigr)$ are independent of
  the choice of the $\IZ$-basis for $\ker(u)$ and $\coker(u)_f$, and we have
  \begin{eqnarray*}
    {\det}_{\caln(\{1\})}(u^{(2)}) & = &
    {\det}_{\caln(\{1\})}\bigl(j^{(2)}\bigr) \cdot
    \bigl|\tors(\coker(u))\bigr|
    \cdot {\det}_{\caln(\{1\})}\bigl(\pr^{(2)}\bigr);
  \end{eqnarray*}
  and
  \[
   \begin{array}{lclcl}
    1 & \le & {\det}_{\caln(\{1\})}(j^{(2)}) & \le & {\det}_{\caln(\{1\})}(u^{(2)});
    \\
    1 & \le &  {\det}_{\caln(\{1\})}\bigl(\pr^{(2)}\bigr) & \le & {\det}_{\caln(\{1\})}(u^{(2)});
    \\
    1 & \le & \bigl|\tors(\coker(u))\bigr| & \le & {\det}_{\caln(\{1\})}(u^{(2)}).
  \end{array}
  \]
  \end{lemma}

The point of the next lemma is that the chain complexes live over $\IZ G$ but the chain homotopy
equivalence has only to exist over $\IQ G$.

\begin{lemma} \label{lem:invariance_of_homological_conjecture_under_Q-homotopy_equivalence}
Let $C_*$ and $D_*$ be two finite free $\IZ G$-chain complexes. Suppose
that $C_* \otimes_{\IZ} \IQ$ and $D_* \otimes_{\IZ} \IQ$ are $\IQ G$-chain homotopy
equivalent and that $C_*^{(2)}$ is $L^2$-acyclic. Then $D_*^{(2)}$ is $L^2$-acyclic and
\begin{eqnarray*}
\rho^{(2)}\bigl(D_*^{(2)}\bigr) - \rho^{(2)}\bigl(C_*^{(2)}\bigr)
& = & 
\lim_{i \in I} \;\frac{\rho^{\IZ}\bigl(D[i]_*\bigr) - \rho^{\IZ}\bigl(C[i]_*\bigr)}{[G:G_i]}.
\end{eqnarray*}
\end{lemma}
\begin{proof}
  Let $g_* \colon C_* \otimes_{\IZ} \IQ \to D_* \otimes_{\IZ} \IQ$ be a $\IQ
  G$-chain homotopy equivalence.  Since $C_*$ and $D_*$ are finite based free
  $\IZ G$-chain complexes, we can find a $\IZ G$-chain map $f_* \colon C_* \to
  D_*$ and an integer $l$ such that $f_* \otimes_{\IZ} \IQ = l \cdot
  g_*$. Obviously $f_* \otimes_{\IZ} \IQ$ is  a $\IQ G$-chain homotopy
  equivalence.  In the sequel we abbreviate $C_*' := \cyl(f_*)$ and $C_*'' :=
  \cone(f_*)$.  By the chain homotopy invariance of integral torsion and of
  $L^2$-torsion (see Lemma~\ref{lem:homotopy_invariance}) it suffices to prove
  the claim for $C_*$ and $C_*'$ instead of $C_*$ and $D_*$.

We have the obvious exact sequence of
finite based free $\IZ G$-chain complexes 
\[0 \to C_* \xrightarrow{i_*} C_*' \xrightarrow{p_*} C_*'' \to 0.
\]

Since $f_* \otimes_{\IZ} \IQ$ is a $\IQ G$-chain homotopy equivalence, we can choose a
$\IQ G$-chain contraction $\gamma_*\colon C_*''\otimes_{\IZ} \IQ \to
C_{*+1}''\otimes_{\IZ} \IQ$.  Since each $C_*''$ is a finite free  $\IZ G$-chain complex,
we can find an integer $m$
and $\IZ G$-maps $\delta_p \colon C_p \to C_{p+1}''$ such that $m \cdot \gamma_p
= \delta_p \otimes_{\IZ} \id_{\IQ}$ holds for all $p \ge 0$. Hence $\delta_*
\colon C_*'' \to C_{*+1}''$ is a $\IZ G$-chain homotopy from $m \cdot \id_{C_*}$
to the zero homomorphism.  Moreover, $\delta[i]_* \colon C''[i]_* \to
C''[i]_{*+1}$ is a $\IZ [G/G_i]$-chain homotopy from $m \cdot \id_{C''[i]_*}$ to
the zero homomorphism for all $i \in I$. Hence multiplication with $m$
annihilates $H_p(C''[i])$ for all $n \ge 0$ and $p \in I$.

 We have the long exact homology sequence
\[
\cdots \to H_p(C[i]_*) \to H_p(C'[i]_*) \to H_p\bigl(C''[i]_*\bigr) \to H_{p-1}(C[i]_*) \to \cdots.
\]
The group $H_p\bigl(C''[i]_*\bigr)$ is a finite abelian group for each $p \ge 0$.
We obtain the following commutative diagram with exact rows
\begin{eqnarray}
&&
\label{exact_sequence_of_chain_complexes_new}
\\
\xymatrix{
& \vdots \ar[d] & \vdots \ar[d] & \vdots \ar[d] & 
\\
0 \ar[r]
& 
\tors\bigl(H_p(C[i]_*)\bigr) \ar[r] \ar[d]
&
H_p(C[i]_*) \ar[r] \ar[d] 
&
H_p(C[i]_*)_f \ar[r] \ar[d] 
&
0
\\
0 \ar[r]
& \tors\bigl(H_p(C'[i]_*)\bigr) \ar[r] \ar[d]
&
H_p(C'[i]_*) \ar[r] \ar[d] 
&
H_p(C'[i]_*)_f \ar[r] \ar[d] 
&
0
\\
0 \ar[r]
&
\tors\bigl(H_p(C''[i]_*)\bigr)  \ar[d]  \ar[r]^{\cong} 
&
H_p(C''[i]_*) \ar[r]\ar[d] 
&
0 \ar[r]\ar[d] 
&
0
\\
0 \ar[r]
& 
\tors\bigl(H_{p-1}(C[i]_*)\bigr) \ar[r] \ar[d]
&
H_{p-1}(C[i]_*) \ar[r] \ar[d] 
&
H_{p-1}(C[i]_*)_f \ar[r] \ar[d] 
&
0
\\
& \vdots & \vdots& \vdots  & 
}
\nonumber
\end{eqnarray}
We view it as a short exact sequence of $\IZ$-chain complexes and hence can
consider the associated long homology sequence.  Notice that the chain complex
given by the middle column is acyclic. Hence we obtain isomorphisms
\begin{multline}
\ker\left(\tors\bigl(H_p(C'[i]_*)\bigr) \to \tors\bigl(H_p(C''[i]_*)\bigr)\right)\left/
\im\left(\tors\bigl(H_p(C[i]_*)\bigr) \to \tors\bigl(H_p(C'[i]_*)\bigr)\right)\right.
\\
 \cong 
\ker\left(H_p(C[i]_*)_f \to H_p(C'[i]_*)_f\right),
\label{ker/coker_I}
\end{multline}
\begin{multline}
\ker\left(\tors\bigl(H_p(C''[i]_*)\bigr)  \to \tors\bigl(H_{p-1} (C[i]_*)\bigr) \right)
\left/
\im\left(\tors\bigl(H_p(C'[i]_*)\bigr) \to \tors\bigl(H_p(C''[i]_*)\bigr)\right)\right.
\\
 \cong 
\coker\left(H_p(C[i]_*)_f \to H_p(C'[i]_*)_f\right),
\label{ker/coker_II}
\end{multline}
and
\begin{multline}
\ker\left(\tors\bigl(H_p(C[i]_*)\bigr) \to \tors\bigl(H_p(C'[i]_*)\bigr)\right)\left/
\im\left(\tors\bigl(H_{p+1}(C[i]''_*)\bigr) \to \tors\bigl(H_p(C[i]_*)\bigr)\right)\right.
\\
 \cong 0.
\label{ker/coker_III}
\end{multline}
Obviously $(H_p(C[i]_*)_f$ and hence $\ker\left(H_p(C[i]_*)_f \to H_p(C'[i]_*)_f\right)$ are torsion-free.
On the other hand $\ker\left(H_p(C[i]_*)_f \to H_p(C'[i]_*)_f\right)$ is finite,
since $\tors(H_p(C'[i]_*))$ is finite and $\ker\left(H_p(C[i]_*) \to H_p(C'[i]_*)\right)$ is a quotient 
of $H_{p+1}(C''[i]_*)$ and hence finite. Hence $\ker\left(H_p(C[i]_*)_f \to H_p(C'[i]_*)_f\right)$ is trivial.
We conclude from~\eqref{ker/coker_I}
\begin{multline}
\ker\left(\tors\bigl(H_p(C'[i]_*)\bigr) \to \tors\bigl(H_p(C''[i]_*)\bigr)\right)
\\ 
=   \im\left(\tors\bigl(H_p(C[i]_*)\bigr) \to \tors\bigl(H_p(C'[i]_*)\bigr)\right).
\label{kernel(tors)_is_im(tors)}
\end{multline}
The cokernel of the map $H_p(C[i]_*) \to H_p(C'[i]_*)$ is a submodule
of $H_{p}(C''[i]_*)$ and hence annihilated by multiplication with $m$. 
The cokernel of $H_p(C[i]_*)_f \to H_p(C'[i]_*)_f$ is a quotient of
the cokernel of $H_p(C[i]_*)\to H_p(C'[i]_*)$. We conclude that
$\coker\bigl(H_p(C[i]_*)_f \to H_p(C'[i]_*)_f\bigr)$ is annihilated by multiplication with $m$.
Hence we obtain an epimorphism
\[
H_p(C'[i]_*)_f/m \cdot H_p(C'[i]_*)_f \to \coker\bigl(H_p(C[i]_*)_f \to H_p(C'[i]_*)_f\bigr).
\]
This implies
\begin{eqnarray*}
  \left|\coker\big(H_p(C[i]_*)_f \to H_p(C'[i]_*)_f\bigr)\right|
  & \le & m^{\rk_{\IZ}\left(H_p(C'[i]_*)\right)}.
\end{eqnarray*}
Since $C_*^{(2)}$ is $L^2$-acyclic, and $C_*\otimes_{\IZ} \IQ$ and
$C_*'\otimes_{\IZ} \IQ$ are $\IQ G$-chain homotopy equivalent, $(C')_*^{(2)}$ is
$L^2$-acyclic.  We conclude from~\cite[Theorem~0.1]{Lueck(1994c)} for all $p
\ge 0$
\begin{eqnarray*}
\lim_{i \in I} \frac{\rk_{\IZ}\bigl(H_p(C_*'[i])\bigr)}{[G:G_i]} = 0.
\end{eqnarray*}
Since $m$ is independent of $p$, we conclude
\begin{eqnarray}
  \lim_{i \in I} \frac{\ln\left(\left|\coker\big(H_p(C[i]_*)_f 
\to H_p(C'[i]_*)_f\bigr)\right|\right)}{[G:G_i]} = 0.
  \label{approxi_for_ln(rk_ZH_k(C[i]_toH_k(Cprime[i])}
\end{eqnarray}
Taking the logarithm of the order of a finite abelian group is additive under
short exact sequences of finite abelian groups. Hence we get for any
finite-dimensional chain complex $E_*$ of finite abelian groups
\[\
\sum_{p \ge 0} (-1)^p \cdot |E_p| = \sum_{p \ge 0} (-1)^p \cdot |H_p(E_*)|.
\]
If we apply this to the left column in the
diagram~\eqref{exact_sequence_of_chain_complexes_new}, we conclude
from~\eqref{ker/coker_II},~\eqref{ker/coker_III},
and~\eqref{kernel(tors)_is_im(tors)}
\begin{eqnarray*}
\lefteqn{\left|\, \sum_{p \ge 0} (-1)^p \cdot \ln\left(\left|\tors\bigl(H_p(C[i]_*))\right|\right) 
- \sum_{p \ge 0} (-1)^p \cdot \ln\left(\left|\tors\bigl(H_p(C'[i]_*)\bigr)\right| \right)\right.}
& & 
\\
& &
\left. + \sum_{p \ge 0} (-1)^p \cdot \ln\left(\left|\tors\bigl(H_p(C''[i]_*)\right|\right)\,\right|
\\
& & \hspace{30mm} = \; 
\left|\,\sum_{p \ge 0} (-1)^p \cdot 
\ln\left(\left|\coker\left(H_p(C[i]_*)_f \to H_p(C'[i]_*)_f\right)\right|\right)\,\right|
\\
& & \hspace{30mm} \le  \; 
\sum_{p \ge 0} \ln\left(\left|\coker\left(H_p(C[i]_*)_f \to H_p(C'[i]_*)_f\right)\right|\right).
\end{eqnarray*}
This together with~\eqref{approxi_for_ln(rk_ZH_k(C[i]_toH_k(Cprime[i])} implies
\begin{eqnarray}
\lim_{i \in I} \left(\frac{\rho^{\IZ}(C[i]_*)}{[G:G_i]} - \frac{\rho^{\IZ}(C'[i]_*)}{[G:G_i]}  
+ \frac{\rho^{\IZ}(C''[i]_*)}{[G:G_i]}\right)
& = & 0.
\label{additivity_approxi_for_three_rhoZ-s}
\end{eqnarray}
We conclude from~\cite[Lemma~3.68 on page~153]{Lueck(2002)}
\begin{eqnarray}
\rho^{(2)}\bigl(C^{(2)}_*\bigr) - \rho^{(2)}\bigl((C')^{(2)}_*\bigr) + \rho^{(2)}\bigl((C'')^{(2)}_*\bigr) 
& = & 0.
\label{additivity_for_three_rho(2)-s}
\end{eqnarray}
Hence it suffices to show
\begin{eqnarray}
\rho^{(2)}\bigl((C'')^{(2)}_*\bigr) 
& = & 
\lim_{i \in I}  \frac{\rho^{\IZ}(C''[i]_*)}{[G:G_i]}.
\label{equality_for_C_prime_prime}
\end{eqnarray}
We conclude from Corollary~\ref{cor:l1-chain_equivalence}~\eqref{cor:l1-chain_equivalence:acyclic}
\begin{eqnarray*}
\rho^{(2)}\bigl((C'')^{(2)}_*\bigr) 
& = & 
\lim_{i \in I}  \frac{\rho^{(2)}\bigl(C''[i]_*^{(2)}\bigr)}{[G:G_i]}.
\end{eqnarray*}
Since $H_p\bigl((C''[i])_*\bigr) \otimes_{\IZ} \IQ$ vanishes for all $p\ge 0$ and $i \in I$,
\eqref{equality_for_C_prime_prime} follows from~Lemma~\ref{lem:rho(2)-rhoZ}.
This finishes the proof of 
Lemma~\ref{lem:invariance_of_homological_conjecture_under_Q-homotopy_equivalence}.
\end{proof}

Now we are ready to prove Theorem~\ref{the:about_approximating_by_IZ-torsion}.
\begin{proof}[Proof of Theorem~\ref{the:about_approximating_by_IZ-torsion}]%
~\eqref{the:about_approximating_by_IZ-torsion:inequality}
Notice that 
\begin{eqnarray*}
\frac{\ln\bigl({\det}_{\caln(\{1\})}(f[i]^{(2)})\bigr)}{[G:G_i]} 
& = & 
\ln\bigl({\det}_{\caln(G/G_i)}(f[i]^{(2)})\bigr)
\end{eqnarray*}
holds by~\cite[Theorem~3.14~(5) on page~128]{Lueck(2002)}.
Theorem~\ref{the:inequality_det_det} implies
\begin{eqnarray*}
\ln\bigl({\det}_{\caln(G)}(f^{(2)})\bigr)
& \ge  & 
\limsup_{i \in I} \frac{\ln\bigl({\det}_{\caln(\{1\})}(f[i]^{(2)})\bigr)}{[G:G_i]}.
\end{eqnarray*}
Now apply Lemma~\ref{lem:Fuglede-Kadison_and_tors_for_G_is_trivial}.
\\[1mm]~\eqref{the:about_approximating_by_IZ-torsion:acyclic_in_each_degree}
Obviously it suffices to prove the claim for chain complexes.
Notice that 
\begin{eqnarray*}
\frac{\rho^{(2)}\bigl(C[i]_*^{(2)}\bigr)}{[G:G_i]}
& = & 
\rho^{(2)}\bigl(C[i]_*^{(2)};\caln(G/G_i)\bigr)
\end{eqnarray*}
holds by~\cite[Theorem~3.35~(7) on page~143]{Lueck(2002)}.
We conclude from Corollary~\ref{cor:l1-chain_equivalence}~\eqref{cor:l1-chain_equivalence:acyclic}
\begin{eqnarray*}
\rho^{(2)}\bigl(C^{(2)}_*\bigr) 
& = & 
\lim_{i \in I}  \frac{\rho^{(2)}\bigl(C[i]_*^{(2)}\bigr)}{[G:G_i]}.
\end{eqnarray*}
Since $H_p\bigl(C[i]_*\bigr) \otimes_{\IZ} \IQ$ vanishes for all $p \ge 0$ and $i \in I$,
assertion~\eqref{the:about_approximating_by_IZ-torsion:acyclic_in_each_degree}
follows from~Lemma~\ref{lem:rho(2)-rhoZ}
\\[1mm]~\eqref{the:about_approximating_by_IZ-torsion:G_is_Z}
Obviously it suffices to prove the chain complex version.
Let $C_*$ be a finite based free $\IZ[\IZ]$-chain complex
that is $L^2$-acyclic.  If $\IQ[\IZ]_{(0)}$ is the quotient field of the
integral domain $\IQ[\IZ]$, then $H_k(C_*) \otimes_{\IZ[\IZ]} \IQ[\IZ]_{(0)}$
is trivial for $k \ge 0$ because  of~\cite[Lemma~1.34~(1) on page~35]{Lueck(2002)}.
Since $\IQ[\IZ]$ is a principal ideal domain, we can find non-negative
integers $t_k$ and pairwise prime
irreducible elements $p_{k,1}$, $p_{k,2}$, \ldots , $p_{k,t_k}$ in $\IQ[\IZ]$
and natural numbers $m_{k,1}$, $m_{k,2}$, \ldots $m_{k,t_k}$ such that we have
isomorphisms of $\IQ[\IZ]$-modules
\[
H_k(C_*) \otimes_{\IZ} \IQ
\cong H_k(C_*) \otimes_{\IZ[\IZ]} \IQ[\IZ] \cong
\bigoplus_{j=1}^{t_k} \IQ[\IZ]/(p_{k,j}^{m_{k,j}}).
\] 
By multiplying the elements
$p_{k,j}$ with some natural number and a power of the generator $t \in \IZ$, 
we can arrange that the
elements $p_{k,1}$, $p_{k,2}$, \ldots , $p_{k,t_k}$ belong to $\IZ[t]$.
Since $\IQ[\IZ]/(p_{k,j}^{m_{k,j}}) \cong
\left(\IZ[\IZ]/(p_{k,j}^{m_{k,j}})\right)\otimes_{\IZ} \IQ$, there is a map of
$\IZ[\IZ]$-modules
\[\xi_k \colon \bigoplus_{j=1}^{t_k} \IZ[\IZ]/(p_{k,j}^{m_{k,j}}) \to H_k(C_*)
\]
which becomes an isomorphism of $\IQ[\IZ]$-modules after applying $-
\otimes_{\IZ} \IQ$.  By possibly enumerating the polynomials $p_{k,j}$ we can
arrange, that for some integer $s_k$ with $0 \le s_k \le t_k +1$ a polynomial
$p_{k,j}$ has some  root of unity as a
root if and only if $j\le s_k$. Consider $j \in\{1,2,\ldots, s_k\}$. Let
$d_{k,j} \ge 2$ be the natural number for which $p_{k,j}$ has a primitive
$d_{k,j}$-th root of unity as zero.  Recall the $d$-th cyclotomic polynomial
$\Phi_d$ is a polynomial over $\IZ[t]$ with $\Phi_{d_{k,j}}(0) = \pm 1$ and is
irreducible over $\IQ[t]$.  Hence we can find a unit in $u \in \IQ[\IZ]$ such that
$u \cdot \Phi_{d_{k,j}} = p_{k,j}$.  Every unit in $\IQ[\IZ] = \IQ[t,t^{-1}]$ is
of the shape $rt^l$ for some $r \in \IQ, r \not=0$ and $l \in \IZ$.  Since
$p_{k,j}$ is a polynomial in $\IZ[t]$, we can arrange $p_{k,j} =
\Phi_{d_{k,j}}$. To summarize, we have achieved that $p_{k,j}$ is
$\Phi_{d_{k,j}}$ for $j \in\{1,2,\ldots, s_k\}$ and that no 
root of unity  is a root of $p_{k,j}$ for 
$j \in\{s_k+1, s_k +2, \ldots , t_k\}$.

Let $F^{k,j}_*$ for $j \in\{1,2,\ldots, t_k\}$ be the $\IZ[\IZ]$-chain complex
which is concentrated in dimensions $(k+1)$ and $k$ and whose $(k+1)$-th
differential is the $\IZ[\IZ]$-homomorphism $\IZ[\IZ] \xrightarrow{p_{k,j}} \IZ[\IZ]$ 
given by multiplication with $p_{k,j}$. There is an obvious
identification of $\IZ[\IZ]$-modules
\[
H_k\bigl(F^{k,j}_*\bigr) \cong \IZ[\IZ]/(p_{k,j})
\] 
and 
$H_i\bigl(F^{k,j}_*\bigr) = 0$ for $ i \not= k$.
Since $F^{k,j}_*$ has projective chain modules and is concentrated in dimensions
$(k+1)$ and $k$ and we have the exact sequence of $\IZ[\IZ]$-modules $C_{k+1}
\xrightarrow{c_{k+1}} \ker(c_k) \to H_k(C_*)$, we can construct a
$\IZ[\IZ|$-chain map
\[f^{k,j}_* \colon F^{k,j}_* \to C_*
\]
such that $H_k\bigl(f^{k,j}_*\bigr)$ agrees with the restriction of $\xi_k$
to the $j$-th summand.  Define a $\IZ[\IZ]$-chain map
\[
f_*:= \bigoplus_{k \ge 0} \bigoplus_{j = 1}^{t_k} f^{k,j}_* \colon   
\bigoplus_{k \ge 0} \bigoplus_{j = 1}^{t_k} F^{k,j}_*  \to C_*.
\]
By construction $H_k(f_*) \otimes_{\IZ} \IQ$ is bijective for all $k \ge 0$.

We conclude from 
Lemma~\ref{lem:invariance_of_homological_conjecture_under_Q-homotopy_equivalence}
that we can assume without loss of generality
\[C_* = \bigoplus_{k \ge 0} \bigoplus_{j = 1}^{t_k} F^{k,j}_*.
\]
Obviously assertion~\eqref{the:about_approximating_by_IZ-torsion:G_is_Z}  is
satisfied for a direct sum $D_* \oplus E_*$ of two based free
$L^2$-acyclic $\IZ[\IZ]$-chain complexes if both $D_*$ and $E_*$
satisfy assertion~\eqref{the:about_approximating_by_IZ-torsion:G_is_Z}.
Hence we only have to treat the case, where $C_*$ is concentrated in
dimension $0$ and $1$ and its first differential is given by $p \cdot
\id \colon \IZ[\IZ] \to \IZ[\IZ]$ for some non-trivial polynomial $p$
such that either $p$ is of the shape $\phi_d^m$ for some natural
numbers $d$ and $m$ or no root of
unity is a root of $p$.

We begin with the case where $p$ is of the shape $\phi_d^m$ for some natural
numbers $d$ and $m$. Then all roots of $p$ have norm $1$ and hence
\[
\ln\bigl(\rho^{(2)}(C_*)\bigr) =  
\ln\bigl({\det}_{\caln(\IZ)}\bigl(p \cdot \id \colon L^2(\IZ) \to L^2(\IZ)\bigr)\bigr) = 0
\]
by~\cite[(3.23) on page~136]{Lueck(2002)}. Now the claim follows from
assertion~\eqref{the:about_approximating_by_IZ-torsion:inequality}.

Finally we treat the case, where no root of unity is a root of $p$.
Fix $i \in I$. Put $n = [\IZ:\IZ_i]$. Then $\IZ/\IZ_i = \IZ/n$. 
For $l \in \IZ/n$ let $\IC_l$ be the unitary
$\IZ/n$-representation whose underlying Hilbert space is $\IC$ and on which the
generator in $\IZ/n$ acts by multiplication with $\zeta_n^l$, where we put
$\zeta_n := \exp(2 \pi i/n)$. We obtain a unitary $\IZ/n$-isomorphism
\[
\omega \colon \bigoplus_{l \in \IZ/n} \IC_l \xrightarrow{\cong} \IC[\IZ/n].
\]
The following diagram of Hilbert $\caln(\IZ/n)$-modules commutes
\[\xymatrix{
  \bigoplus_{l \in \IZ/n}\IC_l \ar[d]_{\bigoplus_{l \in \IZ/n}
    p(\zeta_n^l)}\ar[r]^{\omega}_{\cong} & \IC[\IZ/n] \ar[d]^{p[i]}
  \\
  \bigoplus_{l \in \IZ/n} \IC_l\ar[r]^{\omega}_{\cong} & \IC[\IZ/n] }
\]
Hence $p[i]^{(2)} \colon \IZ[\IZ/n]^{(2)} \to \IZ[\IZ/n]^{(2)}$ is an isomorphism.
Therefore  $p[i]\colon \IZ[\IZ/n] \to \IZ[\IZ/n]$ is rationally an isomorphism.
Now the claim follows from 
assertion~\eqref{the:about_approximating_by_IZ-torsion:acyclic_in_each_degree}.
This finishes the proof of Theorem~\ref{the:about_approximating_by_IZ-torsion}.
\end{proof}


\typeout{------------------------ Section 19: Miscellaneous --------------------}

\section{Miscellaneous}
\label{sec:Miscellaneous}

We briefly mention some variations of the problems considered here or some other prominent open conjectures about
$L^2$-invariants.

\subsection{Approximation for lattices}
\label{subsec:Approximation_for_lattices}

In our setting we approximate the universal covering of a closed manifold or compact
CW-complex by a tower of finite coverings corresponding to the normal chain 
$(G_i)_ {i \ge  0}$ of normal subgroups of $G$ with finite index and trivial intersection.

One can also look at a uniformly discrete sequence of lattices $(G_i)_{i \ge 0}$ in a
connected center-free semisimple Lie group $L$ without compact factors and study the
quotients $M[i] = X/G_i$, where $X$ is the associated symmetric space $L/K$ for $K
\subseteq L$ a maximal compact subgroup.  There is a notion of BS-convergence for lattices
which generalizes our setting. One can ask whether for such a convergence sequence of
cocompact lattices the sequence $\frac{b_n(M[i];\IQ)}{\vol(M[i])}$ converges to the
$L^2$-Betti number of $X$. This setup and various convergence questions are systematically
examined in the papers by
Abert-Bergeron-Biringer-Gelander-Nikolov-Raimbault-Samet~\cite{Abert-Bergeron-Biringer-Gelander-Nokolov_Raimbault-Samet(2011),
  Abert-Bergeron-Biringer-Gelander-Nokolov_Raimbault-Samet(2012)}. 

Another paper containing interesting information
about these questions is Bergeron-Lipnowski~\cite{Bergeron-Lipnowski(2014)}.

\subsection{Twisting with representations}
\label{subsec:Twisting_with_representations}

We have already mentioned that one can twist the analytic torsion with special representations.
This has in favorite  situations the effect that one obtains a uniform gap for the spectrum of
the Laplace operators and can prove the desired approximations results, see
Remark~\ref{rem:The_Zero-in-the-Spectrum_Conjecture}.  For more information we refer for
instance
to~\cite{Bergeron-Venkatesh(2013),Marshall-Mueller(2013),Mueller-Pfaff(2013locsym),Mueller-Pfaff(2013asymp)}.

In~\cite{Lueck(2015twisting)} twisted $L^2$-torsion for finite $CW$-complex $X$ with
$b_n^{(2)}(\widetilde{X}) = 0$ for all $n \ge 0$ is introduced for finite-dimensional
representations which are given by restricting finite-dimensional
$\IZ^d$-representations with any homomorphism
$\pi_1(M) \to \IZ^d$. In particular one can twist the $L^2$-torsion for a given element
$\phi \in H^1(X;\IZ)$ with the $1$-dimensional representation whose underlying complex
vector space is $\IC$ and on which $g \in \pi_1(X)$ acts by multiplication with
$t^{\phi(g)}$. This yields the $L^2$-torsion function $(0,\infty) \to \IR$ whose value at
$1$ is the $L^2$-torsion itself.  The proof that this function is well-defined is based on
approximation techniques. This function seem to contain very interesting information, in
particular  for $3$-manifolds, see for
instance~\cite{Dubois-Friedl-Lueck(2014Alexander),Dubois-Friedl-Lueck(2015symmetric),
Dubois-Friedl-Lueck(2015flavors),Dubois-Friedl-Lueck(2015CRMASP),Lueck(2015twisting)}.
In particular one can
read off the Thurston norm of $\phi$ from the asymptotic behavior at $0$ and $\infty$ if
$X$ is a connected compact orientable $3$-manifold with infinite fundamental group and
empty or toroidal boundary which is not $S^1 \times D^2$, see \cite{Friedl-Lueck(2015l2+Thurston),Liu(2015)}.

\subsection{Atiyah's Question}
\label{subsec:Atiyahs_question}

Atiyah~\cite[page 72]{Atiyah(1976)} asked the question, whether the $L^2$-Betti numbers
$b_p^{(2)}(\widetilde{M})$ for a closed Riemannian manifold $M$ are always rational
numbers. Meanwhile it is known that the answer can be negative, see for
instance~\cite{Austin(2009),Grabowski(2010),Pichot-Schick-Zuk(2015)}.  However, the
following problem, often referred to as the strong Atiyah Conjecture, remains open.

\begin{question}\label{que:Atiyahs_strong_conjecture}
Let $G$ be a group for which there exists natural number $d$ such that the order of any finite subgroup
divides $d$. Then:

\begin{enumerate}

\item \label{que:Atiyahs_strong_conjecture:matrices}
For any $A \in M_{m,n}(\IZ G)$ we get for the von Neumann dimension of the kernel of the
induced $G$-equivariant bounded operator $r_A^{(2)} \colon L^2(G)^m \to L^2(G)^n$
\[
d \cdot \dim_{\caln(G)}\bigl(\ker(r_A^{(2)})\bigr) \in \IZ;
\]

\item \label{que:Atiyahs_strong_conjecture:manifolds}
For every closed manifold $M$ with $G \cong \pi_1(M)$ and $n \ge 0$ we have
\[d \cdot b_n^{(2)}(\widetilde{M}) \in \IZ.
\]
\end{enumerate}
\end{question}
Notice that we can choose $d = 1$ if $G$ is torsion-free. For a discussion, a survey on the literature  and the status of this
Question~\ref{que:Atiyahs_strong_conjecture}, we refer for instance to~\cite[Chapter~10]{Lueck(2002)}.

The Approximation Conjecture~\ref{con:Approximation_conjecture_for_L2-Betti_numbers},
which is known by Remark~\ref{rem:status_of_Determinant_Conjecture} and 
Theorem~\ref{the:The_Determinant_Conjecture_implies_the_Approximation_Conjecture_for_L2-Betti_numbers}
for a large class of groups, can be used to enlarge the class of groups for which
the answer to part~\eqref{que:Atiyahs_strong_conjecture:matrices} of Question~\ref{que:Atiyahs_strong_conjecture}
is positive. Namely, if $G$ is torsion-free and possesses a chain of normal subgroups $(G_i)_{i \ge 0}$ 
 with  trivial intersection $\bigcap_{i \ge 0} G_i = \{1\}$ such that
the answer to part~\eqref{que:Atiyahs_strong_conjecture:matrices} of Question~\ref{que:Atiyahs_strong_conjecture}
is positive for each quotient $G/G_i$, then
the answer to part~\eqref{que:Atiyahs_strong_conjecture:matrices} of Question~\ref{que:Atiyahs_strong_conjecture}
is positive for each quotient $G/G_i$. Here it becomes important that we could drop the condition that
each $G/G_i$ is finite. An example for $G$ is a finitely generated free group 
whose descending central series gives such a chain  $(G_i)_{i \ge 0}$  with torsion-free nilpotent quotients $G/G_i$.

Notice that Conjecture~\ref{con:Homological_growth_and_L2-torsion_for_aspherical_manifolds}
implies a positive answer to  part~\eqref{que:Atiyahs_strong_conjecture:manifolds}
of Question~\ref{que:Atiyahs_strong_conjecture} if $M$ is an aspherical
closed manifold.

One can  ask an analogous question in the mod $p$ case as soon as one has
a replacement for the $L^2$-Betti number in the mod $p$ case. In some special cases this replacement exists and the answer is positive,
 see for instance Theorem~\ref{the:dim_approximation_over_fields} 
for torsion-free elementary amenable groups, 
and Theorem~\ref{the:BLLS} for torsion-free $G$ taking into account that the 
$n$th mod $p$ $L^2$-Betti numbers $b_n^{(2)}(\overline{X};F)$
occurring  in~\cite[Definition~1.3]{Bergeron-Linnell-Lueck-Sauer(2014)} is an integer for torsion-free $G$.

\subsection{Simplicial volume}
\label{subsec:Simplicial_volume}

The following conjecture is discussed in~\cite[Chapter~14]{Lueck(2002)}.

\begin{conjecture}[Simplicial volume and $L^2$-invariants]
  \label{con:simplicial_volume_and_L2-invariants}
  Let $M$ be an aspherical closed orientable manifold of dimension $\ge 1$.  Suppose that
  its simplicial volume $||M||$ vanishes. Then
  \begin{eqnarray*}
    b_p^{(2)}(\widetilde{M}) & = & 0 \hspace{5mm} \mbox{ for } p \ge 0;
    \\
    \rho^{(2)}(\widetilde{M}) & = & 0.
  \end{eqnarray*}
\end{conjecture}

If the closed orientable manifold $M$ has a self-map $f \colon M \to M$ of degree different
from $-1$, $0$, $1$, then one easily checks that its simplicial
volume $||M||$ vanishes. If its minimal volume is zero, i.e., for every $\epsilon > 0$ one
can find a Riemannian metric on $M$ whose sectional curvature is pinched between $-1$ and
$1$ and for which the volume of $M$ is less or equal to $\epsilon$, then its simplicial
volume $||M||$ vanishes. This follows from~\cite[page~37]{Gromov(1982)}. 

If one replaces in Conjecture~\ref{con:simplicial_volume_and_L2-invariants}
the simplicial volume by the minimal volume, whose vanishing implies the vanishing of the simplicial volume,
then the claim for the $L^2$-Betti numbers in  Conjecture~\ref{con:simplicial_volume_and_L2-invariants} has been proved by 
Sauer~\cite[Second Corollary of Theorem~A]{Sauer(2009amen)}.

There are versions of the simplicial volume such as the integral foliated simplicial volume and stable
integral simplicial volume which are related to Conjecture~\ref{con:simplicial_volume_and_L2-invariants}
and may be helpful for a possible proof,  and reflect a kind of approximation conjecture for the simplicial volume,
see for instance~\cite{Francaviglia-Frigerio-Martelli(2012), Loeh-Pagliantini(2014),Schmidt(2005)}.

More information about the simplicial volume and the literature can be found for instance
in~\cite{Gromov(1982)},~\cite{Loeh(2011Atlas)},~\cite[Section~14.1]{Lueck(2002)}.

\subsection{Entropy, Fuglede-Kadison determinants and amenable exhaustions}
\label{subsec:Entropy_and_L2-torsion}

In  recent years the connection between entropy and Fuglede-Kadison determinant has
been investigated in detail, see for
instance~\cite{Deninger(2006Fuglentramen),Deninger(2009entropy),Li(2012Annals),Li-Thom(2014)}.
In particular the amenable exhaustion approximation result for Fuglede-Kadison
determinants of Li-Thom~\cite[Theorem~0.7]{Li-Thom(2014)} for amenable groups $G$  is very interesting,
where the Fuglede-Kadison determinant of a matrix over $\IZ G$ is approximated by 
finite-dimensional analogues of its ``restrictions'' to  finite F\"olner subsets of the group $G$.

\subsection{Lehmer's problem}
\label{subsec:Lehmers_problem}

Let $p(z) \in \IC[\IZ] = \IC[z,z^{-1}]$ be a non-trivial element.  Its \emph{Mahler
  measure} is defined by
\begin{eqnarray}
  M(p) &:= &  \exp\left(\int_{S^1} \ln(|p(z)|) d\mu\right).
  \label{Mahler_measure}
\end{eqnarray}
By Jensen's inequality we have
\begin{eqnarray}
  \int_{S^1} \ln(|p(z)|) d\mu
  & = &
  \sum_{\substack{i=1,2, \ldots, r\\|a_i| > 1}} \ln(|a_i|),
  \label{Jensen}
\end{eqnarray}
if we write $p(z)$ as a product
\[
p(z) = c \cdot z^k \cdot \prod_{i=1}^r (z - a_i)
\]
for an integer $r \ge 0$, non-zero complex numbers $c$, $a_1$, $\ldots$, $a_r$ and an
integer $k$.  This implies $M(p) \ge 1$.

\begin{problem}[Lehmer's Problem]\label{pro:Lehmers_problem}
  Does there exist a constant $\Lambda > 1$ such that for all non-trivial elements $p(z)
  \in \IZ[\IZ] = \IZ[z,z^{-1}]$ with $M(p) \not= 1$ we have
  \[
  M(p) \ge \Lambda?
  \]
\end{problem}

\begin{remark}[Lehmer's polynomial]
  \label{rem:Lehmers_polynomial}
  There is even a candidate for which the minimal Mahler measure is attained, namely,
  \emph{Lehmer's polynomial}
  \[
  L(z) := z^{10} + z^9 - z^7 - z^6 -z^5 -z^4- z^3 + z +1.
  \]
  It is conceivable that for any non-trivial element $p \in \IZ[\IZ]$ with $M(p) \not= 1$
  \[
  M(p) \ge M(L) = 1.17628 \ldots \] holds.
\end{remark}

For a survey on Lehmer's problem were refer for instance 
to~\cite{Boyd-Lind-Villegas-Deninger(1999),Boyd(1981speculations),Carrizosa(2009),Smyth(2008)}.

Consider an element $p = p(z) \in \IC[\IZ] = \IC[z,z^{-1}]$. It defines a
bounded $\IZ$-operator $r^{(2)}_p \colon L^2(\IZ) \to  L^2(\IZ)$ by multiplication
  with $p$. Suppose that $p$ is not zero. Then the Fuglede-Kadison determinant of $r_p^{(2)}$
agrees  with  the Mahler measure of $p$ by~\cite[(3.23) on page~136]{Lueck(2002)}.

\begin{definition}[Lehmer's constant of a group]
  \label{def:Lehmers_constant_of_a_group}
  Define \emph{Lehmer's constant} $\Lambda(G)$ of a group $G$
  \[
  \Lambda(G) \in [1,\infty)
  \]
  to be the infimum of the set of Fuglede-Kadison determinants
  \[ {\det}^{(2)}_{\caln(G)}\bigl(r_A^{(2)} \colon \caln(G)^r \to \caln(G)^s\bigr),
  \]
  where $A$ runs through all $(r,s)$-matrices $A \in M_{r,s}(\IZ G)$ for all $r,s \in \IZ$
  with $r,s \ge 1$ for which ${\det}^{(2)}_{\caln(G)}(r_A^{(2)}) > 1$ holds.

  If we only allow square matrices $A$ such that $r_A^{(2)} \colon \caln(G)^r \to \caln(G)^r$ is
  injective and ${\det}^{(2)}_{\caln(G)}(r_A^{(2)}) > 1$, then we denote the corresponding
  infimum by
  \[
  \Lambda^w(G) \in [1,\infty)
  \]
\end{definition}

Obviously we have $\Lambda(G) \le \Lambda^w(G)$. We suggest the following 
generalization of Lehmer's problem to arbitrary groups.

\begin{problem}[Lehmer's problem for arbitrary groups]
  \label{pro:Lehmers_problem_for_arbitrary_groups}
  For which groups $G$ is $\Lambda(G) > 1$ or  $\Lambda^w(G) > 1$?
\end{problem}

For a discussion and results on these problems see~\cite[Question~4.7]{Chung-Thom(2015)} and~\cite{Lueck(2015_lehmer_general)}.




\end{document}